\documentclass[UTF-8,reqno]{amsart}
\usepackage{enumerate}
\usepackage{slashed}
\usepackage{appendix}
\usepackage{mhequ, tikz, wasysym, stmaryrd}
\usepackage{amssymb,url,color, booktabs,nccmath}
\usepackage{amsmath}
\usepackage{fancyhdr}
\usepackage[left=3.2cm,right=3.2cm,top=3cm,bottom=4cm]{geometry}
\usepackage{mathrsfs}
\usepackage{enumitem,dsfont}
\usepackage{color}
\usepackage[colorlinks=true]{hyperref}
\hypersetup{
	linkcolor=blue,          
	citecolor=red,        
	filecolor=blue,      
	urlcolor=cyan
}

\definecolor{darkergreen}{rgb}{0.0, 0.5, 0.0}

\numberwithin{equation}{section}
\newtheorem{theorem}{Theorem}[section]
\newtheorem{lemma}[theorem]{Lemma}
\newtheorem{proposition}[theorem]{Proposition}

\newtheorem{assum}{Assumption}[section]
\newtheorem{definition}[theorem]{Definition}
\theoremstyle{definition}
\newtheorem{remark}[theorem]{Remark}

\begin{document}

\title{Weak coupling limit of a Brownian particle in the curl of the 2D GFF}
\author{{\bf Huanyu Yang}$^{\mbox{a,b}}$, {\bf Zhilin Yang}$^{\mbox{c}*}$\\ \\
{
\small $^{\mbox{a}}$ College of Mathematics and Statistics, Chongqing University, Chongqing 401331, China,\\
\small $^{\mbox{b}}$ Key Laboratory of Nonlinear Analysis and its Applications (Chongqing University), Ministry of Education,\\
\small $^{\mbox{c}}$ Tongji University, China.\\}
}
\date{}
\thanks
{E-mail\;address:yanghuanyu@cqu.edu.cn(H.Y.Yang),yangzhilin0112@163.com(Z.L.Yang). *Corresponding Author.}

\begin{abstract}
In this article, we study the weak coupling limit of the following equation in $\mathbb{R}^2$:
$$dX_t^\varepsilon=\frac{\hat{\lambda}}{\sqrt{\log\frac1\varepsilon}}\omega^\varepsilon(X_t^\varepsilon)dt+\nu dB_t,\quad X_0^\varepsilon=0. $$ Here $\omega^\varepsilon=\nabla^{\perp}\rho_\varepsilon*\xi$ with $\xi$ representing the $2d$ Gaussian Free Field (GFF) and $\rho_\varepsilon$ denoting an appropriate identity. $B_t$ denotes a two-dimensional standard Brownian motion, and $\hat{\lambda},\nu>0$ are two given constants. 
We use the approach from \cite{Cannizzaro.2023} to show that the second moment of $X_t^\varepsilon$ under the annealed law converges to $(c(\nu)^2+2\nu^2)t$ with a precisely determined constant $c(\nu)>0$, which implies a non-trivial limit of the drift terms as $\varepsilon$ vanishes. We also prove that in this weak coupling regime, the sequence of solutions converges in distribution to $\left(\sqrt{\frac{c(\nu)^2}{2}+\nu^2}\right)\widetilde{B}_t$ as $\varepsilon$ vanishes, where $\widetilde{B}_t$ is a two-dimensional standard Brownian motion.
\end{abstract}
\maketitle
\tableofcontents

\section{Introduction}
\subsection{The model and motivations}
In this article, we study the motion of a Brownian particle in $\mathbb{R}^2$, involving in a random environment, given by the solution to the SDE formally,
	\begin{align}\left\{\begin{aligned}\label{SDE1}
            &dX_t=b(X_t)dt+\nu dB_t, \quad t\geq 0,\\
            &X_0=0,
	\end{aligned}\right.\end{align}
	where $(B_t)_{t\geq 0}$ is a two-dimensional standard Brownian motion, $\nu$ is a given positive parameter and the drift term
	\begin{equation}
		b(x)=\lambda\nabla^\perp\xi(x):=\lambda(\partial_{x_2}\xi(x),-\partial_{x_1}\xi(x)),
	\end{equation}
	with $\xi$ the two-dimensional Gaussian Free Field ($2d$ GFF) and $\lambda$ the positive coupling constant. 

Since the $2d$ GFF $\xi$ is a distribution in $\mathcal{C}^{\kappa}$, the space of H\"older distributions with regularity $\kappa$, for any $\kappa<0$, the equation \eqref{SDE1} cannot be understood in the classical sense. In fact, for a general distributional drift term $b\in\mathcal{C}^\alpha$ in \eqref{SDE1} for some $\alpha<0$, it is more reasonable to consider the weak solution to such singular SDEs. 
For $\alpha\in[-\frac12,0)$, 
the authors in \cite{Zhang.2017} demonstrated the existence and uniqueness of the weak solution by analyzing the associated Kolmogorov backward equation
$$\partial_tu=\frac{\nu^2}{2}\Delta u+ b\cdot\nabla u.$$
But for the case $\alpha<-\frac12$, by {the} Schauder's estimate, $u$ is expected to belong to $\mathcal{C}^{\alpha+2}$, hence the product $b\cdot\nabla u$ is not well-defined.  
For $\alpha\in (-\frac23,-\frac12)$, if $b$ can be enhanced to a rough distribution,
the authors in \cite{Cannizzaro.2015} and \cite{Delarue.2016} independently proved the existence of a unique martingale solution by applying rough path and the theory of paracontrolled distribution. Recently, with $\alpha\in (-1,0]$ and a divergence-free $b$, \cite{Grafner.2023} constructed a unique
energy solution on the torus with initial probability distribution possessing an $L^2$-density with respect to the Lebesgue measure by constructing the infinitesimal generators of SDE \eqref{SDE1} via the Wiener-Chaos decomposition.
Hao and Zhang in \cite{Hao.2023} further improved the existence results in  \cite{Grafner.2023} for any initial probability measure and on the whole space by using the Zvonkin’s transformation. 
Additionally, they
established the existence of the weak solution in the case that $\alpha=-1$ and $b$ is divergence-free.
We also mention that the singular SDEs driven by the L\'evy process have been considered in \cite{Kremp.2022}.

In our case, the stationary divergence-free Gaussian vector $b=\lambda\nabla^\perp\xi\in\mathcal{C}^{-1+\kappa},$ for any $\kappa<0$, which falls into the supercritical regime in the scope of \cite{Hao.2023}. Moreover, it is below the threshold $-1$ given in the literatures mentioned above. Thus we consider a regularized version of \eqref{SDE1}. 

For a better comprehension of SDE (\ref{SDE1}) for general $d\geq 2$, we now consider a regularized version of \eqref{SDE1} in $\mathbb{R}^d$ by convolving $\nabla^\perp\xi$ with $\rho_\varepsilon:=\frac{1}{\varepsilon^d}\rho_1(\frac1\varepsilon\cdot)$, where $\rho_1$ represents a smooth bump function. It is given by 
\begin{align}\left\{\begin{aligned}\label{SDE2}
&dY_t^\varepsilon=\lambda\rho_{\varepsilon}*\nabla^\perp\xi(Y_t^\varepsilon)dt+\nu dB_t, \quad t\geq 0,\\
&Y_0^\varepsilon=0.
\end{aligned}\right.\end{align}
Since the drift term in (\ref{SDE2}) is smooth, we could obtain the global existence and uniqueness of the strong solutions $Y_t^\varepsilon$ to \eqref{SDE2} in $\mathbb{R}^d$, where $B_t$ is a $d$-dimensional standard Brownian motion, and $\nabla^\perp\xi$ is the Leray projection of white noise with the Fourier transform of the covariance function {$R(x-y)=\mathbb{E}(\nabla^\perp\xi)(x)\otimes (\nabla^\perp\xi)(y)$} given by
\begin{align}\label{def:cor-d}
\hat{R}(p)=(2\pi)^2(I_{d\times d}-|p|^{-2}p\otimes p).
\end{align}
For $d=2$, in the context of \cite{Hao.2023}, there is no hope to see the convergence of $Y^\varepsilon$ as $\varepsilon\to 0$. Instead, they made a conjecture in \cite[Remark 1.6]{Hao.2023} that the convergence results for $Y^\varepsilon$ holds if we consider the coupling constant $\lambda=O(1/\sqrt{\log \varepsilon})$. Such $1/\sqrt{\log \varepsilon}$ coupling has been observed in variant of models (see e.g. \cite{Caravenna.2017,Caravenna.2020,Caravenna.2023,Cannizzaro.2023,Cannizzaro.2024,Cannizzaro.2024cmp,Cannizzaro.2024cg}). 

To see why $1/\sqrt{\log \varepsilon}$ coupling works in $\mathbb{R}^2$, we can look at the long-time fluctuation of (\ref{SDE2}) for $\varepsilon=1$ by rescaling the solution $Y^1$ to \eqref{SDE2} in $\mathbb{R}^d$ diffusively as
\begin{align}\label{scale}
X^\varepsilon_t:=\varepsilon Y_{\frac{t}{\varepsilon^2}}^1
\end{align}
for all $\varepsilon>0$. It is worth noting that the Brownian motion $B_t$ remains invariant in law under the scaling $\varepsilon B_{\frac{t}{\varepsilon^2}}$. Hence we retain the notation $B_t$ for simplicity. Moreover, since $\nabla^\perp\xi$ is invariant in law under the scaling $\varepsilon^{\frac d2}\nabla^\perp\xi(\varepsilon\cdot)$ by \eqref{def:cor-d}, $X^\varepsilon$ satisfies
\begin{equation}\label{eq:scale}
dX_t^\varepsilon=\lambda\varepsilon^{\frac d2-1}\rho_\varepsilon*\nabla^\perp\xi(X_t^\varepsilon)dt+\nu dB_t, \quad X_0^\varepsilon=0.
\end{equation}

 It can be observed that $d=2$ is the critical dimension since the drift term in \eqref{eq:scale} is scaled invariant under the diffusive scaling. Thus $Y^\varepsilon$ has the same law as $X^\varepsilon$. In this case, $Y^1$ admits the $\sqrt{\log t}$-superdiffusivity i.e.
\begin{align}
\frac{\textbf{E}[|Y^1(t)|^2]}{t}\approx \sqrt{\log t},\quad t\to\infty.\label{conjecture}
\end{align}
This property was conjectured by T\'oth and Valk\'o in \cite{Toth.2012} and recently proved in
 \cite{Cannizzaro.2022} and \cite{Chatzigeorgiou.2022} separately. As a consequence,  the solution $X_t^\varepsilon$ to \eqref{eq:scale} does not converge as $\varepsilon\to 0$. More precisely, 
by \eqref{conjecture}, we observe that $\textbf{E}[|X_t^\varepsilon|^2]=\varepsilon^2\textbf{E}[|Y^1_{\frac{t}{\varepsilon^2}}|^2]\approx t\sqrt{\log\frac{t}{\varepsilon^2}},\varepsilon\to 0.$ This suggests that there is no hope for the rescaled equation (\ref{eq:scale}) to have meaningful limit for fixed $t$ as $\varepsilon \to 0$. Furthermore, since {for every $t\geq 0$, $\textbf{E}[|X_t^\varepsilon|^2]\sim \sqrt{\log\frac{1}{\varepsilon}}, \varepsilon\to 0$}
 , it is natural to ask whether the weak coupling constant $\lambda=\frac{\hat{\lambda}}{\sqrt{\log\frac{1}{\varepsilon}}}$ could remove the log-superdiffusivity and make the law of $\{X^\varepsilon\}_\varepsilon$ tight. 

In the case where $d\geq3$, an invariance principle was obtained by \cite{Komorowski.2003} and \cite{Komorowski.2012}.
More precisely, they showed that as $\varepsilon\to0$, the solution $X^\varepsilon$ to \eqref{eq:scale} converges under the annealed law to $\sigma B$, where $B$ denotes a d-dimensional standard Brownian motion and $\sigma\geq\nu$ is only implicitly provided. In spite of the lack of explicit verification in the above references, we believe that $\sigma>\nu$ can be proved by applying the method in \cite{Toth.2012}, which implies a non-trivial limit for the drift terms.

The main result of this paper is that, in the weak coupling regime $\lambda=\frac{\hat{\lambda}}{\sqrt{\log\frac1\varepsilon}}$, the solution $X^\varepsilon$ to \eqref{eq:scale} converges under the annealed law to $\sigma \widetilde{B}$ with the explicitly diffusive coefficient $\sigma>\nu$, where $\widetilde{B}$ is a standard $2d$ Brownian motion. This implies that in the weak coupling regime, the drift term converges to a Brownian motion non-trivially.
Very recently, Armstrong, Bou-Rabee and Kuusi in \cite{Armstrong.2024} obtained a quenched invariance principle for $|\log\varepsilon^2|^{-\frac14}\varepsilon Y^1_{t/\varepsilon^2}$, which has the same limit as in the strong coupling regime $\varepsilon Y^1_{t/(\varepsilon^2\sqrt{\log\frac1\varepsilon})}$. Additionally, they derived the explicit diffusive coefficient which is independent of the parameter $\nu$.
They estimated the rate of superdiffusivity based on a quantitative, scale-by-scale approach in stochastic homogenization.

\subsection{Main results and ideas of proofs}\label{sec1.2} 
Let us first clarify the assumption on the mollifier $\rho$, which will be used throughout the paper.	
\begin{assum}\label{rho} 
    Let $\rho\in \mathcal{S}(\mathbb{R}^2)$ be radially symmetric, i.e. $\rho(x)=\gamma(|x|),\forall x\in\mathbb{R}^2$, for some $\gamma\in\mathcal{S}(\mathbb{R})$, and $\int_{\mathbb{R}^2}\rho(x)dx=1.$ Additionally, we assume that the Fourier transform of $\rho$ is supported in the ball of radius 1.
\end{assum} 
Let $\Omega$ be the space of smooth vector fields: 
\begin{align*}
\Omega=\{\sigma\in C^{\infty}(\mathbb{R}^2\to\mathbb{R}^2): \nabla\cdot\sigma=0,\; \sup_{x\in\mathbb{R}^2}(1+|x|)^{-\frac1r}|\partial^{m}\sigma_i(x)|<\infty,\;\forall i=1,2,\;r\in\mathbb{N},\;m\in\mathbb{N}^2\},
\end{align*}
with the $\sigma$-algebra $\mathcal{F}$ generated by cylindrical functionals on $\Omega$. Let $\xi$ be the $2d$ GFF. We assume that $\omega:=\nabla^\perp\rho*\xi$, the curl of the $2d$ GFF convolved with $\rho$, has law $\mathbb{P}$ on $\Omega$ (and corresponding expectation $\mathbb{E}$). Then it is a centered Gaussian field defined on the probability space $(\Omega,\mathcal{F},\mathbb{P})$ with covariance given by
\begin{align*}
    \mathbb{E}[\omega_i(x)\omega_j(y)]=\partial_{x_i}^\perp\partial_{x_j}^\perp V*g(x-y),\quad \forall x,y\in \mathbb{R}^2, \quad i,j=1,2,
\end{align*} 
 where $g(x)=-(2\pi)\log|x|$ and 
 $V(x)=\rho*\rho(x)=\int_{\mathbb{R}^2}\rho(x-y)\rho(y)dy.$ 
 
We focus on the rescaled equation of \eqref{eq:scale} in the weak coupling regime in $\mathbb{R}^2$. For all $\varepsilon>0$, let $\xi^\varepsilon$ denote the $2d$ GFF regularized by $\rho_\varepsilon:=\varepsilon^{-2}\rho(\varepsilon^{-1}\cdot)$ on the probability space $(\Omega,\mathcal{F},\mathbb{P})$. Consequently, its covariance function is given by 
\begin{align*}
    \mathbb{E}[\xi^\varepsilon(x)\xi^\varepsilon(y)]=V_\varepsilon*g(x-y),\quad \forall x,y\in \mathbb{R}^2,
\end{align*} 
where $V_\varepsilon=\rho_\varepsilon*\rho_\varepsilon$. Define $\omega^\varepsilon:=\nabla^\perp\xi^\varepsilon$, then the scaled solutions are represented by
\begin{equation}
    dX_t^\varepsilon=\frac{\hat{\lambda}}{\sqrt{\log\frac{1}{\varepsilon}}}\omega^{\varepsilon}(X_t^\varepsilon)dt+\nu dB_t, \quad X_0^\varepsilon=0, \label{scaleSDE}
\end{equation}
where $\hat{\lambda},\nu>0$ are two constants, and $(B_t)_{t\geq 0}$ is a two-dimensional standard Brownian motion defined on another probability space $(\Sigma,\mathcal{A},Q)$.

Let $\textbf{P}=\mathbb{P}\otimes\mathbb{Q}$ be the annealed measure on the product space $(\Omega\times\Sigma,\mathcal{F}\otimes\mathcal{A}),$ corresponding to the expectation $\textbf{E}$. Then for every $\varepsilon>0$, $X_t^\varepsilon$ and $N_t^\varepsilon:=\frac{\hat{\lambda}}{\sqrt{\log\frac{1}{\varepsilon}}}\int_0^t\omega^{\varepsilon}(X_s^\varepsilon)ds$ are both defined on this probability space $(\Omega\times\Sigma,\mathcal{F}\otimes\mathcal{A},\textbf{P}),$ given by \eqref{scaleSDE}. Since the state space of $\omega^\varepsilon$ is $\Omega$, for every $\sigma\in \Omega$, by \cite[Theorem 8.3]{Le-Gall.2016}, $\{X_t^{\varepsilon,\sigma}\}_{t\geq 0}$ is defined as the unique strong solution to the SDE: 
\begin{align}\label{def:Y}
dY_t=\frac{\hat{\lambda}}{\sqrt{\log\frac{1}{\varepsilon}}}\sigma(Y_t)dt+\nu dB_t
\end{align}
on $\Sigma$. We are now ready to state the results of this paper.
The first one is about the asymptotical diffusivity of the drift $N_t^\varepsilon$ in the weak coupling regime.

\begin{theorem}
For every $t\geq0$, 
    $\lim_{\varepsilon\to 0}\mathbf{E}|N_t^{\varepsilon}|^2=c(\nu)^2t$, where the effective diffusion coefficient $c(\nu)=\sqrt{4\left(\sqrt{2\pi\hat{\lambda}^2+\frac{\nu^4}{4}}-\frac{\nu^2}{2}\right)}.$ \label{thm2.1}
\end{theorem}
\begin{remark}
Note that the limit diffusive constant $c(\nu)$ depends on $\nu$, and it converges to a non-zero constant when $\nu\to0$. It is different from
the case in \cite{Armstrong.2024} where the diffusion coefficient is independent of $\nu$ since the Brownian motion term vanishes as $\varepsilon\to0$ in the strong coupling regime. 
\end{remark}

Theorem \ref{thm2.1} implies that as $\varepsilon\to 0$, the limit of $N^\varepsilon$ has a similar diffusive behaviour as Brownian Motion. 
Indeed, the following theorem demonstrates an invariance principle for $\{X_{\cdot}^\varepsilon\}_\varepsilon$ in the weak coupling regime. 
{\begin{theorem}
    For every $\varepsilon>0$, let $X^{\varepsilon}$ be the solution to \eqref{scaleSDE}. Then as $\varepsilon\to 0$, it converges under the annealed law $\textbf{P}$ to a Brownian Motion $X_\cdot=\left(\sqrt{\frac{c(\nu)^2}{2}+\nu^2}\right)\widetilde{B_\cdot}$, where $\widetilde{B_\cdot}$ is a two-dimensional standard Brownian motion on $\Omega\times\Sigma$ and $c(\nu)$ is the diffusion coefficient given in Theorem \ref{thm2.1}.\label{thm2.2}
\end{theorem}
}
\begin{remark}
Recall that $\nabla^\perp\xi\in\mathcal{C}^{-1+\kappa}$, for any $\kappa<0$. This weak coupling limit implies that the result in \cite{Hao.2023} is sharp, i.e. there's no hope to generalize the classical approach to treat singular SDE \eqref{SDE1} when $b\in\mathcal{C}^\alpha,\;\alpha<-1$.
\end{remark}

{Our argument for the proof of Theorem \ref{thm2.2} is based on the classical argument in the study of random walk in random environment (RWRE). For every $\varepsilon>0$, we consider the environment process seen from the particle, denoted as $\eta^\varepsilon_t:=\tau_{X_t^\varepsilon}\omega^\varepsilon$, where $\{\tau_x\}_{x\in\mathbb{R}^2}$ is the family of translation operators on the state space $\Omega$ of $\eta_t^\varepsilon$: for $x\in\mathbb{R}^2$, $\sigma\in\Omega,$ $\tau_x \sigma:=\sigma(x+\cdot).$
The advantage of working with $\{\eta_t^\varepsilon\}_{t\geq0}$ is that it is a Markov process with generator $\mathcal{L}^\varepsilon$ and an ergodic invariant measure $\mathbb{P}^\varepsilon$.
Then $X_t^\varepsilon$ can be decomposed into a martingale and an additive functional of the environment process,
$$
X^\varepsilon_t=\int_0^t\mathcal{V}^\varepsilon(\eta^\varepsilon_s)ds+\nu\varepsilon B_{\frac{t}{\varepsilon^2}}.
$$
{Since $\varepsilon B_{\frac{t}{\varepsilon^2}}\overset{d}{=}B_t$, we only neet to treat the additive functional part. By L\'evy's characterization, we need to prove the limit of such additive functional is a martingale. This is in the setting of the Kipnis-Varadhan's theory \cite{Kipnis.1986} (see also \cite{Komorowski.2012} for a comprehensive study). Heuristically, the idea is to solve the Poisson equation w.r.t the generator $\mathcal{L}^\varepsilon$. Once we can find $\Phi^\varepsilon$ such that $-\mathcal{L}^\varepsilon\Phi^\varepsilon=\mathcal{V}^\varepsilon$, by Dynkin's formula, {there exists a two-dimensional martingale $M(\Phi^\varepsilon)$ w.r.t $\mathcal{F}_t=\sigma(\eta_s^\varepsilon,s\leq t)$ such that
$$\Phi^\varepsilon(\eta_t^\varepsilon)-\Phi^\varepsilon(\eta_0^\varepsilon)-\int_0^t\mathcal{L}^\varepsilon\Phi^\varepsilon(\eta_s^\varepsilon)ds=M_t(\Phi^\varepsilon).$$
Then the additive functional can be expressed as 
$$
\int_0^t\mathcal{V}^\varepsilon(\eta^\varepsilon_s)ds=M_t(\Phi^\varepsilon)+o(1).
$$
Here $o(1)$ contains the boundary terms $\Phi^\varepsilon(\eta_s^\varepsilon), s=0$ and $t$, which vanish as $\varepsilon\to 0.$ Therefore it suffices to prove that $\{M(\Phi^\varepsilon)\}_\varepsilon$ converges to a continuous martingale with the correct quadratic variation, which can be done via the classical approach such as \cite[Theorem 2.29]{Komorowski.2012}.

This strategy is hard to be applied to our case since the generator
$\mathcal{L}^\varepsilon$ is not invertible thus the generator equation cannot be solved directly. Moreover, the approximation in \cite{Komorowski.2012} does not work for our case since the superdiffusivity in \cite{Cannizzaro.2022}\cite{Chatzigeorgiou.2022} implies that for every $\varepsilon>0$,  $\lim_{\lambda\to0}\mathbb{E}^\varepsilon(\mathcal{V}^\varepsilon(\lambda-\mathcal{L}^\varepsilon)^{-1}\mathcal{V}^\varepsilon)=\infty.$  Instead, we apply the iterative truncation method to approximate the resolvent $(\lambda-\mathcal{L}^\varepsilon)^{-1}$ and use the Replacement Lemma to get the explicit diffusion coefficient. 
This iterative truncation method for the generator first appeared in \cite{Yau.2004} and \cite{Landim.2002}. Later in \cite{Cannizzaro.2023cpam}, Cannizzaro, Erhard and Toninelli applied this method to the energy solution framework developed in a sequence of papers \cite{Goncalves.2013, Goncalves.2014,Goncalves.2015, Gubinelli.2017,Gubinelli.2020} for the $2d$ anisotropic KPZ equation (AKPZ) to obtain the upper and lower bounds for $(\lambda-\mathcal{L}^\varepsilon)^{-1}$ for fixed $\lambda,\varepsilon$, thereby proving the logarithmic superdiffusivity of the $2d$ AKPZ. This approach was later applied to our model in \cite{Cannizzaro.2022}, and very recently to $2d$ stochastic Burgers equation in \cite{Gaspari.2024}. Additionally, in order to determine the large-scale (Gaussian) fluctuations at criticality, it is required to derive a family of approximate operators. The authors in \cite{Cannizzaro.2023} first use the Replacement Lemma to get such an approximate sequence for $(\mathcal{L}^\varepsilon)^{-1}$, and hence they get the weak coupling limit for $2d$ AKPZ. Very recently, a simplified version of this lemma was applied in \cite{Cannizzaro.2024cmp} for the weak coupling limit of $2d$ Stochastic Burgers equation, and in \cite{Cannizzaro.2024cg} for the model of $2d$ Self-Repelling Brownian Polymer.
Note that the symmetric part of $\mathcal{L}^\varepsilon$ is not negative definite in our model, which is different from the AKPZ in \cite{Cannizzaro.2023}. Thus we need to consider the approximation of $(\lambda-\mathcal{L}^\varepsilon)^{-1}$ for fixed $\lambda>0$, instead of $\lambda=0$.}
 
\textbf{Organization of the paper:} 
The rest of this paper is organized as follows. In Section \ref{sec3}, we recall some results on Wiener chaos decomposition form \cite{Cannizzaro.2022}. Then we prove some preliminary estimates on the generator of the environment process. The proof of Theorem \ref{thm2.1} is detailed in Section \ref{sec4}, where we introduce the Replacement Lemma as a tool for approximating the resolvent. Section \ref{sec5} and \ref{sec6} are dedicated to the proof of Theorem \ref{thm2.2}. We outline the proof strategy and use the Replacement Lemma to approximate the solution of generator equation. 
The appendix contains technical results necessary for the proof of the Replacement Lemma in Section \ref{sec4}.

\textbf{Notations:}
Throughout the paper, we employ the notation $a\lesssim b$ if there exists a constant $c>0$ such that $a\leq cb$, and $\lesssim_{\lambda}$ means the constant $c=c(\lambda)$, only depending on $\lambda$. $\nabla^\perp=(\partial^\perp_{x_1},\partial^\perp_{x_2}):=(\partial_{x_2},-\partial_{x_1})$. We use $\mathcal{S}(\mathbb{R}^2)$ to denote the space of Schwartz functions on $\mathbb{R}^2$, and we define the Fourier transform $\mathcal{F}(f)=\hat{f}$ for $f \in\mathcal{S}(\mathbb{R}^2)$ by
 $$\mathcal{F}(f)(p)=\hat{f}(p):=\int_{\mathbb{R}^2}f(x)e^{-\imath p\cdot x}dx.$$

\section{Preliminaries}\label{sec3}
In this section, we start by establishing some preliminary properties of the environment process $\eta_\cdot^\varepsilon$ seen by the particle, thereby deriving its generator $\mathcal{L}^\varepsilon$. Then we recall the Wiener chaos decomposition and obatin the action of $\mathcal{L}^\varepsilon$ on the Fock space. In order to iteratively approximate the solution to the resolvent equation discussed in the following section, we conduct estimates on the anti-symmetry component of $\mathcal{L}^\varepsilon$.

Recall that $N_t^{\varepsilon}=\begin{pmatrix}
N_{t}^{\varepsilon,1}\\N_{t}^{\varepsilon,2}
\end{pmatrix}$ is defined as $\frac{\hat{\lambda}}{\sqrt{\log\frac{1}{\varepsilon}}}\int_0^t\omega^{\varepsilon}(X_s^\varepsilon)ds$. For every $\varepsilon>0$, according to Assumption \ref{rho} and the definition of $\rho_\varepsilon$, $V_\varepsilon$ exhibits rotational invariance. Hence $\textbf{E}|N_t^{\varepsilon}|^2=2\textbf{E} (N_{t}^{\varepsilon,1})^2$. To prove Theorem \ref{thm2.1}, similarly as in \cite{Cannizzaro.2022}, it is more convenient to consider the Laplace transform of $\textbf{E} (N_{t}^{\varepsilon,1})^2$ rather than the expectation itself. Since we are dealing with diffusion in a random environment, we introduce the environment process $\eta_t^{\varepsilon}=\tau_{X_t^{\varepsilon}}\omega^{\varepsilon}$, which has been mentioned in Subsection \ref{sec1.2} and $\{\tau_x,x\in\mathbb{R}^2\}$ is the group of translations on $\Omega$. Denote $\mathbb{P}^
\varepsilon$ as the distribution of $\omega^\varepsilon$ (corresponding to the expectation $\mathbb{E}^\varepsilon$). 
{As described in \cite[Section 9.3.2]{Komorowski.2012}, for every $\varepsilon>0$, we introduce the generators of the group $\{\tau_x:x\in\mathbb{R}^2\}$ and denote them by $D_k: \mathcal{D}(D_k)\to L^2(\mathbb{P}^\varepsilon),k=1,2,$ where $\mathcal{D}(D_k)$ represents the domain of the generator. Then for $f\in\mathcal{D}(D_k),$ 
\begin{align}\label{def:D}
D_kf(\sigma)=\lim_{t\to0}\frac{f(\tau_{te_k}\sigma)-f(\sigma)}{t},\quad\forall\sigma\in\Omega,
\end{align}
where $\{e_1,e_2\}$ is the canonical basis of $\mathbb{R}^2$.
We also define the subspace $H^{2,\varepsilon}$ of $L^2(\mathbb{P}^\varepsilon)$ as the intersection of domains of $(D_1)^{m_1}(D_2)^{m_2}, m_1+m_2=2.$

First, we give some basic properties about the environment process for every $\varepsilon>0$.
\begin{lemma}
For every $\varepsilon>0$, the process $\{\eta_t^\varepsilon\}_{t\geq 0}$ is Markovian, and $\mathbb{P}^\varepsilon$ is an invariant and ergodic measure for its semigroup $\{P_t^\varepsilon\}_{t\geq 0}$. Denote its generator as $\mathcal{L}^\varepsilon$, then it can be decomposed into symmetric part $\mathcal{L}_0$ and anti-symmetric part $\mathcal{A}^\varepsilon$. More precisely, for every $f\in H^{2,\varepsilon}$, we have
\begin{align}\label{gen1}
\mathcal{L}_0 f=\frac{\nu^2}{2}\sum_{k=1}^2 (D_k)^2 f,\quad
\mathcal{A}^\varepsilon f=\mathcal{V}^\varepsilon \cdot Df,
\end{align}
where $Df=\begin{pmatrix}
    D_1f\\D_2f
\end{pmatrix}$, and $\mathcal{V}^\varepsilon$ is a random vector on $(\Omega,\mathbb{P}^\varepsilon)$ defined by 
\begin{align}\label{def:V}
\mathcal{V}^\varepsilon(\sigma):=\frac{\hat{\lambda}}{\sqrt{\log\frac1\varepsilon}}\sigma(0),\quad\forall\sigma\in\Omega.
\end{align}
\end{lemma}

\begin{proof}
According to \cite[Proposition 9.7]{Komorowski.2012}, $\{\eta_t^\varepsilon\}_{t\geq 0}$ is Markovian and its transition semigroup is given by
\begin{align*}
P_t^\varepsilon f(\sigma)=\int_{\mathbb{R}^2}p_t^{\varepsilon,\sigma}(0,x)f(\tau_x\sigma)dx,\quad \forall\sigma\in\Omega,
\end{align*}
where $f$ is bounded and measurable for $\mathbb{P}^\varepsilon$, $p_t^{\varepsilon,\sigma}(x,y)$ is the transition probability density of $X_t^{\varepsilon,\sigma}$, which is the solution to \eqref{def:Y}.
Then for $f\in H^{2,\varepsilon},$ $\sigma\in\Omega$, let $h^\sigma(x):=f(\tau_x\sigma)$, we have 
\begin{align*}
\mathcal{L}^\varepsilon f(\sigma)&=\lim_{t\to 0}\frac{P_t^\varepsilon f(\sigma)-f(\sigma)}{t}
=\lim_{t\to0}\int_{\mathbb{R}^2}\frac{p_t^{\varepsilon,\sigma}(0,x)}{t}(h^\sigma(x)-h^\sigma(0))dx\\
&=L^{\varepsilon,\sigma}h^\sigma(0),
\end{align*}
where $L^{\varepsilon,\sigma}$ is the generator of $\{X_t^{\varepsilon,\sigma}\}_t$. Therefore by It\^o's formula, 
\begin{align*}
\mathcal{L}^\varepsilon f(\sigma)&=\frac{\nu^2}{2}\Delta h^\sigma(0)+\frac{\hat{\lambda}}{\sqrt{\log\frac1\varepsilon}}\sigma(0)\cdot\nabla h^\sigma(0)
=\frac{\nu^2}{2}\sum_{k=1}^2 (D_k)^2 f(\sigma)+\mathcal{V}^\varepsilon(\sigma) \cdot Df(\sigma).
\end{align*}
In the last, we prove that $\mathbb{P}^\varepsilon$ is an invariant and ergodic measure for the semigroup $\{P_t^\varepsilon\}_{t\geq 0}$. Since $\omega^\varepsilon$ is a stationary Gaussian field, $\mathbb{P}^\varepsilon$ is invariant under transition group, i.e.$\mathbb{P}^\varepsilon\circ\tau_x=\mathbb{P}^\varepsilon$, for all $x\in\mathbb{R}^2.$ Then for $f,g\in H^{2,\varepsilon}$, $k=1,2$, the following integration by parts formula holds:
\begin{align}\label{bypart}
\mathbb{E}^\varepsilon(gD_kf)=-\mathbb{E}^\varepsilon(fD_kg).
\end{align}
Since every $\sigma\in\Omega$ is divergence-free, we have $\mathbb{E}^\varepsilon(\mathcal{L}^\varepsilon f)=0$ by \eqref{gen1}, which implies the invariance of $\mathbb{P}^\varepsilon$. As for the ergodicity, for any $f \in\mathcal{D}(\mathcal{L}^\varepsilon)$ that satisfies $\mathcal{L}^\varepsilon f=0$, $\mathbb{E}^\varepsilon(f\mathcal{L}^\varepsilon f)=0$. We conclude from \eqref{bypart} that $Df=0$, which implies that $f$ is a constant function $\mathbb{P}^\varepsilon-a.s.$ Hence $\mathbb{P}^\varepsilon$ is ergodic.
\end{proof}
}
From the above lemma, $\mathbb{P}^\varepsilon$ is an invariant measure for the semigroup $\{P_t^\varepsilon\}_{t\geq 0}$. This ensures that the following identity holds, as demonstrated in \cite[Lemma 5.1]{Cannizzaro.2021},
\begin{equation}\label{lap}
    \int_0^\infty e^{-\lambda t}\textbf{E} (N_{t}^{\varepsilon,1})^2dt
		=\frac{2}{\lambda^2}\mathbb{E}^\varepsilon(\mathcal{V}_1^\varepsilon(\lambda-\mathcal{L}^{\varepsilon})^{-1}\mathcal{V}_1^\varepsilon),
\end{equation}
where $\mathcal{V}^\varepsilon$ is defined in \eqref{def:V}, with its first component denoted as $\mathcal{V}_1^\varepsilon$. 

In the following, we will focus on the resolvent $(\lambda-\mathcal{L}^\varepsilon)^{-1}$. In order to describe how $\mathcal{L}^\varepsilon$ acts on elements in $L^2(\mathbb{P}^\varepsilon)$ more explicitly, we first focus on the space $L^2(\mathbb{P}^\varepsilon)$.
 
\subsection{Wiener chaos decomposition}
	
Since $\mathbb{P}^
\varepsilon$, the distribution of $\omega^\varepsilon$, is Gaussian, $L^2(\mathbb{P}^\varepsilon)$ admits a Wiener chaos decomposition. We recall it from \cite{Cannizzaro.2022}.

{For $x\in\mathbb{R}^2$, define $\pi^x$ be the random variable on $(\Omega,\mathbb{P}^\varepsilon)$: $\pi^x(\sigma):=\sigma(x)$, for all $\sigma\in \Omega.$ Denote $x_{1:n}$ as the vector $(x_1,\dots,x_n),$ for $x_1,\dots,x_n\in\mathbb{R}^2.$ Let $H_0^\varepsilon$ be the set of constant random variables, and $H_n^\varepsilon$ be the closure of the span of
\begin{equation}\label{H_n}
    \{\psi_n=\sum_{j_1,\dots,j_n=1}^{2}\int_{\mathbb{R}^{2n}}f_{j_{1:n}}(x_{1:n}):\prod_{l=1}^{n}\pi_{j_l}^{x_l}:dx_{1:n}\}
\end{equation}
for $n\geq 1$, where $:\dots:$ denotes the Wick product with respect to $\mathbb{P}^\varepsilon$, $\pi_{j_l}^{x_l}$ is the $j_l$-th component of $\pi^{x_l}$, and $f_{j_{1:n}}\in L_{sym}^2(\mathbb{R}^{2n})$ is symmetric on $x_{1:n}$.} Note that the wick product here depends on $\varepsilon$. To simplify the notation, we remove $\varepsilon$ in the Wick product if there is no confusion.
It is well known (see e.g. \cite[Theorem 1.1.1]{Nualart.2013}) that
\begin{equation}\label{chaos}
L^2(\mathbb{P}^\varepsilon)=\bigoplus_{n=0}^\infty H_n^\varepsilon.
\end{equation}
For any $\psi, \phi\in L^2(\mathbb{P}^\varepsilon),$ the $L^2(\mathbb{P}^\varepsilon)$ inner product can be given by $$\langle\psi,\phi\rangle_\varepsilon:=\mathbb{E}^\varepsilon(\psi\phi)=\sum_{n=0}^{\infty}\langle\psi_n,\phi_n\rangle_\varepsilon,$$ where 
$\psi_n$ and $\phi_n$ are the projections of $\psi$ and $\phi$ onto $H^\varepsilon_n$ respectively, and for $n\geq1$,
\begin{equation}\label{product}
    \langle\psi_n,\phi_n\rangle_\varepsilon=
		n!\int_{\mathbb{R}^{2n}}\prod_{i=1}^{n}
		\frac{\widehat{V}_{\varepsilon}(p_i)}{|p_i|^2}\overline{\hat{\psi}_n(p_{1:n})}\hat{\phi}_n(p_{1:n})dp_{1:n}.
\end{equation}
Here the Fourier transform of $\psi_n$ is defined as
\begin{equation}
\mathcal{F}(\psi_n)(p_{1:n})=\hat{\psi}_n(p_{1:n}):=(-\imath)^n\sum_{j_1,\dots,j_n=1}^{2}\prod_{k=1}^{n}p_{k,j_k}^{\perp}\hat{f}_{j_{1:n}}(p_{1:n}),
\end{equation}
where $(p_{k,1}^\perp,p_{k,2}^\perp):=(p_{k,2},-p_{k,1}),$ for $k=1,\dots,n.$ Indeed, $\hat{\psi}_n$ is the kernel of $\psi_n$ in Fourier space, and this mapping serves as an isometry from $H^\varepsilon_n$ to the Fock space $\Gamma L_n^{2,\varepsilon}$. The Fock space denotes $L^2_{sym}(\mathbb{R}^{2n},\mu_n^\varepsilon)$ equipped with the scalar product $\langle\cdot,\cdot\rangle_\varepsilon$ defined in \eqref{product}, where the measure $\mu_n^\varepsilon(dp_{1:n})=n!\displaystyle\prod_{i=1}^{n}
\frac{\widehat{V}_{\varepsilon}(p_i)}{|p_i|^2}dp_{1:n}$.
 
\begin{remark}\label{rem:V}
    Note that since $\omega^\varepsilon\in\Omega$ is smooth, $f_{j_{1:n}}$ can be extended to the distributions of negative regularity. In particular, for fixed $\varepsilon>0$, $$\mathcal{V}_1^\varepsilon=\frac{\hat{\lambda}}{\sqrt{\log\frac1\varepsilon}}\int_{\mathbb{R}^2}\delta_0(x)\pi_1^xdx \in H^\varepsilon_1,$$ and $$\hat{\mathcal{V}}_1^\varepsilon(p)=-\frac{\hat{\lambda}}{\sqrt{\log\frac1\varepsilon}}\imath p_2.$$
\end{remark}

Since $L^2(\mathbb{P}^\varepsilon)$ is isometric to $\Gamma L^{2,\varepsilon}=\oplus_{n=0}^\infty \Gamma L^{2,\varepsilon}_n$, by \eqref{gen1}, we can now describe how $\mathcal{L}^\varepsilon$ acts on the Fock space.

\subsection{The action of $\mathcal{L}^\varepsilon$ on the Fock space}
For every $\varepsilon>0$, the generator $\mathcal{L}^\varepsilon$ of $\eta^\varepsilon$ is given by \eqref{gen1}. Recall from \cite{Cannizzaro.2022}, $\mathcal{A}^\varepsilon$ can be decomposed further into two parts, denoted as $\mathcal{A}_{+}^\varepsilon$ and $\mathcal{A}_{-}^\varepsilon$, satisfying $(\mathcal{A}_{+}^\varepsilon)^{*}=-\mathcal{A}_{-}^\varepsilon$.
For every $\psi_n\in H_n^\varepsilon$, we have
\begin{equation}
    \begin{aligned}\label{gen}
		&\widehat{(-\mathcal{L}_0)\psi_n}(p_{1:n})=\frac{\nu^2}{2}\left|\sum_{i=1}^{n}p_i\right|^2\hat{\psi}_n(p_{1;n}),\\
		&\widehat{\mathcal{A}_{+}^\varepsilon\psi_n}(p_{1:n+1})=
			-\frac{1}{n+1}\frac{\hat{\lambda}}{\sqrt{\log\frac1\varepsilon}}\sum_{i=1}^{n+1}\left(p_i\times \sum_{j=1}^{n+1}p_j\right)\hat{\psi}_n(p_{1:n+1\backslash i}),\\
		&\widehat{\mathcal{A}_{-}^\varepsilon\psi_n}(p_{1:n-1})=n\frac{\hat{\lambda}}{\sqrt{\log\frac1\varepsilon}}\int_{\mathbb{R}^2}\frac{\widehat{V}_{\varepsilon}(p)}{|p|^2}
			\hat{\psi}_n(p_{1:n-1},p)\left(p\times\sum_{k=1}^{n-1}p_k\right)dp,
		\end{aligned}
\end{equation}
where $p_{1:n+1\backslash i}$ denotes $p_{1:n+1}$ with $p_i$ removed, $a\times b=|a||b|\sin(\theta)$, $\theta$ is the angle between a and b.

\begin{remark}
Note that $\mathcal{L}_0$ is different from the negative operator "Laplacian" $\Delta$ in $2d$ AKPZ case (see e.g. \cite{Cannizzaro.2023}), which is given by $$\widehat{\Delta\psi_n}(k_{1:n})=-\sum_{i=1}^{n}|k_i|^2\hat{\psi}_n(k_{1;n}),\quad k_i\in \mathbb{Z}^2\backslash\{0\},\;i=1.\dots,n.$$ Here the symmetry operator $-\mathcal{L}_0$ is positive semi-definite, not invertible. So we need to modify the method in \cite{Cannizzaro.2023}.
\end{remark}

The following is the lemma of It\^o's trick, which is an important tool in proving the tightness of $\{N^\varepsilon_{\cdot}\}_\varepsilon.$
\begin{lemma}\label{lem:Ito}
Let $\lambda>0,\;p\geq 1,\;T\geq 0$. For all $\varepsilon>0,\;n\in\mathbb{N}^{+}$ and $f\in \oplus_{k=0}^n H_k^\varepsilon$, the following estimate holds:
\begin{align}\label{eq:Ito}
\mathbf{E}\left(\sup_{0\leq t\leq T}\left|\int_{0}^{t}f(\eta_s^\varepsilon)ds\right|^p\right)^{\frac1p}
\lesssim_{n,p} (T^{\frac12}+T\lambda^\frac12)||(\lambda-\mathcal{L}_0)^{-\frac12}f||_{\varepsilon}.
\end{align}		
\end{lemma}
\begin{proof}
For $\varepsilon>0$, let $\mathcal{F}_t:=\sigma\{\eta_s^\varepsilon,s\leq t\}$, the natural filtration generated by $\{\eta_t^\varepsilon\}_{t\geq 0}$. For $f\in \oplus_{k=0}^n H_k^\varepsilon$, define $h:=(\lambda-\mathcal{L}_0)^{-1}f$. Since $ h\in \oplus_{k=0}^n H_k^\varepsilon$ belongs to the domain of $\mathcal{L}^\varepsilon$, we can define an $\mathcal{F}_t$-martingale $M_t$,
\begin{align*}
M_t:=h(\eta_t^\varepsilon)-h(\eta_0^\varepsilon)-\int_0^t\mathcal{L}^\varepsilon h(\eta_s^\varepsilon)ds.
\end{align*}
By It\^o's formula and \eqref{gen1}, $M_t$ and its quadratic variation can be explicitly given by \begin{align}\label{eq:mart}
M_t=\nu\int_0^t Dh(\eta_{s}^\varepsilon)\cdot dB_s,\quad \langle M\rangle_t=\nu^2\sum_{k=1}^{2}\int_0^t (D_kh(\eta_{s}^\varepsilon))^2ds.
\end{align}
Similarly, denote $\{\mathcal{F}_t^{-}:t\in [0,T]\}$ for the backward filtration generated by $\{\eta_{T-t}^\varepsilon\}_{t\in [0,T]}$. Then the generator of this process is the adjoint of $\mathcal{L}^\varepsilon$ w.r.t $\mathbb{P}^\varepsilon$, i.e. $(\mathcal{L}^{\varepsilon})^{*}=\mathcal{L}_0-\mathcal{A}^\varepsilon_{+}-\mathcal{A}^\varepsilon_{-}$, and $h$ is also in the domain of $(\mathcal{L}^\varepsilon)^{*}$. Thus the process $\{M_t^{-}:0\leq t\leq T\}$ defined by
\begin{align*}
M_t^{-}:=h(\eta_{T-t}^\varepsilon)-h(\eta_T^\varepsilon)-\int_0^t(\mathcal{L}^{\varepsilon})^{*}h(\eta_{T-s}^\varepsilon)ds
\end{align*}
is an $\mathcal{F}_t^{-}$-martingale, and there holds 
\begin{align*}
M_t+M_T^{-}-M_{T-t}^{-}=-2\int_0^t\mathcal{L}_0h(\eta_s^\varepsilon)ds.
\end{align*}
Therefore by the definition of $h$, we have
\begin{align}\label{eq:3.12}
\int_0^tf(\eta_s^\varepsilon)ds=\frac12(M_t+M_T^{-}-M_{T-t}^{-})+\lambda\int_0^t h(\eta_s^\varepsilon)ds.
\end{align}
By Burkholder-Davis-Gundy inequality \cite[Chap4 Corollary 4.2]{Revuz.1999} and \eqref{eq:mart}, 
\begin{align}
\textbf{E}[\sup_{t\in[0,T]}|M_t|^p]&\lesssim_p \textbf{E}[\langle M\rangle_T^{\frac p2}]=\nu^{p}\textbf{E}\left[\left(\int_0^T \left(\sum_{k=1}^{2}(D_kh(\eta_{s}^\varepsilon))^2\right)ds\right)^{\frac p2}\right]\nonumber\\
&\lesssim_p \left(\int_0^T \textbf{E}\left[\left(\sum_{k=1}^{2}(D_kh(\eta_{s}^\varepsilon))^2\right)^{\frac p2}\right]^{\frac2p}ds\right)^{\frac p2}\label{eq:bdg}.
\end{align}
By the stationarity and Gaussian Hypercontractivity of Gaussian measure \cite[Theorem 5.10]{Janson.1997}, the right hand side of \eqref{eq:bdg} is equal to
$$T^{\frac p2}\mathbb{E}^\varepsilon\left[\left(\sum_{k=1}^{2}(D_kh)^2\right)^{\frac p2}\right]\lesssim_{n} T^{\frac p2}\mathbb{E}^\varepsilon\left[\left(\sum_{k=1}^{2}(D_kh)^2\right)\right]^{\frac p2}=T^{\frac p2}\langle-\mathcal{L}_0h,h\rangle_\varepsilon^\frac p2\leq T^{\frac p2}\langle(\lambda-\mathcal{L}_0)^{-1}f,f\rangle_\varepsilon^\frac p2.$$
Similarly, we can get the same uniform bound for the $L^p$ norm of $M_t^{-}$. Finally we turn to $\lambda\int_0^th(\eta_s^\varepsilon)ds$ in \eqref{eq:3.12}. Similarly as before, there holds $$\textbf{E}\left[\sup_{t\in[0,T]}\left|\lambda\int_0^th(\eta_s^\varepsilon)ds\right|^p\right]\lesssim_{n,p} \lambda^pT^p||h||_\varepsilon^p\leq \lambda^{\frac p2}T^p\langle(\lambda-\mathcal{L}_0)^{-1}f,f\rangle_\varepsilon^\frac p2.$$ So the proof is concluded by \eqref{eq:3.12}.
\end{proof}
\begin{remark}
 According to \cite[Lemma 2.4]{Komorowski.2012}, for $\varepsilon>0, T>0$, if $f\in L^2(\mathbb{P}^\varepsilon)$ satisfies $||f||^2_{-1,\varepsilon}:=\lim_{\lambda\to 0}\langle f,(\lambda-\mathcal{L}_0)^{-1}f\rangle_{\varepsilon}<\infty$, then we have
\begin{equation}\label{eq:Ito'}
    \textbf{E}\left(\sup_{0\leq t\leq T}\left|\int_{0}^{t}f(\eta_s^\varepsilon)ds\right|^2\right)\lesssim T||f||^2_{-1,\varepsilon}.
\end{equation}
However, since $-\mathcal{L}_0$ is not invertible, we mainly use \eqref{eq:Ito} below instead of \eqref{eq:Ito'}.
\end{remark}

To approximate $(\lambda-\mathcal{L}^\varepsilon)$ in an interative way as we will introduce in the next section, we will need to treat the terms as $\langle\mathcal{A}_{+}^{\varepsilon}\psi,(\lambda-\mathcal{L}_0)^{-1}\mathcal{A}_{+}^{\varepsilon}\psi\rangle_{\varepsilon}$ frequently, for $\psi\in L^2(\mathbb{P}^\varepsilon)$.
To this purpose, we introduce the following decomposition.
\begin{lemma}{\cite[Lemma 4.4]{Cannizzaro.2022}}\label{lem:diag}
		Suppose $\mathcal{S}$ is a diagonal operator with Fourier multiplier $\sigma=\{\sigma_n:\mathbb{R}^{2n}\to\mathbb{R},n\in\mathbb{N}^{+}\}$, i.e. for all $\psi\in H_n^\varepsilon$, $\widehat{\mathcal{S}\psi}(p_{1:n})=\sigma_n(p_{1:n})\hat{\psi}(p_{1:n})$.
		Then it holds that
		$$\langle\psi,(\mathcal{A}_{+}^{\varepsilon})^{*}\mathcal{S}\mathcal{A}_{+}^{\varepsilon}\psi\rangle_{\varepsilon}
		=\langle\psi,(\mathcal{A}_{+}^{\varepsilon})^{*}\mathcal{S}\mathcal{A}_{+}^{\varepsilon}\psi\rangle_{\varepsilon,diag}
		+\langle\psi,(\mathcal{A}_{+}^{\varepsilon})^{*}\mathcal{S}\mathcal{A}_{+}^{\varepsilon}\psi\rangle_{\varepsilon,off},$$
		where
		\begin{equation}
			\langle\psi,(\mathcal{A}_{+}^{\varepsilon})^{*}\mathcal{S}\mathcal{A}_{+}^{\varepsilon}\psi\rangle_{\varepsilon,diag}
			=n!\frac{\hat{\lambda}^2}{\log\frac{1}{\varepsilon}}\int_{\mathbb{R}^{2(n+1)}}\prod_{k=1}^{n+1}\frac{\widehat{V}_\varepsilon(p_k)}{|p_k|^2}\left(p_{n+1}\times\sum_{i=1}^{n}p_i\right)^2|\hat{\psi}(p_{1:n})|^2\sigma_{n+1}(p_{1:n+1})dp_{1:n+1},
		\end{equation}
		and
		\begin{equation}
			\begin{aligned}
				\langle\psi,(\mathcal{A}_{+}^{\varepsilon})^{*}\mathcal{S}\mathcal{A}_{+}^{\varepsilon}\psi\rangle_{\varepsilon,off}
				=n!n\frac{\hat{\lambda}^2}{\log\frac{1}{\varepsilon}}\int_{\mathbb{R}^{2(n+1)}}\prod_{k=1}^{n+1}\frac{\widehat{V}_\varepsilon(p_k)}{|p_k|^2}\left(p_{n+1}\times\sum_{i=1}^{n}p_i\right)\left(p_{n}\times\sum_{i=1}^{n+1}p_i\right)\times\\
				\overline{\hat{\psi}(p_{1:n})}\hat{\psi}(p_{1:n+1\backslash n})\sigma_{n+1}(p_{1:n+1})dp_{1:n+1}.
			\end{aligned}
		\end{equation}
\end{lemma}

In order to state the following lemma, we introduce the so-called number operator.
\begin{definition}\label{dfn:num}
The number operator $\mathcal{N}:L^2(\mathbb{P}^\varepsilon)\to L^2(\mathbb{P}^\varepsilon)$ is defined in the way that 
$\mathcal{N}\psi=n\psi$ for $\psi\in H_n^\varepsilon$.
\end{definition}
	
By Lemma \ref{lem:diag}, for $\mathcal{S}=(\lambda-\mathcal{L}_0)^{-1}$, we can obtain the following modification of graded sector condition in \cite[Section 2.7]{Komorowski.2012}.
	
\begin{lemma}\label{lem:estimate}
Let $\mathcal{N}$ be the number operator in Definition \ref{dfn:num}. Then for each $\lambda> 0,\;\varepsilon\in(0,\frac12)$, and $\varphi\in L^2(\mathbb{P}^\varepsilon)$, one has
		\begin{align}
			&||(\lambda-\mathcal{L}_0)^{-\frac12}\mathcal{A}_{+}^{\varepsilon}\varphi||_{\varepsilon}\lesssim ||\sqrt{\mathcal{N}}(-\mathcal{L}_0)^{\frac12}\varphi||_{\varepsilon}, \label{ineq:A+}\\
			&||(\lambda-\mathcal{L}_0)^{-\frac12}\mathcal{A}_{-}^{\varepsilon}\varphi||_{\varepsilon}\lesssim\frac{1}{\lambda}||\sqrt{\mathcal{N}}(-\mathcal{L}_0)^{\frac12}\varphi||_{\varepsilon}, \label{ineq:A-}\\
			&||(\lambda-\mathcal{L}_0)^{-1}\mathcal{A}_{+}^{\varepsilon}\varphi||_{\varepsilon}\lesssim||\sqrt{\mathcal{N}}\varphi||_{\varepsilon}\label{ineq:A+'}.
        \end{align}
\end{lemma}
	
	\begin{proof}
		By Wiener chaos decomposition \eqref{chaos}, for $n\in\mathbb{N}$, there exists $\varphi_n\in H_n^{\varepsilon}$ such that $\varphi=\sum_{n=0}^{\infty}\varphi_n$.
		Then the square of the left hand side of \eqref{ineq:A+} is
		$$||(\lambda-\mathcal{L}_0)^{-\frac12}\mathcal{A}_{+}^{\varepsilon}\varphi||_{\varepsilon}^2
		=\langle(\lambda-\mathcal{L}_0)^{-1}\mathcal{A}_{+}^{\varepsilon}\varphi,\mathcal{A}_{+}^{\varepsilon}\varphi\rangle_{\varepsilon}
		=\sum_n\langle(\lambda-\mathcal{L}_0)^{-1}\mathcal{A}_{+}^{\varepsilon}\varphi_n,\mathcal{A}_{+}^{\varepsilon}\varphi_n\rangle_{\varepsilon}.$$ Thus we only need to estimate the term $\langle(\lambda-\mathcal{L}_0)^{-1}\mathcal{A}_{+}^{\varepsilon}\varphi_n,\mathcal{A}_{+}^{\varepsilon}\varphi_n\rangle_{\varepsilon}$ for each $n$.
	By Lemma \ref{lem:diag}, we decompose the inner product as diagonal and off-diagonal terms. Since $\hat{\varphi}_n(p_{1:n})$ is symmetric in $p_{1:n}$, using Cauchy-Schwarz inequality, we have
		$$\langle(\lambda-\mathcal{L}_0)^{-1}\mathcal{A}_{+}^{\varepsilon}\varphi_n,\mathcal{A}_{+}^{\varepsilon}\varphi_n\rangle_{\varepsilon,off}\leq n\langle(\lambda-\mathcal{L}_0)^{-1}\mathcal{A}_{+}^{\varepsilon}\varphi_n,\mathcal{A}_{+}^{\varepsilon}\varphi_n\rangle_{\varepsilon,diag}.$$
While for the diagonal term,
            \begin{align*}
			&\langle(\lambda-\mathcal{L}_0)^{-1}\mathcal{A}_{+}^{\varepsilon}\varphi_n,\mathcal{A}_{+}^{\varepsilon}\varphi_n\rangle_{\varepsilon,diag}\\
			=&n!\int_{\mathbb{R}^{2n}}\prod_{k=1}^{n}\frac{\widehat{V}_\varepsilon(p_k)}{|p_k|^2}|\sum_{i=1}^{n}p_i|^2|\hat{\varphi}_n(p_{1:n})|^2dp_{1:n}\left(\frac{\hat{\lambda}^2}{\log\frac{1}{\varepsilon}}\int_{\mathbb{R}^2}\frac{\widehat{V}_\varepsilon(q)(\sin\theta)^2}{\lambda+\frac{\nu^2}{2}|\sum_{i=1}^{n}p_i+q|^2}dq\right)\\
		\leq& \frac{2}{\nu^2}||(-\mathcal{L}_0)^{\frac12}\varphi_n||_{\varepsilon}^2\sup_{p_{1:n}}\left(\frac{\hat{\lambda}^2}{\log\frac{1}{\varepsilon}}\int_{\mathbb{R}^2}\frac{\widehat{V}_\varepsilon(q)(\sin\theta)^2}{\lambda+\frac{\nu^2}{2}|\sum_{i=1}^{n}p_i+q|^2}dq\right),
		\end{align*}
		where $\theta$ is the angle between $q$ and $\sum_{i=1}^{n}p_i$. Furthermore, since $\widehat{V}_{\varepsilon}(q)= \widehat{V}(\varepsilon q)$ for all $q\in\mathbb{R}^2$, the term in the above supremum can be bounded by
		\begin{equation}
			\frac{\hat{\lambda}^2}{\log\frac{1}{\varepsilon}}\int_{\mathbb{R}^2}\frac{\widehat{V}_\varepsilon(q)(\sin\theta)^2}{\lambda+\frac{\nu^2}{2}|\sum_{i=1}^{n}p_i+q|^2}dq
			\lesssim
			\frac{1}{\log\frac{1}{\varepsilon}}\int_{\mathbb{R}^2}\frac{\widehat{V}(q)}{\varepsilon^2\lambda+\frac{\nu^2}{2}|\varepsilon\sum_{i=1}^n p_i+q|^2}dq\lesssim
			\frac{\log(1+\frac{\nu^2}{2\varepsilon^2\lambda})}{\nu^2\log\frac{1}{\varepsilon}}\lesssim 1.
			\label{3.15}
		\end{equation}
		Therefore we obtain that 
  $$||(\lambda-\mathcal{L}_0)^{-\frac12}\mathcal{A}_{+}^{\varepsilon}\varphi_n||_{\varepsilon}^2\leq(1+n)\langle(\lambda-\mathcal{L}_0)^{-1}\mathcal{A}_{+}^{\varepsilon}\varphi_n,\mathcal{A}_{+}^{\varepsilon}\varphi_n\rangle_{\varepsilon,diag}
		\lesssim||\sqrt{\mathcal{N}}(-\mathcal{L}_0)^{\frac12}\varphi_n||_{\varepsilon}^2.$$

		For \eqref{ineq:A+'}, similarly we have	
	\begin{align*}
		  &\langle(\lambda-\mathcal{L}_0)^{-1}\mathcal{A}_{+}^{\varepsilon}\varphi_n,(\lambda-\mathcal{L}_0)^{-1}\mathcal{A}_{+}^{\varepsilon}\varphi_n\rangle_{\varepsilon,diag}\\
            =&n!\int_{\mathbb{R}^{2n}}\prod_{k=1}^{n}\frac{\widehat{V}_\varepsilon(p_k)}{|p_k|^2}|\hat{\varphi}_n|^2dp_{1:n}\left(\frac{\hat{\lambda}^2}{\log\frac{1}{\varepsilon}}\int_{\mathbb{R}^2}\widehat{V}_\varepsilon(q)(\sin\theta)^2\frac{|\sum_{i=1}^{n}p_i|^2}{(\lambda+\frac{\nu^2}{2}|\sum_{i=1}^{n}p_i+q|^2)^2}dq\right)\\
		  \lesssim& \frac{1}{\log\frac{1}{\varepsilon}}||\varphi_n||_{\varepsilon}^2\sup_{p_{1:n}}\left(\int_{\mathbb{R}^2}\widehat{V}_\varepsilon(q)(\sin\theta)^2\frac{|\sum_{i=1}^{n}p_i|^2}{(\lambda+\frac{\nu^2}{2}|\sum_{i=1}^{n}p_i+q|^2)^2}dq\right).
	\end{align*}
		where $\theta$ is the angle between $q$ and $\sum_{i=1}^{n}p_i$. Since $(sin\theta)^2\leq\frac{|\sum_{i=1}^{n}p_i+q|^2}{|\sum_{i=1}^{n}p_i|^2}$, the term in the supremum in the right hand side of the inequality above can be bounded by
				\begin{align*}
		            \int_{\mathbb{R}^2}\widehat{V}_\varepsilon(q)(\sin\theta)^2\frac{|\sum_{i=1}^{n}p_i|^2}{(\lambda+\frac{\nu^2}{2}|\sum_{i=1}^{n}p_i+q|^2)^2}dq
					\leq
					\frac{2}{\nu^2}\int_{\mathbb{R}^2}\frac{\widehat{V}_\varepsilon(q)}{\lambda+\frac{\nu^2}{2}|\sum_{i=1}^{n}p_i+q|^2}dq
					\lesssim \frac{1}{\nu^4}\log(1+\frac{\nu^2}{2\varepsilon^2\lambda}).
				\end{align*}
As $\lambda$ is fixed, and $\varepsilon$ is sufficiently small, $$\log(1+\frac{\nu^2}{2\varepsilon^2\lambda})\lesssim 2\log\frac{1}{\varepsilon}+\log\frac{\nu^2}{2\lambda}\lesssim\log\frac{1}{\varepsilon}.$$ Thus \eqref{ineq:A+'} holds.
    
For \eqref{ineq:A-}, by \eqref{gen}, for $\varphi_n\in H^\varepsilon_n$, the square of the left hand side can be written as
				\begin{align*}
					&\langle(\lambda-\mathcal{L}_0)^{-1}\mathcal{A}_{-}^{\varepsilon}\varphi_n,\mathcal{A}_{-}^{\varepsilon}\varphi_n\rangle_{\varepsilon}\\
					=&(n-1)!\frac{\hat{\lambda}^2}{\log\frac{1}{\varepsilon}}\int_{\mathbb{R}^{2(n-1)}}\prod_{k=1}^{n-1}\frac{\widehat{V}_\varepsilon(p_k)}{|p_k|^2}\frac{1}{\lambda+\frac{\nu^2}{2}|\sum_{i=1}^{n-1}p_i|^2}\times\\
					&\left|n\int_{\mathbb{R}^2}\frac{\widehat{V}_{\varepsilon}(p)}{|p|^2}\hat{\varphi}_n(p_{1:n-1},p)\left(p\times\sum_{k=1}^{n-1}p_k\right)dp\right|^2dp_{1:n-1}.
				\end{align*}
				Using Cauchy-Schwarz inequality again, it is upper bounded by
				\begin{align*}
					&nn!\frac{\hat{\lambda}^2}{\log\frac{1}{\varepsilon}}\int_{\mathbb{R}^{2n}}\prod_{k=1}^{n}\frac{\widehat{V}_\varepsilon(p_k)}{|p_k|^2}\frac{\widehat{V}_\varepsilon(p_n)}{|p_n|^2}\frac{1}{\lambda+\frac{\nu^2}{2}|\sum_{i=1}^{n-1}p_i|^2}
					|\hat{\varphi}_n(p_{1:n})|^2\left(p_n\times\sum_{k=1}^{n}p_k\right)^2dp_{1:n}\\
				=&nn!\frac{\hat{\lambda}^2}{\log\frac{1}{\varepsilon}}\int_{\mathbb{R}^{2n}}\prod_{k=1}^{n}\frac{\widehat{V}_\varepsilon(p_k)}{|p_k|^2}\left|\sum_{k=1}^{n}p_k\right|^2|\hat{\varphi}_n(p_{1:n})|^2\frac{\widehat{V}_\varepsilon(p_n)\sin^2\theta}{\lambda+\frac{\nu^2}{2}|\sum_{i=1}^{n-1}p_i|^2} dp_{1:n},
				\end{align*}
		where $\theta$ is the angle between $p_n$ and $\sum_{k=1}^{n}p_k$. Then $\frac{\widehat{V}_\varepsilon(p_n)\sin^2\theta}{\lambda+\frac{\nu^2}{2}|\sum_{i=1}^{n-1}p_i|^2}$ can be bounded by a positive constant times $ \frac1\lambda$. Therefore
$$||(\lambda-\mathcal{L}_0)^{-\frac12}\mathcal{A}_{-}^{\varepsilon}\varphi_n||_{\varepsilon}^2
\lesssim\frac{1}{\lambda\log\frac1\varepsilon}\frac{2}{\nu^2}||\sqrt{\mathcal{N}}(-\mathcal{L}_0)^{\frac12}\varphi_n||_{\varepsilon}^2
\lesssim\frac1\lambda||\sqrt{\mathcal{N}}(-\mathcal{L}_0)^{\frac12}\varphi_n||_{\varepsilon}^2.$$
\end{proof}

\section{The diffusion coefficient}\label{sec4}
This section is devoted to the proof of Theorem \ref{thm2.1}. 
We start by recalling the truncation method from \cite{Cannizzaro.2022} to truncate the generator $\mathcal{L}^\varepsilon$, yielding a truncated representation for \eqref{lap}, which is characterized by a family of recursively-defined operators.
Subsequently, in Subsection \ref{sec4.1}, we prove the replacement lemma and use it to approximate the aforementioned family of operators. The truncated limit of \eqref{lap} for each $n\in\mathbb{N}$ as $\varepsilon\to 0$ is derived in Subsection \ref{sec4.2}, and finally the proof of Theorem \ref{thm2.1} is completed.

Before considering the limit of $\textbf{E}|N_t^\varepsilon|^2$ when $\varepsilon\to 0$, we take care of its Laplace transform first. As given by \eqref{lap}, we need to obtain the limit of $\langle\mathcal{V}_1^\varepsilon,(\lambda-\mathcal{L}^{\varepsilon})^{-1}\mathcal{V}_1^\varepsilon\rangle_\varepsilon$ when $\varepsilon\to 0$. It seems that we need to find a solution $\psi^{\varepsilon}$ to the generator equation
\begin{equation}
(\lambda-\mathcal{L}^{\varepsilon})\psi^{\varepsilon}=\mathcal{V}_1^\varepsilon
\end{equation} for each $\varepsilon>0$.
However, since $\mathcal{L}^\varepsilon$ is not symmetric, $\psi^{\varepsilon}$ has components in each $H_n^\varepsilon$ and it is hard to find $\psi^{\varepsilon}$ explicitly. We apply the method in Section 4 of \cite{Cannizzaro.2022} to truncate $\mathcal{L}^\varepsilon$ by defining $\mathcal{L}^\varepsilon_n:= I_{\leq n}\mathcal{L}^\varepsilon I_{\leq n}$ with $I_{\leq n}$ the othogonal projection onto $\bigoplus_{k=0}^nH_k^\varepsilon$. Then we denote $\psi^{\varepsilon,n}$ as the solution to
\begin{equation}\label{truc-geratoreq}
(\lambda-\mathcal{L}_n^{\varepsilon})\psi^{\varepsilon,n}=\mathcal{V}_1^\varepsilon.
\end{equation}
By (4.3) of \cite{Cannizzaro.2022}, we have $\psi^{\varepsilon,n}=\sum_{j=0}^\infty \psi_{j}^{\varepsilon,n}$, and $\psi_{j}^{\varepsilon,n}\in H_j^\varepsilon$ can be computed recursively by
\begin{align}
&\psi_{1}^{\varepsilon,n}=(\lambda-\mathcal{L}_0+\mathcal{H}_n^\varepsilon)^{-1}\mathcal{V}_1^\varepsilon,\nonumber\\
&\psi_{j}^{\varepsilon,n}=(\lambda-\mathcal{L}_0+\mathcal{H}^\varepsilon_{n+1-j})^{-1}\mathcal{A}_{+}^\varepsilon \psi_{j-1}^{\varepsilon,n},\quad j\geq 2,\label{trunc}\\
&\psi_0^{\varepsilon,n}=(\lambda-\mathcal{L}_0)^{-1}\mathcal{A}^\varepsilon_{-}\psi_{1}^{\varepsilon,n},\nonumber
\end{align}
where $$\mathcal{H}_1^{\varepsilon}=0,\quad \mathcal{H}_{j+1}^{\varepsilon}=(\mathcal{A}_{+}^{\varepsilon})^{*}(\lambda-\mathcal{L}_0+\mathcal{H}_j^{\varepsilon})^{-1}\mathcal{A}_{+}^{\varepsilon}.$$ Since $\psi_{1}^{\varepsilon,n}=\int_{\mathbb{R}^2}f(x)\pi_1^xdx\in H_1^\varepsilon$ for some distribution $f$, by \eqref{gen1}, for any $\sigma\in\Omega,$ we have
   \begin{align*}
\mathcal{A}^\varepsilon\psi_{1}^{\varepsilon,n}(\sigma)&=\frac{\hat{\lambda}}{\sqrt{\log\frac{1}{\varepsilon}}}\sigma(0)\cdot\int_{\mathbb{R}^2}f(x)\nabla\sigma_1(x)dx\\
&=-\frac{\hat{\lambda}}{\sqrt{\log\frac{1}{\varepsilon}}}\sum_{k=1}^2\int_{\mathbb{R}^2\times\mathbb{R}^2}\partial_{x_k}f(x)\delta_0(y)\sigma_1(x)\sigma_k(y)dxdy.
\end{align*}
By the definition of wick product, it equals 
\begin{align*}
&-\frac{\hat{\lambda}}{\sqrt{\log\frac{1}{\varepsilon}}}\left(\sum_{k=1}^2\int_{\mathbb{R}^2\times\mathbb{R}^2}\partial_{x_k}f(x)\delta_0(y):\pi_1^x\pi_k^y:dxdy(\sigma)+\sum_{k=1}^2\int_{\mathbb{R}^2\times\mathbb{R}^2}\partial_{x_k}f(x)\delta_0(y)\mathbb{E}^\varepsilon(\pi_1^x\pi_k^y)dxdy\right).
\end{align*}
Therefore 
\begin{align*}
\mathcal{A}_{-}^\varepsilon\psi_{1}^{\varepsilon,n}&=\frac{\hat{\lambda}}{\sqrt{\log\frac{1}{\varepsilon}}}\sum_{k=1}^2\int_{\mathbb{R}^2\times\mathbb{R}^2}\partial_{x_k}f(x)\delta_0(y)\partial_{x_1}^\perp\partial_{x_k}^\perp V^\varepsilon*g(x-y)dxdy\\
&=\frac{\hat{\lambda}}{\sqrt{\log\frac{1}{\varepsilon}}}\int_{\mathbb{R}^2}\sum_{k=1}^2\partial_{x_k}\partial_{x_k}^\perp\partial_{x_1}^\perp f(x) V^\varepsilon*g(x)dx=0.
\end{align*}
Consequently, we have $\psi_0^{\varepsilon,n}=0,$ and $\psi^{\varepsilon,n}=\sum_{j=1}^\infty \psi_{j}^{\varepsilon,n}\in\oplus_{j=1}^\infty H_j^\varepsilon.$ Since $\mathcal{V}_1^\varepsilon\in H_1^\varepsilon$, and $\mathcal{L}^\varepsilon_n$ is a truncated approximation for $\mathcal{L}^\varepsilon$, we turn to consider the limit of
\begin{equation}\label{n-diffusive}
\langle\mathcal{V}_1^\varepsilon,(\lambda-\mathcal{L}_0+\mathcal{H}_n^{\varepsilon})^{-1}\mathcal{V}_1^\varepsilon\rangle_\varepsilon
\end{equation}
when $\varepsilon\to0$ instead of $\langle\mathcal{V}_1^\varepsilon,(\lambda-\mathcal{L}^{\varepsilon})^{-1}\mathcal{V}_1^\varepsilon\rangle_\varepsilon$.
Since $\mathcal{H}_n^\varepsilon$ is defined recursively, we need to give an estimate for the operator $\mathcal{H}_n^\varepsilon$ as $\varepsilon\to 0$.

\subsection{The replacement operator}\label{sec4.1} The main result of this subsection is the following proposition, which shows that $\mathcal{H}_n^\varepsilon$ can be approximated by a diagonal operator which is given by a nonlinear transformation of $\mathcal{L}_0$. To do this,
we recall some notations from \cite{Cannizzaro.2023} first.
For $0<\varepsilon<1$, let $L^{\varepsilon}$ be the non-negative function on $(0,\infty)$ defined as
\begin{equation}\label{Le}
L^{\varepsilon}(x):=\frac{\pi\hat{\lambda}^2}{|\log\varepsilon^2|}\log\left(1+\frac{1}{\varepsilon^2x}\right).
\end{equation}
\begin{proposition}\label{prop:rep}
Let $0<\varepsilon<\frac12$, $L^\varepsilon$ be defined as \eqref{Le}. Then for every $\lambda> 0$, $n,j\in \mathbb{N}^{+}$, there exists a constant $C=C(n,j,\lambda)$ such that for any $\varphi_1, \varphi_2\in H_n^{\varepsilon}$, there holds that
\begin{equation}\label{ineq:proprep}
|\langle\varphi_1,\left[\mathcal{H}_j^{\varepsilon}+\frac{4}{\nu^4}G_j(L^{\varepsilon}(\lambda-\mathcal{L}_0))\mathcal{L}_0\right]\varphi_2\rangle_{\varepsilon}|\leq C\delta_{\varepsilon}||(-\mathcal{L}_0)^{\frac12}\varphi_1||_{\varepsilon}||(-\mathcal{L}_0)^{\frac12}\varphi_2||_{\varepsilon},
\end{equation}
where $\delta_{\varepsilon}\to 0$ as $\varepsilon\to 0$ uniformly in $n,j$ and $\lambda$.
$G_j$ is defined on $[0,\infty)$ by
\begin{align*}
 G_{j}(x)=\left\{
\begin{aligned}
&0,    &j=1, \\
&\int_{0}^{x}\frac{1}{1+\frac{4}{\nu^4}G_{j-1}(y)}dy, &j\geq 2,
\end{aligned}\right.
\end{align*}
and for all $j\in\mathbb{N}^{+},x\geq 0$, it satisfies

$(i)\; G_j(0)=0,\;G_j(x)\geq 0,$

$(ii)\;|G_j(x)|\leq x,\;|G_j^{'}(x)|\leq 1,\;|G_j^{''}(x)|\leq \frac{4}{\nu^4}.$
\end{proposition}
The proof of this proposition relies on the Replacement Lemma below, which provides an estimate for a broader class of operators and plays a crucial role in the following section.
\begin{lemma}[Replacement Lemma]\label{lem:rep}
Suppose $H,H^{+}$ be real differentiable functions on $[0,\infty)$ satisfying
				
$(i)$ for every compact subset $D\subseteq[0,\infty)$, there exists a finite constant $K>0$ such that $\sup_{y\in D}\{|H(y)|,\;|H^{'}(y)|,\;|H^{+}(y)|,\;|(H^{+})^{'}(y)|\}\leq K$,
				
$(ii)$ for all $y\geq 0$, $H(y)\geq 1, \;H^{+}(y)\geq 0.$ 
    
Let $\lambda> 0$, $0<\varepsilon<\frac12$, $\mathcal{P}^{\varepsilon}:=\big(\lambda-H^{+}(L^{\varepsilon}(\lambda-\mathcal{L}_0))\mathcal{L}_0\big)\big(\lambda-H(L^{\varepsilon}(\lambda-\mathcal{L}_0))\mathcal{L}_0\big)^{-2}$. Then for every $n\in \mathbb{N}^{+}$, $\varphi_1,\varphi_2\in H_n^{\varepsilon}$, there exists a constant $C=C(n,K,\lambda)$ such that
\begin{equation}\label{ineq:rep}
|\langle\big((\mathcal{A}_{+}^{\varepsilon})^{*}\mathcal{P}^{\varepsilon}\mathcal{A}_{+}^{\varepsilon}+\frac{4}{\nu^4}\widetilde{H}(L^{\varepsilon}(\lambda-\mathcal{L}_0))\mathcal{L}_0\big)\varphi_1,\varphi_2\rangle_{\varepsilon}|\leq C\delta_{\varepsilon}||(-\mathcal{L}_0)^{\frac12}\varphi_1||_{\varepsilon}||(-\mathcal{L}_0)^{\frac12}\varphi_2||_{\varepsilon},
\end{equation}
where $\delta_{\varepsilon}\to 0$ as $\varepsilon\to 0$ uniformly in $n,K,\lambda$, and $\widetilde{H}$ is defined as 
\begin{equation}\label{wH}
\widetilde{H}(x):=\int_{0}^{x}\frac{H^{+}(y)}{(H(y))^2}dy.
\end{equation}
\end{lemma}
\begin{proof}
By Lemma \ref{lem:diag}, the first part of the left hand side of \eqref{ineq:rep} can be decomposed as diagonal and off-diagonal terms. Then we can obtain that
\begin{align}
&|\langle\big((\mathcal{A}_{+}^{\varepsilon})^{*}\mathcal{P}^{\varepsilon}\mathcal{A}_{+}^{\varepsilon}+\frac{4}{\nu^4}\widetilde{H}(L^{\varepsilon}(\lambda-\mathcal{L}_0))\mathcal{L}_0\big)\varphi_1,\varphi_2\rangle_{\varepsilon}| \notag\\
\leq& |\langle\mathcal{P}^{\varepsilon}\mathcal{A}_{+}^{\varepsilon}\varphi_1,\mathcal{A}_{+}^{\varepsilon}\varphi_2\rangle_{\varepsilon,diag}+\frac{4}{\nu^4}\langle\widetilde{H}(L^{\varepsilon}(\lambda-\mathcal{L}_0))\mathcal{L}_0\varphi_1,\varphi_2\rangle_{\varepsilon}| \label{rep1}\\
&+|\langle\mathcal{P}^{\varepsilon}\mathcal{A}_{+}^{\varepsilon}\varphi_1,\mathcal{A}_{+}^{\varepsilon}\varphi_2\rangle_{\varepsilon,off}|. \label{rep2}
\end{align}
We will estimate \eqref{rep1} and \eqref{rep2} separately.
				
First for \eqref{rep1}, $\mathcal{P}^{\varepsilon}$ is a diagonal operator with the Fourier multiplier $s^\varepsilon:=\{s_n^\varepsilon:n\in\mathbb{N}_+\}$, such that for all $\psi\in H_n^\varepsilon$,
$$\widehat{\mathcal{P}^{\varepsilon}\psi}(p_{1:n})=s_n^\varepsilon(p_{1:n})\hat{\psi}(p_{1:n}),\quad s_n^\varepsilon(p_{1:n})=\frac{\lambda+\frac{\nu^2}{2}|\sum_{i=1}^{n}p_i|^2H^{+}(L^{\varepsilon}(\lambda+\frac{\nu^2}{2}|\sum_{i=1}^{n}p_i|^2))}{(\lambda+\frac{\nu^2}{2}|\sum_{i=1}^{n}p_i|^2H(L^{\varepsilon}(\lambda+\frac{\nu^2}{2}|\sum_{i=1}^{n}p_i|^2)))^2}.$$
Therefore \eqref{rep1} is equal to
\begin{align*}
&\left|n!\frac{\hat{\lambda}^2}{\log\frac{1}{\varepsilon}}\int_{\mathbb{R}^{2(n+1)}}\prod_{k=1}^{n+1}\frac{\widehat{V}_\varepsilon(p_k)}{|p_k|^2}s_{n+1}^\varepsilon(p_{1:n+1})\left(p_{n+1}\times\sum_{i=1}^{n}p_i\right)^2\hat{\varphi}_1(p_{1:n})\overline{\hat{\varphi}_2(p_{1:n})}dp_{1:n+1}\right.\\
 &\left.-n!\int_{\mathbb{R}^{2n}}\prod_{k=1}^{n}\frac{\widehat{V}_\varepsilon(p_k)}{|p_k|^2}\frac{\nu^2}{2}\left|\sum_{i=1}^{n}p_i\right|^2\frac{4}{\nu^4}\widetilde{H}(L^{\varepsilon}(\lambda+\frac{\nu^2}{2}|\sum_{i=1}^{n}p_i|^2))\hat{\varphi}_1(p_{1:n})\overline{\hat{\varphi}_2(p_{1:n})}dp_{1:n}\right|\\
=&n!\left|\int_{\mathbb{R}^{2n}}\prod_{k=1}^{n}\frac{\widehat{V}_\varepsilon(p_k)}{|p_k|^2}\frac{\nu^2}{2}\left|\sum_{i=1}^{n}p_i\right|^2\hat{\varphi}_1(p_{1:n})\overline{\hat{\varphi}_2(p_{1:n})}\times\right.\\
&\left.\left(\int_{\mathbb{R}^2}\frac{\hat{\lambda}^2}{\log\frac{1}{\varepsilon}}\frac{2}{\nu^2}\widehat{V}_\varepsilon(q)(\sin^2\theta)s^\varepsilon_{n+1}(p_{1:n},q)dq-\frac{4}{\nu^4}\widetilde{H}(L^{\varepsilon}(\lambda+\frac{\nu^2}{2}|\sum_{i=1}^{n}p_i|^2))\right)dp_{1:n}\right|\\
\leq& \widetilde{\delta}_\varepsilon n!\left|\int_{\mathbb{R}^{2n}}\prod_{k=1}^{n}\frac{\widehat{V}_\varepsilon(p_k)}{|p_k|^2}\frac{\nu^2}{2}\left|\sum_{i=1}^{n}p_i\right|^2\hat{\varphi}_1(p_{1:n})\overline{\hat{\varphi}_2(p_{1:n})}dp_{1:n}\right|\\
\leq&\widetilde{\delta_{\varepsilon}}||(-\mathcal{L}_0)^{\frac12}\varphi_1||_{\varepsilon}||(-\mathcal{L}_0)^{\frac12}\varphi_2||_{\varepsilon},
\end{align*}
where we use Cauchy-Schwarz inequality in the last step, and 
$$\widetilde{\delta_{\varepsilon}}:=\sup_{x\in \mathbb{R}^{2n}:\sum_{i=1}^nx_i\neq0}\left|\frac{\hat{\lambda}^2}{\log\frac{1}{\varepsilon}}\frac{2}{\nu^2}\int_{\mathbb{R}^2}\widehat{V}_\varepsilon(q)(\sin^2\theta)s_{n+1}^\varepsilon(x_{1:n},q)dq-\frac{4}{\nu^4}\widetilde{H}(L^{\varepsilon}(\lambda+\frac{\nu^2}{2}|\sum_{i=1}^{n}x_i|^2))\right|,$$
$\theta$ is the angle between $q$ and $\sum_{i=1}^{n}x_i$.
By Lemma \ref{lem7.1} in Appendix, where $\Gamma_1(0)=\varepsilon^2\lambda+\frac{\nu^2}{2}|\varepsilon\sum_{i=1}^nx_i|^2$ in this case, we deduce that
\begin{align*}
\widetilde{\delta_{\varepsilon}}=&\sup_{x\in \mathbb{R}^{2n}:\sum_{i=1}^nx_i\neq0}\left|\frac{\hat{\lambda}^2}{\log\frac{1}{\varepsilon}}\frac{2\pi}{\nu^4}\int_{\Gamma_1(0)}^{1}\frac{H^{+}(L^{\varepsilon}(\frac{\rho}{\varepsilon^2}))}{\rho(\rho+1)(H(L^{\varepsilon}(\frac{\rho}{\varepsilon^2})))^2}d\rho-
\frac{4}{\nu^4}\widetilde{H}(L^{\varepsilon}(\lambda+\frac{\nu^2}{2}|\sum_{i=1}^{n}x_i|^2))\right|\\
&+\frac{\hat{\lambda}^2}{\log\frac{1}{\varepsilon}}O(1),
\end{align*}
where $O(1)$ denotes a uniform bound as $\varepsilon\to0$, $L^\varepsilon$ and $\widetilde{H}$ are given by \eqref{Le}, \eqref{wH} respectively. So we only need to estimate for the supremum term above. By making the variable substitution for the integral:  $t=L^{\varepsilon}(\frac{\rho}{\varepsilon^2})$, $dt=-\frac{\pi\hat{\lambda}^2}{\log\frac{1}{\varepsilon^2}}\frac{1}{\rho(\rho+1)}d\rho,$ the term in the supremum equals
\begin{align}
&\left|\frac{4}{\nu^4}\int_{L^{\varepsilon}(\frac{1}{\varepsilon^2})}^{L^{\varepsilon}(\lambda+\frac{\nu^2}{2}|\sum_{i=1}^{n}x_i|^2)}\frac{H^{+}(t)}{(H(t))^2}dt-\frac{4}{\nu^4}\int_{0}^{L^{\varepsilon}(\lambda+\frac{\nu^2}{2}|\sum_{i=1}^{n}x_i|^2)}\frac{H^{+}(t)}{(H(t))^2}dt\right|\nonumber\\
=&\frac{4}{\nu^4}\int_{0}^{L^{\varepsilon}(\frac{1}{\varepsilon^2})}\frac{H^{+}(t)}{(H(t))^2}dt
.\label{prf:rep1}
\end{align}
For all $0<\varepsilon<\frac12,$ $L^{\varepsilon}(\frac{1}{\varepsilon^2})=\frac{\pi\hat{\lambda}^2}{\log\frac{1}{\varepsilon^2}}\log2
\leq C,$ for some constant $C>0.$ Thus by the condition $(i),(ii)$, there exists a constant $K$ such that for all $x\in [0,C]$, $|H^{+}(x)|\leq K$ and $H(x)\geq 1$. Then we have 
$\eqref{prf:rep1}\lesssim L^{\varepsilon}(\frac{1}{\varepsilon^2}).$
Therefore there exists a constant $C=C(n,K,\lambda),$ such that $\widetilde{\delta_{\varepsilon}}\leq C\frac{1}{\log\frac{1}{\varepsilon^2}}:=C\delta_{\varepsilon},$ which implies that \eqref{rep1}$\lesssim\delta_{\varepsilon}||(-\mathcal{L}_0)^{\frac12}\varphi_1||_{\varepsilon}||(-\mathcal{L}_0)^{\frac12}\varphi_2||_{\varepsilon}.$

We now treat \eqref{rep2}. Since $\hat{\varphi}_1,\hat{\varphi}_2$ are symmetric, from the proof of \cite[Lemma 4.6]{Cannizzaro.2022}, \eqref{rep2} is upper bounded by
		\begin{align}
				&n!n\frac{\hat{\lambda}^2}{\log\frac{1}{\varepsilon}}\int_{\mathbb{R}^{2(n+1)}}\prod_{k=1}^{n+1}\frac{\widehat{V}_\varepsilon(p_k)}{|p_k|^2}\left(p_{n+1}\times\sum_{i=1}^{n}p_i\right)^2\frac{|\sum_{i=1}^{n}p_i|}{|p_{n+1}+\sum_{i=1}^{n-1}p_i|}s_{n+1}^\varepsilon(p_{1:n+1})\hat{\varphi_1}(p_{1:n})\overline{\hat{\varphi_1}(p_{1:n})}dp_{1:n+1} \nonumber\\
                =&n!\int_{\mathbb{R}^{2n}}\prod_{k=1}^{n}\frac{\widehat{V}_\varepsilon(p_k)}{|p_k|^2}\left|\sum_{i=1}^{n}p_i\right|^2\hat{\varphi_1}(p_{1:n})\overline{\hat{\varphi_1}(p_{1:n})}dp_{1:n}\times\nonumber\\
                &\left(n\frac{\hat{\lambda}^2}{\log\frac{1}{\varepsilon}}\int_{R^2}\widehat{V}_\varepsilon(q)(\sin\theta)^2\frac{|\sum_{i=1}^{n}p_i|}{|q+\sum_{i=1}^{n-1}p_i|}s_{n+1}^\varepsilon(p_{1:n},q)dq\right) \nonumber\\
				\leq&\frac{2}{\nu^2}||(-\mathcal{L}_0)^{\frac12}\varphi_1||_{\varepsilon}||(-\mathcal{L}_0)^{\frac12}\varphi_2||_{\varepsilon}\times \nonumber\\
					&\sup_{p_{1:n}}\left(n\frac{\hat{\lambda}^2}{\log\frac{1}{\varepsilon}}\int_{\mathbb{R}^2}\widehat{V}_\varepsilon(q)(\sin\theta)^2\frac{|\sum_{i=1}^{n}p_i|}{|q+\sum_{i=1}^{n-1}p_i|}
					\frac{\lambda+\frac{\nu^2}{2}|\sum_{i=1}^{n}p_i+q|^2H^{+}(L^{\varepsilon}(\lambda+\frac{\nu^2}{2}|\sum_{i=1}^{n}p_i+q|^2))}{(\lambda+\frac{\nu^2}{2}|\sum_{i=1}^{n}p_i+q|^2H(L^{\varepsilon}(\lambda+\frac{\nu^2}{2}|\sum_{i=1}^{n}p_i+q|^2)))^2}dq\right), \label{rep:sup}
				\end{align}
    where $\theta$ is the angle between $q$ and $\sum_{i=1}^{n}p_i$.
    Since $\widehat{V}_\varepsilon(q)=\hat{V}(\varepsilon q)$, we make variable substitution and let $q_1:=\varepsilon\sum_{i=1}^{n-1}p_i,\;q_2:=\varepsilon\sum_{i=1}^{n}p_i$ for simplicity. Then the term in the supremum of \eqref{rep:sup} is bounded by
\begin{align}
		&n\frac{\hat{\lambda}^2}{\log\frac{1}{\varepsilon}}\int_{\mathbb{R}^2}\hat{V}(q)(\sin\theta)^2\frac{|q_2|}{|q+q_1|}
		\frac{\varepsilon^2\lambda+\frac{\nu^2}{2}|q_2+q|^2H^{+}(L^{\varepsilon}(\lambda+\frac{\nu^2}{2\varepsilon^2}|q_2+q|^2))}{(\varepsilon^2\lambda+\frac{\nu^2}{2}|q_2+q|^2H(L^{\varepsilon}(\lambda+\frac{\nu^2}{2\varepsilon^2}|q_2+q|^2)))^2}dq \nonumber\\
		\lesssim_n&\frac{\hat{\lambda}^2}{\log\frac{1}{\varepsilon}}\int_{\mathbb{R}^2}\hat{V}(q)(\sin\theta)^2\frac{|q_2|}{|q+q_1|(\varepsilon^2\lambda+\frac{\nu^2}{2}|q_2+q|^2)}dq.\label{rep:integral}
\end{align}
		In order to estimate the integral in \eqref{rep:integral}, similarly to the proof of	\cite[Lemma A.3]{Cannizzaro.2022}, let $\mathbb{R}^2=\Omega_1\cup\Omega_2\cup\Omega_3$, where $\Omega_1=\{x:|q_1+x|<\frac{|q_2|}{2}\},\;\Omega_2=\{x:|q_2+x|<\frac{|q_2|}{2}\}\backslash\Omega_1,\;\Omega_3=\mathbb{R}^2\backslash (\Omega_1\cup\Omega_2)$.
		Since $\theta$ is the angle between $q$ and $q_2$, from \cite[Lemma A.3]{Cannizzaro.2022}, $(\sin\theta)^2\leq\frac{4|q+q_2|^2}{|q_2|^2\vee(\frac14|q_1|^2)}$, for every $q\in\Omega_1$. Thus the integral of \eqref{rep:integral} in $\Omega_1$ can be bounded by
				\begin{align*}
					\frac{\hat{\lambda}^2}{\log\frac{1}{\varepsilon}}\int_{\Omega_1}\frac{\hat{V}(q)(\sin\theta)^2}{\varepsilon^2\lambda+\frac{\nu^2}{2}|q_2+q|^2}\frac{|q_2|}{|q+q_1|}dq
					&\lesssim\frac{1}{\log\frac{1}{\varepsilon}}\int_{\Omega_1}\frac{2}{\nu^2|q_2||q+q_1|}dq\\
					&\lesssim\frac{1}{\log\frac{1}{\varepsilon}}\int_{|q|\leq\frac{|q_2|}{2}}\frac{1}{|q_2||q|}dq\lesssim\frac{1}{\log\frac{1}{\varepsilon}}.
				\end{align*}
		For $\Omega_2$, since $(\sin\theta)^2\leq\frac{|q_2+q|^2}{|q_2|^2}$ and $|x|\leq|x+q_2|+|q_2|\leq\frac{3|q_2|}{2}$, for every $x\in\Omega_2$, we obtain that
				\begin{align*}
					\frac{\hat{\lambda}^2}{\log\frac{1}{\varepsilon}}\int_{\Omega_2}\frac{\hat{V}(q)(\sin\theta)^2}{\varepsilon^2\lambda+\frac{\nu^2}{2}|q_2+q|^2}\frac{|q_2|}{|q+q_1|}dq
					&\lesssim\frac{1}{\log\frac{1}{\varepsilon}}\int_{\Omega_2}\frac{2}{\nu^2|q_2||q+q_1|}dq\\
					&\leq\frac{1}{\log\frac{1}{\varepsilon}}\int_{|q|\leq\frac{3|q_2|}{2}}\frac{1}{|q_2||q+q_1|}dq
					\lesssim\frac{1}{\log\frac{1}{\varepsilon}}.
				\end{align*}
For $\Omega_3$, we bound $(\sin\theta)^2$ by 1. Recall that the support of $\hat{V}$ is in the ball of radius 1,
\begin{align}\label{eq:4.14}
	\frac{\hat{\lambda}^2}{\log\frac{1}{\varepsilon}}\int_{\Omega_3}\frac{\hat{V}(q)(\sin\theta)^2}{\varepsilon^2\lambda+\frac{\nu^2}{2}|q_2+q|^2}\frac{|q_2|}{|q+q_1|}dq
	\lesssim\frac{1}{\log\frac{1}{\varepsilon}}|q_2|\int_{\Omega_3}\frac{1}{|q_2+q|^2}\frac{1}{|q_1+q|}dq.
\end{align}
Then by applying the H\"older's inequality with exponents $\frac32$ and 3, the right hand side of \eqref{eq:4.14} is bounded by
\begin{align}
\frac{1}{\log\frac{1}{\varepsilon}}|q_2|\left(\int_{\Omega_3}\frac{1}{|q_2+q|^3}dq\right)^{\frac23}\left(\int_{\Omega_3}\frac{1}{|q_1+q|^3}dq\right)^{\frac13}
\leq \frac{|q_2|}{\log\frac{1}{\varepsilon}}\int_{|q|\geq\frac{|q_2|}{2}}\frac{1}{|q|^3}dq
\lesssim \frac{1}{\log\frac{1}{\varepsilon}}.
\end{align}
Therefore $$\eqref{rep2}\lesssim\delta_{\varepsilon}||(-\mathcal{L}_0)^{\frac12}\varphi_1||_{\varepsilon}||(-\mathcal{L}_0)^{\frac12}\varphi_2||_{\varepsilon}.$$ 

Collect all the bounds so far, \eqref{ineq:rep} follows at once.
\end{proof}

Now we can give the proof of Proposition \ref{prop:rep}.
\begin{proof}[Proof of Proposition \ref{prop:rep}]
For every $0<\varepsilon<\frac12,$ we take induction for $j$. If $j=1$, $\mathcal{H}_1^\varepsilon=G_1=0$, \eqref{ineq:proprep} holds naturally.
Suppose \eqref{ineq:proprep} holds for some $j\geq 1$, then for $j+1$, in order to apply Lemma \ref{lem:rep}, using the inductive definition of $\mathcal{H}_{j+1}^{\varepsilon}$, we plus and minus one term to get $$\langle[\mathcal{H}_{j+1}^{\varepsilon}+
\frac{4}{\nu^4}G_{j+1}(L^{\varepsilon}(\lambda-\mathcal{L}_0))\mathcal{L}_0]\varphi_1,\varphi_2\rangle_{\varepsilon}=:(I)+(II),$$
where
\begin{align*}
(I)&=\langle\Big[(\lambda-\mathcal{L}_0+\mathcal{H}_j^{\varepsilon})^{-1}-\Big(\lambda-\big(1+\frac{4}{\nu^4}G_{j}(L^{\varepsilon}(\lambda-\mathcal{L}_0))\big)\mathcal{L}_0\Big)^{-1}\Big]\mathcal{A}^{\varepsilon}_{+}\varphi_1,\mathcal{A}^{\varepsilon}_{+}\varphi_2\rangle_{\varepsilon},\\
(II)&=\langle\Big[(\mathcal{A}^{\varepsilon}_{+})^{*}\Big(\lambda-\big(1+\frac{4}{\nu^4}G_{j}(L^{\varepsilon}(\lambda-\mathcal{L}_0))\big)\mathcal{L}_0\Big)^{-1}\mathcal{A}^{\varepsilon}_{+}+\frac{4}{\nu^4}G_{j+1}(L^{\varepsilon}(\lambda-\mathcal{L}_0))\mathcal{L}_0\Big]\varphi_1,\varphi_2\rangle_{\varepsilon}.
\end{align*}
				
For $(I)$, let $A=\lambda-\mathcal{L}_0+\mathcal{C}_{A},\;B=\lambda-\mathcal{L}_0+\mathcal{C}_{B}$, where $\mathcal{C}_{A}:=\mathcal{H}_j^\varepsilon$, $\mathcal{C}_{B}:=-\frac{4}{\nu^4}G_{j}(L^{\varepsilon}(\lambda-\mathcal{L}_0))\mathcal{L}_0$ are both non-negative definite. Then $(I)$ equals $\langle(A^{-1}-B^{-1})\mathcal{A}^{\varepsilon}_{+}\varphi_1,\mathcal{A}^{\varepsilon}_{+}\varphi_2\rangle_{\varepsilon}.$ By the symmetry of $A$, using the relation $A^{-1}-B^{-1}=A^{-1}(B-A)B^{-1}$ and the induction, we get
    \begin{align*}
    |(I)|&=|\langle(B-A)B^{-1}\mathcal{A}^{\varepsilon}_{+}\varphi_1,A^{-1}\mathcal{A}^{\varepsilon}_{+}\varphi_2\rangle_{\varepsilon}|=|\langle(-\mathcal{H}_{j}^{\varepsilon}-\frac{4}{\nu^4}G_{j}(L^{\varepsilon}(\lambda-\mathcal{L}_0))\mathcal{L}_0)B^{-1}\mathcal{A}^{\varepsilon}_{+}\varphi_1,A^{-1}\mathcal{A}^{\varepsilon}_{+}\varphi_2\rangle_{\varepsilon}|\\
    &\leq C\delta_{\varepsilon}||(-\mathcal{L}_0)^{\frac12}B^{-1}\mathcal{A}^{\varepsilon}_{+}\varphi_1||_\varepsilon||(-\mathcal{L}_0)^{\frac12}A^{-1}\mathcal{A}^{\varepsilon}_{+}\varphi_2||_\varepsilon.
    \end{align*}
    Moreover, let $\mathcal{C}$ be $\mathcal{C}_{A}$ or $\mathcal{C}_{B}$, which is positive semi-definite. Then
    \begin{align*}
		||(-\mathcal{L}_0)^{\frac12}(\lambda-\mathcal{L}_0+\mathcal{C})^{-1}\mathcal{A}^{\varepsilon}_{+}\varphi_i||^2_\varepsilon
		\leq\langle\mathcal{A}^{\varepsilon}_{+}\varphi_i,(\lambda-\mathcal{L}_0+\mathcal{C})^{-1}\mathcal{A}^{\varepsilon}_{+}\varphi_i\rangle_{\varepsilon}
		\leq\langle\mathcal{A}^{\varepsilon}_{+}\varphi_i,(\lambda-\mathcal{L}_0)^{-1}\mathcal{A}^{\varepsilon}_{+}\varphi_i\rangle_{\varepsilon}.
    \end{align*}
	By \eqref{ineq:A+}, \eqref{ineq:proprep} holds for $(I)$.

For $(II)$, let $H^+=H=\frac{4}{\nu^4}G_j+1$. Since $G_j$ has property $(i)-(ii)$ in Proposition \ref{prop:rep}, $H^{+}$ and $H$ satisfy the condition $(i),(ii)$ in Lemma \ref{lem:rep} with $\widetilde{H}=G_{j+1}$. Then by using \eqref{ineq:rep}, \eqref{ineq:proprep} follows at once.
\end{proof}

\subsection{The limiting diffusivity}\label{sec4.2}		
The aim of this subsection is to prove Theorem \ref{thm2.1}. Since we have given an estimate for the operator $\mathcal{H}_n^\varepsilon$ in Proposition \ref{prop:rep}, as what we stated in the beginning of Section \ref{sec4}, we can calculate the limit of \eqref{n-diffusive} for each $n\in\mathbb{N}$ as $\varepsilon\to 0$.
\begin{proposition}\label{prop:diffu}
    Let $n\in \mathbb{N}$, $\lambda>0$, and $\psi^{\varepsilon,n}$ be the solution to \eqref{truc-geratoreq} for $\varepsilon\in(0,1)$. Then 
    \begin{align}\label{eq:diffu-prop}
        \lim_{\varepsilon\to 0}	\langle\mathcal{V}_1^\varepsilon,\psi^{\varepsilon,n}\rangle_\varepsilon=\frac{2}{\nu^2}G_{n+1}(\pi\hat{\lambda}^2).
    \end{align}
\end{proposition}
\begin{proof}
Since $\mathcal{V}_1^\varepsilon\in H_1^\varepsilon$, we have $$\langle\mathcal{V}_1^\varepsilon,\psi^{\varepsilon,n}\rangle_\varepsilon
=\langle\mathcal{V}^\varepsilon_1,(\lambda-\mathcal{L}_0+\mathcal{H}_n^{\varepsilon})^{-1}\mathcal{V}^\varepsilon_1\rangle_\varepsilon=:(I)+(II),$$
where $\lambda>0$ and 
\begin{align*}(I)&=\langle\mathcal{V}^\varepsilon_1,[(\lambda-\mathcal{L}_0+\mathcal{H}_n^{\varepsilon})^{-1}-(\lambda-(1+\frac{4}{\nu^4}G_{n}(L^{\varepsilon}(\lambda-\mathcal{L}_0)))\mathcal{L}_0)^{-1}]\mathcal{V}^\varepsilon_1\rangle_{\varepsilon},\\
(II)&=\langle\mathcal{V}^\varepsilon_1,(\lambda-(1+\frac{4}{\nu^4}G_{n}(L^{\varepsilon}(\lambda-\mathcal{L}_0)))\mathcal{L}_0)^{-1}\mathcal{V}^\varepsilon_1\rangle_{\varepsilon}.
\end{align*}
				
	For $(I)$, as in the proof of Proposition \ref{prop:rep}, let $A=\lambda-\mathcal{L}_0+\mathcal{C}_{A},\;B=\lambda-\mathcal{L}_0+\mathcal{C}_{B}$ for $\mathcal{C}_{A}=\mathcal{H}_n^\varepsilon, \;\mathcal{C}_{B}=-\frac{4}{\nu^4}G_{n}(L^{\varepsilon}(\lambda-\mathcal{L}_0))\mathcal{L}_0.$
By Proposition \ref{prop:rep}, we have
\begin{align*}
(I)=|\langle(A^{-1}-B^{-1})\mathcal{V}^\varepsilon_1,\mathcal{V}^\varepsilon_1\rangle_{\varepsilon}|&=|\langle(B-A)B^{-1}\mathcal{V}^\varepsilon_1,A^{-1}\mathcal{V}^\varepsilon_1\rangle_{\varepsilon}|\\
&\lesssim_{n,\lambda}\delta_{\varepsilon}||(-\mathcal{L}_0)^{\frac12}B^{-1}\mathcal{V}^\varepsilon_1||_\varepsilon||(-\mathcal{L}_0)^{\frac12}A^{-1}\mathcal{V}^\varepsilon_1||_\varepsilon.
\end{align*}
Set $\mathcal{C}=\mathcal{C}_{A}$ or $\mathcal{C}_{B}$, which is positive semi-definite. Since $\mathcal{V}_1^\varepsilon\in H_1^\varepsilon$, we obtain that
				\begin{align*}
					||(-\mathcal{L}_0)^{\frac12}(\lambda-\mathcal{L}_0+\mathcal{C})^{-1}\mathcal{V}^\varepsilon_1||^2_\varepsilon
					&\leq\langle\mathcal{V}^\varepsilon_1,(\lambda-\mathcal{L}_0+\mathcal{C})^{-1}\mathcal{V}^\varepsilon_1\rangle_{\varepsilon}
					\leq\langle\mathcal{V}^\varepsilon_1,(\lambda-\mathcal{L}_0)^{-1}\mathcal{V}^\varepsilon_1\rangle_{\varepsilon}\\
				&=\frac{\hat{\lambda}^2}{\log\frac{1}{\varepsilon}}\int_{\mathbb{R}^2}\frac{\widehat{V}_{\varepsilon}(p)|p_1|^2}{|p|^2(\lambda+\frac{\nu^2}{2}|p|^2)}dp\\
					&\lesssim\frac{\hat{\lambda}^2}{\nu^2\log\frac{1}{\varepsilon}}\log(1+\frac{\nu^2}{2\varepsilon^2\lambda})\lesssim 1.
				\end{align*}
Therefore $(I)\lesssim_{n,\lambda}\delta_\varepsilon.$

For $(II)$, recall that the support of $\widehat{V}_\varepsilon=\widehat{V}(\varepsilon\cdot)$ is $\{p\in\mathbb{R}^2:|p|\leq\frac1\varepsilon\}$. $(II)$ equals
				\begin{align}
			       &\langle\mathcal{V}^\varepsilon_1,(\lambda-(1+\frac{4}{\nu^4}G_{n}(L^{\varepsilon}(\lambda-\mathcal{L}_0)))\mathcal{L}_0)^{-1}\mathcal{V}^\varepsilon_1\rangle_{\varepsilon}\nonumber\\
				    =&\frac{\hat{\lambda}^2}{\log\frac{1}{\varepsilon}}\int_{\mathbb{R}^2}\frac{\widehat{V}_{\varepsilon}(p)|p_1|^2}{|p|^2(\lambda+\frac{\nu^2}{2}|p|^2(1+\frac{4}{\nu^4}G_n(L^{\varepsilon}(\lambda+\frac{\nu^2}{2}|p|^2))))}dp \nonumber\\
					=&\frac12\frac{\hat{\lambda}^2}{\log\frac{1}{\varepsilon}}\int_{\mathbb{R}^2}\frac{\widehat{V}(p)}{\varepsilon^2\lambda+\frac{\nu^2}{2}|p|^2(1+\frac{4}{\nu^4}G_n(L^{\varepsilon}(\frac{1}{\varepsilon^2}(\varepsilon^2\lambda+\frac{\nu^2}{2}|p|^2))))}dp. \label{eq:diffu-prop-prf}
				\end{align}
Let $H=H^+=(1+\frac{4}{\nu^4}G_n)$ and $\sum_{i=1}^nx_i=0$. By Lemma \ref{lem7.1} in Appendix, \eqref{eq:diffu-prop-prf} equals
				\begin{align*}
                &\frac{2\pi\hat{\lambda}^2}{\nu^2\log\frac{1}{\varepsilon^2}}
					\int_{\varepsilon^2\lambda}^1\frac{1}{(1+\frac{4}{\nu^4}G_n(L^{\varepsilon}(\frac{1}{\varepsilon^2}\rho)))\rho(\rho+1)}d\rho+o(1)\\
					=&-\frac{2}{\nu^2}\int_{L^{\varepsilon}(\lambda)}^{L^{\varepsilon}(\frac{1}{\varepsilon^2})}\frac{1}{(1+\frac{4}{\nu^4}G_n(t))}dt+o(1)\\
					=&\frac{2}{\nu^2}(G_{n+1}(L^{\varepsilon}(\lambda))-G_{n+1}(L^{\varepsilon}(\frac{1}{\varepsilon^2})))+o(1).
                \end{align*}
As $\varepsilon\to0$, $L^{\varepsilon}(\frac{1}{\varepsilon^2})=\frac{\pi\hat{\lambda}^2}{\log\frac{1}{\varepsilon^2}}\log2\to 0$,
				and $L^{\varepsilon}(\lambda)=\frac{\pi\hat{\lambda}^2}{\log\frac{1}{\varepsilon^2}}\log(1+\frac{1}{\varepsilon^2\lambda})\to \pi\hat{\lambda}^2$. By the continuity of $G_{n+1}$, $(II)\to \frac{2}{\nu^2}G_{n+1}(\pi\hat{\lambda}^2)$ and \eqref{eq:diffu-prop} holds at once.
			\end{proof}

Now, we are ready to prove Theorem \ref{thm2.1}.
    \begin{proof}[Proof of Theorem \ref{thm2.1}]
		As mentioned in the beginning of Section \ref{sec3}, we will first determine the limit of the Laplace transform $\int_0^\infty e^{-\lambda t}\textbf{E}|N_{t}^{\varepsilon}|^2dt$ as $\varepsilon\to 0$.
		By \eqref{lap}, we need to characterize the behaviour of the product
				$$\langle\mathcal{V}^\varepsilon_1,(\lambda-\mathcal{L}^{\varepsilon})^{-1}\mathcal{V}^\varepsilon_1\rangle_\varepsilon,$$
		when $\varepsilon\to 0$. For every $\lambda>0$, let $n\in \mathbb{N}$, $\varepsilon\in(0,1)$, and $\psi^{\varepsilon,n}$ be the solution of \eqref{truc-geratoreq}. Then by \cite[Lemma 4.1]{Cannizzaro.2022}, we have
            $\langle\mathcal{V}^\varepsilon_1,\psi^{\varepsilon,2n+1}\rangle_{\varepsilon}
				\leq\langle\mathcal{V}^\varepsilon_1,(\lambda-\mathcal{L}^{\varepsilon})^{-1}\mathcal{V}^\varepsilon_1\rangle_{\varepsilon}
				\leq\langle\mathcal{V}^\varepsilon_1,\psi^{\varepsilon,2n}\rangle_{\varepsilon}.$
		Therefore we deduce from Proposition \ref{prop:diffu} that
  $$\frac{2}{\nu^2}G_{2n+2}(\pi\hat{\lambda}^2)\leq \liminf_{\varepsilon\to 0}\langle\mathcal{V}^\varepsilon_1,(\lambda-\mathcal{L}^{\varepsilon})^{-1}\mathcal{V}^\varepsilon_1\rangle_{\varepsilon}\leq \limsup_{\varepsilon\to 0}\langle\mathcal{V}^\varepsilon_1,(\lambda-\mathcal{L}^{\varepsilon})^{-1}\mathcal{V}^\varepsilon_1\rangle_{\varepsilon}\leq  \frac{2}{\nu^2}G_{2n+1}(\pi\hat{\lambda}^2).$$
		Hence it suffices to prove that $\lim_{n\to\infty}G_{n}(\pi\hat{\lambda}^2)=G(\pi\hat{\lambda}^2)$ where $G(x)=\frac{\nu^4}{4}(\sqrt{\frac{8}{\nu^4}x+1}-1).$
				
		By the property $(ii)$ of $G_n$ given by Proposition \ref{prop:rep},	the functions $G_n$, $G_n^{'}$ and $G_n^{''}$ are uniformly bounded on $[0,\pi\hat{\lambda}^2]$. By Arzela-Ascoli's theorem, $\{G_n\}$ is relatively compact in $C^1([0,\pi\hat{\lambda}^2])$, the space of continuously differentiable functions on  $[0,\pi\hat{\lambda}^2]$.
		Moreover, since $G^{'}_{n+1}(x)=\frac{1}{1+\frac{4}{\nu^4}G_n(x)}$ and $|G_2(x)-G_1(x)|\leq x$, we can observe that
		$$|G_{n+1}(x)-G_{n}(x)|=\left|\int_0^x\left(\frac{1}{1+\frac{4}{\nu^4}G_n(y)}-\frac{1}{1+\frac{4}{\nu^4}G_{n-1}(y)}\right)dy\right|\leq\int_0^x\frac{4}{\nu^4}|G_{n}(y)-G_{n-1}(y)|dy.$$
By induction, for all $x\in [0,\pi\hat{\lambda}^2]$ we have $$|G_{n}(x)-G_{n-1}(x)|\leq(\frac{4}{\nu^4})^{n-2}\frac{x^{n-1}}{(n-1)!}.$$
Hence $\{G_n\}$ converges uniformly to $G\in C^1([0,\pi\hat{\lambda}^2])$. Then it satisfies $$G^{'}(x)=\frac{1}{1+\frac{4}{\nu^4}G(x)}, \quad G(0)=0,$$
which has the unique solution 
\begin{align}
G(x)=\frac{\nu^4}{4}\left(\sqrt{\frac{8}{\nu^4}x+1}-1\right).\label{def:G}
\end{align}
Therefore by \eqref{lap} we have
  \begin{align*}
    \lim_{\varepsilon\to 0}\int_{0}^{\infty}e^{-\lambda t}\textbf{E}|N_t^{\varepsilon}|^2dt=\frac{4}{\lambda^2}\lim_{\varepsilon\to 0}\langle\mathcal{V}^\varepsilon_1,(\lambda-\mathcal{L}^{\varepsilon})^{-1}\mathcal{V}^\varepsilon_1\rangle_{\varepsilon}=\frac{4}{\lambda^2}\left(\sqrt{2\pi\hat{\lambda}^2+\frac{\nu^4}{4}}-\frac{\nu^2}{2}\right).
  \end{align*}

    In order to translate this result from Laplace transform to the limit of $\textbf{E}|N_t^{\varepsilon}|^2$, we set $x^{\varepsilon}(t):=\textbf{E}|N_t^{\varepsilon}|^2$. By \eqref{eq:Ito} with $\lambda=1$,
    \begin{align*}
    x^{\varepsilon}(t)=2\textbf{E}\left(\int_{0}^{t}\mathcal{V}_1^\varepsilon(\eta_s^{\varepsilon})ds\right)^2 &\lesssim (t+t^2)\langle\mathcal{V}_1^\varepsilon,(1-\mathcal{L}_0)^{-1}\mathcal{V}_1^\varepsilon\rangle_{\varepsilon}\\
    &=\frac12\frac{\hat{\lambda}^2}{\log\frac{1}{\varepsilon}}(t+t^2)\int_{\mathbb{R}^2}\frac{\widehat{V}_\varepsilon(p)}{1+\frac{\nu^2}{2}|p|^2}dp
    \lesssim (t+t^2),
    \end{align*}
        i.e. there exists $c>0$ such that $x^{\varepsilon}(t)\leq c(t+t^2),$ for all $t\geq 0$ and $0<\varepsilon<\frac12$. Using H{\"o}lder's inequality again, we get
        \begin{align*}
        |x^{\varepsilon}(t)-x^{\varepsilon}(s)|=2|\textbf{E}((N_{t}^{\varepsilon,1})^2-(N_{s}^{\varepsilon,1})^2)|&\leq 2(\textbf{E}(N_{t}^{\varepsilon,1}-N_{s}^{\varepsilon,1})^2)^{\frac12}\left[(\textbf{E}(N_{t}^{\varepsilon,1})^2)^\frac12+(\textbf{E}(N_{s}^{\varepsilon,1})^2)^\frac12\right]\\
	&\lesssim (|t-s|+|t-s|^2)^\frac12(t\vee s+t^2\vee s^2)^\frac12.
        \end{align*}	
		By \cite[Corollary 10.12]{Driver.2003}, the sequence of $\{x^{\frac1n}\}_n$ is relatively compact on $C([0,\infty),\mathbb{R})$ with the topology of uniform on compacts. Suppose $x\in C([0,\infty),\mathbb{R})$ be the limit point of a subsequence $\{x^{\frac{1}{n_k}}\}_k$. Then $x(t)\leq c(t+t^2),$ and by the Dominated Convergence Theorem, we obtain that for all $\lambda>0$
  $$\frac{4}{\lambda^2}\left(\sqrt{2\pi\hat{\lambda}^2+\frac{\nu^4}{4}}-\frac{\nu^2}{2}\right)=\lim_{\varepsilon\to 0}\int_{0}^{\infty}e^{-\lambda t}\textbf{E}|N_t^{\varepsilon}|^2dt
				=\lim_{k\to\infty}\int_{0}^{\infty}e^{-\lambda t}x^{\frac{1}{n_k}}(t)dt
				=\int_{0}^{\infty}e^{-\lambda t}x(t)dt.$$
		Thus by uniqueness of the Laplace transform, the limit of $\{x^{\varepsilon}\}_\varepsilon$ exists and equals $x$, $$x(t)=4\left(\sqrt{2\pi\hat{\lambda}^2+\frac{\nu^4}{4}}-\frac{\nu^2}{2}\right)t, \; t\geq 0.$$
  Therefore $\lim_{\varepsilon\to 0}\textbf{E}|N_t^{\varepsilon}|^2=\lim_{\varepsilon\to 0}x^{\varepsilon}(t)=4\left(\sqrt{2\pi\hat{\lambda}^2+\frac{\nu^4}{4}}-\frac{\nu^2}{2}\right)t.$
			\end{proof}

\section{Estimates for the solutions to generator equations}\label{sec5}
This section and the next section are devoted to proving Theorem \ref{thm2.2}. We will start by outlining our proof strategy and deriving key estimates for approximating solutions to generator equations.
First we give the following theorem about the tightness of the distributions of $\{X_{.}^\varepsilon\}_\varepsilon.$
\begin{theorem}\label{thm:tightness}
    For all $T>0$, the distributions of $\{X_{.}^\varepsilon\}_\varepsilon$ are tight in $C([0,T],\mathbb{R}^2)$.
\end{theorem}
\begin{proof}
Since $X_t^\varepsilon=N_t^\varepsilon+\nu B_t$, we only need to prove the distributions of $\{N_{.}^\varepsilon\}_\varepsilon$ are tight in $C([0,T],\mathbb{R}^2)$.
By Kolmogorov's tightness criterion, it suffices to prove that for $i=1$ and 2, there exists $\gamma,\delta>0$, and $C<\infty$, such that for all $s,t\in[0,T]$, $$\sup_\varepsilon\textbf{E}|N_{t}^{\varepsilon,i}-N_{s}^{\varepsilon,i}|^\gamma\leq C|t-s|^{1+\delta},\forall s,t\in[0,T].$$
Recall that $N_t^\varepsilon=\begin{pmatrix}N_{t}^{\varepsilon,1}\\N_{t}^{\varepsilon,2}\end{pmatrix}$. We only prove for $i=1$, the case for $i=2$ is the same. 
For every $0<\varepsilon<1$, let $p>2$ and $0\leq s\leq t\leq T$. Since $\textbf{E}|N_{t}^{\varepsilon,1}-N_{s}^{\varepsilon,1}|^p
=\textbf{E}\left(\left|\int_{s}^{t}\mathcal{V}_1^\varepsilon(\eta_s^\varepsilon)ds\right|^p\right)$, by It\^o's trick \eqref{eq:Ito}, for every $\lambda>0$, $\varepsilon\in (0,\frac12)$, we have
$$\textbf{E}|N_{t}^{\varepsilon,1}-N_{s}^{\varepsilon,1}|^p
\lesssim_{p}[(t-s)^{\frac p2}+(t-s)^p\lambda^{\frac p2}]\left(\langle\mathcal{V}^\varepsilon_1,(\lambda-\mathcal{L}_0)^{-1}\mathcal{V}^\varepsilon_1\rangle_{\varepsilon}\right)^{\frac p2},$$
and $$\langle\mathcal{V}^\varepsilon_1,(\lambda-\mathcal{L}_0)^{-1}\mathcal{V}_1^\varepsilon\rangle_{\varepsilon}\lesssim \frac{2}{\nu^2}\frac{\log(1+\frac{\nu^2}{2\varepsilon^2\lambda})}{\log\frac1\varepsilon}\lesssim 1.$$ Let $\lambda=1$. Since $|t-s|\leq 2T$, we have 
\begin{align}
\textbf{E}|N_{t}^{\varepsilon,1}-N_{s}^{\varepsilon,1}|^p\lesssim(t-s)^{\frac p2}+(t-s)^p\lesssim_T(t-s)^{\frac p2},\label{eq:N}
\end{align}
which completes the proof.
\end{proof}
			
\subsection{The strategy for proving Theorem \ref{thm2.2}}
We use the same notations as before. Recall that $N_t^\varepsilon=\int_0^t\mathcal{V}^\varepsilon(\eta_s^\varepsilon)ds.$ By Theorem \ref{thm:tightness}, for every subsequence of $\{X^{\varepsilon}_{\cdot}\}$, there exists a subsubsequence such that $X^{\varepsilon_k}_{\cdot}$ converges in distribution to $h$, a measurable mapping from $\Omega\times\Sigma$ to $C([0,T],\mathbb{R}^2).$ For simplicity, we still denote $\varepsilon$ for $\varepsilon_k$.   
			
For Theorem \ref{thm2.2}, we need to prove that the limit point $h_t=\begin{pmatrix}
				h_t^1\\h_t^2
\end{pmatrix}$ is a Brownian motion. By L\'evy's Characterization of Brownian motion, it suffices to prove $h^i$ is a martingale for $i=1,2$, and the quadratic variation $$\langle h^i\rangle_t=\left(\frac{c(\nu)^2}{2}+\nu^2\right)t,\quad \langle h^1,h^2\rangle_t=0,$$ where $c(\nu)$ is the diffusion coefficient in Theorem \ref{thm2.1}.
While for all $b^\varepsilon\in\mathcal{D}(\mathcal{L}^\varepsilon)\times \mathcal{D}(\mathcal{L}^\varepsilon)$, and $\lambda>0,$ by Dynkin's martingale formula, 
\begin{equation}\label{dynmart}
b^\varepsilon(\eta_{t}^\varepsilon)-b^\varepsilon(\eta_{0}^\varepsilon)-\int_0^t\mathcal{L}^\varepsilon b^\varepsilon(\eta_{s}^\varepsilon)ds=
b^\varepsilon(\eta_{t}^\varepsilon)-b^\varepsilon(\eta_{0}^\varepsilon)+\int_0^t(\lambda-\mathcal{L}^\varepsilon) b^\varepsilon(\eta_{s}^\varepsilon)ds-\lambda\int_0^tb^\varepsilon(\eta_{s}^\varepsilon)ds
\end{equation} is a martingale in $\mathbb{R}^2$ with respect to the natural filtration of $\{\eta_{t}^\varepsilon\}_{t\geq 0}$, which we denote as $M_t(b^\varepsilon)$.
		Comparing it with $N_t^\varepsilon$, we have $$N_{t}^\varepsilon-M_t(b^\varepsilon)=-b^\varepsilon(\eta_{t}^\varepsilon)+b^\varepsilon(\eta_{0}^\varepsilon)+\int_0^t(\mathcal{V}^\varepsilon(\eta_{s}^\varepsilon)-(\lambda-\mathcal{L}^\varepsilon) b^\varepsilon(\eta_{s}^\varepsilon))ds+\lambda\int_0^tb^\varepsilon(\eta_{s}^\varepsilon)ds.$$ 
Our goal is to show  $\textbf{E}|N_t^\varepsilon-M_t(b^\varepsilon)|^2\to 0$ s.t. $h_i,i=1,2$ is a martingale (see Subsection \ref{sec6.1}). A natural choice of $b^\varepsilon$ is the solution to generator equation $(\lambda-\mathcal{L}^\varepsilon) b^\varepsilon=\mathcal{V}^\varepsilon$ with $||b^\varepsilon||_\varepsilon\to0$. However, as what we state before, it is hard to solve this resolvent equation. Instead we consider the truncated equation
\eqref{truc-geratoreq}:
\begin{align}\label{truc-geratoreq1}
(\lambda-\mathcal{L}_n^\varepsilon) \tilde{b}^{\varepsilon,n}=\mathcal{V}^\varepsilon.
\end{align}
But $\mathcal{H}_n^\varepsilon$ is iteratively defined and is not a diagonal operator, it is still difficult to analyze for $\tilde{b}^{\varepsilon,n}\in\oplus_{k=1}^n H^\varepsilon_k$.
The approximation provided in Proposition \ref{prop:rep} enables us to give an explicit approximation for $\tilde{b}^{\varepsilon,n}$,
			\begin{align}
				&b_{1}^{\varepsilon,n}:=[\lambda-\mathcal{L}_0(1+\frac{4}{\nu^4}G_n(L^\varepsilon(\lambda-\mathcal{L}_0)))]^{-1}\mathcal{V}^\varepsilon,\nonumber\\
				&b_{j}^{\varepsilon,n}:=[\lambda-\mathcal{L}_0(1+\frac{4}{\nu^4}G_{n+1-j}(L^\varepsilon(\lambda-\mathcal{L}_0)))]^{-1}\mathcal{A}_{+}^\varepsilon b_{j-1}^{\varepsilon,n}.\label{iter-b}
			\end{align}
   Let $b^{\varepsilon,n}=\begin{pmatrix}
				(b^{\varepsilon,n})_1\\(b^{\varepsilon,n})_2
			\end{pmatrix}:=\sum_{j=1}^n b_{j}^{\varepsilon,n}$, and we denote $b_{j,i}^{\varepsilon,n}$ as the $i$-th component of $b_j^{\varepsilon,n}.$
			
The main result of this section is the following theorem, estimating for $b^{\varepsilon,n}$ and  $\mathcal{V}^\varepsilon-(\lambda-\mathcal{L}^\varepsilon)b^{\varepsilon,n}$, which will be used in the proof of Theorem \ref{thm2.2}.
\begin{theorem}\label{thm5.2}
Recall that $\mathcal{V}^\varepsilon$ is defined in \eqref{def:V} and $c$ is the coefficient in Theorem \ref{thm2.1}. For all $\lambda>0$, $i=1,2$, we have
\begin{align}
&\lim_{n\to\infty}\lim_{\varepsilon\to 0}||(b^{\varepsilon,n})_i||_\varepsilon=0, \label{5.3}\\
&\lim_{n\to\infty}\lim_{\varepsilon\to 0}||(\lambda-\mathcal{L}_0)^{-\frac12}(\mathcal{V}_i^\varepsilon-(\lambda-\mathcal{L}^\varepsilon) (b^{\varepsilon,n})_i)||_\varepsilon=0, \label{5.4}\\
&\lim_{n\to\infty}\lim_{\varepsilon\to 0}||(\lambda-\mathcal{L}_0)^{\frac12}(b^{\varepsilon,n})_i||^2_\varepsilon=\frac{c(\nu)^2}{2\nu^2}, \label{5.5}\\
&\lim_{n\to\infty}\lim_{\varepsilon\to 0}\langle(\lambda-\mathcal{L}_0)^{\frac12}(b^{\varepsilon,n})_1,(\lambda-\mathcal{L}_0)^{\frac12}(b^{\varepsilon,n})_2\rangle_\varepsilon=0. \label{5.6}
\end{align}
\end{theorem}
\subsection{Iterative estimate}
This subsection concentrates on each component of $b^{\varepsilon,n}$ in $H_j^\varepsilon,$ for $j=1,\dots,n$. Our goal is to establish conclusions in Theorem \ref{thm5.2} concerning $b_j^{\varepsilon,n}$. From now on, we fix $\lambda>0$. Given that $b^{\varepsilon,n}$ is determined iteratively from \eqref{iter-b}, our analysis initially focuses on such functionals.
\begin{lemma}\label{lem:iter-func}
		Let $j\in\mathbb{N}$, $0<\varepsilon<\frac12$, $t^\varepsilon=[\lambda-\mathcal{L}_0(1+\frac{4}{\nu^4}G_{n+1-j}(L^\varepsilon(\lambda-\mathcal{L}_0)))]^{-1}\mathcal{A}_{+}^\varepsilon g$, where $g\in H^\varepsilon_{j-1}$ and $G_j$ is given by Proposition \ref{prop:rep}.
		Then there exists $C=C(j)>0$ such that 
		\begin{align}
			||t^\varepsilon||_\varepsilon&\leq C||g||_\varepsilon, \label{ineq:iter1}\\
			||(-\mathcal{L}_0)^{\frac12}t^\varepsilon||_\varepsilon&\leq C||(-\mathcal{L}_0)^{\frac12}g||_\varepsilon.\label{ineq:iter2}
		\end{align}
\end{lemma}
\begin{proof}
	By Proposition \ref{prop:rep}, the function $G_j$ is non-negative. Then \eqref{ineq:iter1} holds from \eqref{ineq:A+'}. \eqref{ineq:iter2} is similar since $||(-\mathcal{L}_0)^{\frac12}t^\varepsilon||_\varepsilon\leq
	||(\lambda-\mathcal{L}_0)^{-\frac12}\mathcal{A}_{+}^\varepsilon g||_\varepsilon.$ Then it is a direct consequence of \eqref{ineq:A+}.
\end{proof}
			
	Now we can make an estimate to $b_{j}^{\varepsilon,n}$, which will be crucial in the proof of \eqref{5.3} and \eqref{5.4}.
			
\begin{proposition}\label{prop5.1}
	For all $n\in \mathbb{N}$, $0<\varepsilon<\frac12$, there exists $C=C(n)>0$ such that for $j=1,\dots,n$, and $i=1,2$, the following estimates hold,
\begin{align}
	||b_{j,i}^{\varepsilon,n}||_\varepsilon&\leq  \frac{C^j}{\sqrt{\lambda\log\frac1\varepsilon}}, \label{ineq-b1}\\
	||(-\mathcal{L}_0)^{\frac12}b_{j,i}^{\varepsilon,n}||_\varepsilon&\leq C^j.\label{ineq-b2}
\end{align}
\end{proposition}

\begin{proof}
	Recall that the definition of $b_{j}^{\varepsilon,n}$ is given recursively by \eqref{iter-b}. By Lemma \ref{lem:iter-func}, there exists a constant $C=C(n)$ such that $||b_{j,i}^{\varepsilon,n}||_\varepsilon\leq C^{j-1}||b_{1,i}^{\varepsilon,n}||_\varepsilon,$ and $||(-\mathcal{L}_0)^{\frac12}b_{j,i}^{\varepsilon,n}||_\varepsilon\leq C^{j-1}||(-\mathcal{L}_0)^{\frac12}b_{1,i}^{\varepsilon,n}||_\varepsilon$. Using the property that $G_j\geq 0$ again, we have
    \begin{align*}
    ||b_{1,i}^{\varepsilon,n}||_\varepsilon&=||[\lambda-\mathcal{L}_0(1+\frac{4}{\nu^4}G_n(L^\varepsilon(\lambda-\mathcal{L}_0)))]^{-1}\mathcal{V}_i^\varepsilon||_\varepsilon\leq ||(\lambda-\mathcal{L}_0)^{-1}\mathcal{V}_i^\varepsilon||_\varepsilon\\
    &=\frac{\hat{\lambda}}{\sqrt{\log\frac1\varepsilon}}\left(\frac12\int_{\mathbb{R}^2}\frac{\varepsilon^2\hat{V}(p)}{(\varepsilon^2\lambda+\frac{\nu^2}{2}|p|^2)^2}dp\right)^{\frac12}\lesssim \frac{1}{\sqrt{\lambda\log\frac1\varepsilon}}.
    \end{align*}
Similarly 
$$||(-\mathcal{L}_0)^{\frac12}b_{1,i}^{\varepsilon,n}||_\varepsilon\leq ||(\lambda-\mathcal{L}_0)^{-\frac12}\mathcal{V}_i^\varepsilon||_\varepsilon\lesssim \frac{1}{\sqrt{\log\frac1\varepsilon}}\left(\int_{\{|p|\leq 1\}}\frac{1}{\varepsilon^2\lambda+\frac{\nu^2}{2}|p|^2}dp\right)^{\frac12} \lesssim1,$$ which concludes the proof. 
\end{proof}
			
In order to derive such estimate as \eqref{5.5} and \eqref{5.6} for $b_{j}^{\varepsilon,n}$, we recall the functions $G_i^{+,n+1-j}$ from \cite{Cannizzaro.2023}, for $j=1,\dots,n$, $i=0,1,\dots,j-1.$
\begin{align}\label{dfn:fun-G}
    G_i^{+,n+1-j}(x):=\left\{\begin{aligned}
					&1, &&i=0, \\
					&\int_{\Delta_x^{i-1}}\prod_{l=0}^{i-1}\frac{1}{(1+\frac{4}{\nu^4}G_{n+1-j+l}(\frac{\nu^4}{4}x_l))^2}dx_{0:i-1}, &&i=1,\dots,j-1,
		\end{aligned}\right. 
\end{align}
where $\Delta_x^{i-1}=\{x_{0:i-1}=(x_0,\dots,x_{i-1})\in[0,\frac{4}{\nu^4}x]^i:0\leq x_0\leq\dots\leq x_{i-1}\leq \frac{4}{\nu^4}x\}.$
The next lemma gives the properties of $G_i^{+,n+1-j}$.
\begin{lemma}\cite[Lemma 4.7]{Cannizzaro.2023}\label{lem:fun-G}
Let $n\in\mathbb{N}$, $G_i^{+,n+1-j}$ be defined by \eqref{dfn:fun-G}. Then for $j=1,\dots,n$, $i=0,\dots,j-1$,
	
 $(i)\quad G_i^{+,n+1-j}(x)\geq 0$, $G_i^{+,n+1-j}(0)=1_{\{i=0\}},$
				
$(ii)$ for $i\neq 0$, 
\begin{align}
(G_i^{+,n+1-j})^{'}(x)=\frac{4}{\nu^4}\frac{G_{i-1}^{+,n+1-j}(x)}{[1+\frac{4}{\nu^4}G_{n-j+i}(x)]^2},
\end{align}
				
$(iii)$ we have the bounds:
\begin{align}
&G_i^{+,n+1-j}(x)\leq \left(\frac{4}{\nu^4}\right)^{i-1}\frac{x^{i-1}}{(i-1)!},&i\geq 1,\label{5.13}\\
&|(G_i^{+,n+1-j})^{'}(x)|\leq \left(\frac{4}{\nu^4}\right)^{i-1}\frac{x^{i-2}}{(i-2)!},&i\geq 2,\\
&|(G_i^{+,n+1-j})^{''}(x)|\leq \left(\frac{4}{\nu^4}\right)^{i-1}\left(2\frac{x^{i-2}}{(i-2)!}+\frac{x^{i-3}}{(i-3)!}\right), &i\geq 3.
\end{align}
\end{lemma}
			
Similar to Lemma \ref{lem:iter-func}, before estimating $||(\lambda-\mathcal{L}_0)^\frac12 b_{j,i}^{\varepsilon,n}||_\varepsilon$ as $\varepsilon\to 0$, we analyze functionals with a similar structure.
\begin{lemma}\label{lem:iter-func2}
Let $n\in \mathbb{N},1\leq j\leq n$, $0<\varepsilon<\frac12$, $t_a^\varepsilon=[\lambda-\mathcal{L}_0(1+\frac{4}{\nu^4}G_{n+1+i-j}(L^\varepsilon(\lambda-\mathcal{L}_0)))]^{-1}\mathcal{A}_{+}^\varepsilon g_a$, where $g_a\in H_j^\varepsilon$, $a=1,2.$
Then there exists $C=C(n,j,\lambda)>0$ such that for all $i=0,\dots,j-1,$
\begin{align*}
	&|\langle(\lambda-G_i^{+,n+1-j}(L^\varepsilon(\lambda-\mathcal{L}_0))\mathcal{L}_0)t_1^\varepsilon,t_2^\varepsilon\rangle_\varepsilon-\langle(-G_{i+1}^{+,n+1-j}(L^\varepsilon(\lambda-\mathcal{L}_0))\mathcal{L}_0)g_1,g_2\rangle_\varepsilon| \\
	\leq& C\delta_{\varepsilon}||(-\mathcal{L}_0)^{\frac12}g_1||_\varepsilon||(-\mathcal{L}_0)^{\frac12}g_2||_\varepsilon,
\end{align*}
where $\delta_{\varepsilon}\to 0$ uniformly in $n,j,\lambda$ as $\varepsilon\to 0$.
\end{lemma}
\begin{proof}
    Observe that by the definition of $t_a^\varepsilon$ and the property $(ii)$ of $G_i^{+,n+1-j}$, we can apply Replacement Lemma \ref{lem:rep}. Let $H^+=G_i^{+,n+1-j},H=1+\frac{4}{\nu^4}G_{n+1+i-j}$, $\mathcal{P}^\varepsilon$ be defined as in Lemma \ref{lem:rep}. Then $$\widetilde{H}(x)=\int_{0}^{x}\frac{H^{+}(y)}{H^2(y)}dy=\int_{0}^{x}\frac{G_i^{+,n+1-j}(y)}{(1+\frac{4}{\nu^4}G_{n+1-j+i}(y))^2}dy=\frac{\nu^4}{4}G_{i+1}^{+,n+1-j}(x),$$
    and $$\langle(\lambda-G_i^{+,n+1-j}(L^\varepsilon(\lambda-\mathcal{L}_0))\mathcal{L}_0)t_1^\varepsilon,t_2^\varepsilon\rangle_\varepsilon=\langle(\lambda-H^+(L^\varepsilon(\lambda-\mathcal{L}_0))\mathcal{L}_0)t_1^\varepsilon,t_2^\varepsilon\rangle_\varepsilon=\langle(\mathcal{A}_{+}^{\varepsilon})^{*}\mathcal{P}^{\varepsilon}\mathcal{A}_{+}^{\varepsilon}g_1,g_2\rangle_\varepsilon.$$
By Lemma \ref{lem:rep}, there exists $C=C(n,j,\lambda)$ such that $$|\langle((\mathcal{A}_{+}^{\varepsilon})^{*}\mathcal{P}^{\varepsilon}\mathcal{A}_{+}^{\varepsilon}+\frac{4}{\nu^4}\widetilde{H}(L^{\varepsilon}(\lambda-\mathcal{L}_0))\mathcal{L}_0)g_1,g_2\rangle_{\varepsilon}|\leq C\delta_{\varepsilon}||(-\mathcal{L}_0)^{\frac12}g_1||_{\varepsilon}||(-\mathcal{L}_0)^{\frac12}g_2||_{\varepsilon},$$
where $\delta_{\varepsilon}\to 0$ uniformly in $n,j,\lambda$ as $\varepsilon\to 0$.
\end{proof}


By the above lemma, we can derive a truncated limit of \eqref{5.5} and \eqref{5.6} for $b_{j}^{\varepsilon,n}$.
\begin{proposition}\label{prop:lim-b}
For every $n\in\mathbb{N}, j=1,\dots,n$, there holds that
\begin{align}
&\lim_{\varepsilon\to 0}||(\lambda-\mathcal{L}_0)^\frac12 b_{j,i}^{\varepsilon,n}||_\varepsilon^2= \frac{\nu^2}{2}G_j^{+,n+1-j}(\pi\hat{\lambda}^2),\quad i=1,2, \label{ineq:b_j}\\
&\lim_{\varepsilon\to 0}\langle(\lambda-\mathcal{L}_0)^\frac12 b_{j,1}^{\varepsilon,n},(\lambda-\mathcal{L}_0)^\frac12 b_{j,2}^{\varepsilon,n}\rangle_\varepsilon=0. \label{ineq:b_j_12}
\end{align}
\end{proposition}
\begin{proof}
    For each $j=1,\dots,n$, since $b^{\varepsilon,n}_j$ has the recursive structure by \eqref{iter-b}, we apply Lemma \ref{lem:iter-func2} inductively from $j$. By \eqref{dfn:fun-G}, $G_0^{+,n+1-j}=1$, for $i=1,2$, take $g_1=g_2=b_{j-1,i}^{\varepsilon,n}$. Then $t_1=t_2=b_{j,i}^{\varepsilon,n}$. By \eqref{ineq-b2}, we have
\begin{align*}
	||(\lambda-\mathcal{L}_0)^\frac12 b_{j,i}^{\varepsilon,n}||_\varepsilon^2
        =&\langle(\lambda-G_0^{+,n+1-j}(L^\varepsilon(\lambda-\mathcal{L}_0))\mathcal{L}_0)b_{j,i}^{\varepsilon,n},b_{j,i}^{\varepsilon,n}\rangle_\varepsilon\\
	=&||(-G_1^{+,n+1-j}(L^\varepsilon(\lambda-\mathcal{L}_0))\mathcal{L}_0)^\frac12b_{j-1,i}^{\varepsilon,n}||_\varepsilon^2+ o(1)\\
	=&||(\lambda-G_1^{+,n+1-j}(L^\varepsilon(\lambda-\mathcal{L}_0))\mathcal{L}_0)^\frac12b_{j-1,i}^{\varepsilon,n}||_\varepsilon^2-\lambda||b_{j-1,i}^{\varepsilon,n}||_\varepsilon^2+ o(1).
\end{align*}
Repeat this induction for $j-1,\dots,1$, and by \eqref{ineq-b1}, we obtain
\begin{align*}
||(\lambda-\mathcal{L}_0)^\frac12 b_{j,i}^{\varepsilon,n}||_\varepsilon^2
=||(\lambda-G_{j-1}^{+,n+1-j}(L^\varepsilon(\lambda-\mathcal{L}_0))\mathcal{L}_0)^\frac12b_{1,i}^{\varepsilon,n}||_\varepsilon^2-\lambda\sum_{l=1}^{j-1}||b_{l,i}^{\varepsilon,n}||_\varepsilon^2+o(1).
\end{align*} 
Again, since $b_{1,i}^{\varepsilon,n}=[\lambda-\mathcal{L}_0(1+\frac{4}{\nu^4}G_n(L^\varepsilon(\lambda-\mathcal{L}_0)))]^{-1}\mathcal{V}_i^\varepsilon$, let $H^+=G_{j-1}^{+,n+1-j},H=1+\frac{4}{\nu^4}G_n$. By Lemma \ref{lem7.1} in Appendix, we have
    \begin{align*}
	&||(\lambda-G_{j-1}^{+,n+1-j}(L^\varepsilon(\lambda-\mathcal{L}_0))\mathcal{L}_0)^\frac12b_{1,i}^{\varepsilon,n}||_\varepsilon^2\\
	=&||(\lambda-G_{j-1}^{+,n+1-j}(L^\varepsilon(\lambda-\mathcal{L}_0))\mathcal{L}_0)^\frac12[\lambda-\mathcal{L}_0(1+\frac{4}{\nu^4}G_n(L^\varepsilon(\lambda-\mathcal{L}_0)))]^{-1}\mathcal{V}_i^\varepsilon||_\varepsilon^2\\
	=&\frac{2\pi\hat{\lambda}^2}{\nu^2\log\frac{1}{\varepsilon^2}}				\int_{\varepsilon^2\lambda}^{1}\frac{G_{j-1}^{+,n+1-j}(L^{\varepsilon}(\frac{\rho}{\varepsilon^2}))}{\rho(\rho+1)(1+\frac{4}{\nu^4}G_n(L^{\varepsilon}(\frac{\rho}{\varepsilon^2})))^2}d\rho+o(1)\\
	=&-\frac{2}{\nu^2}\int_{L^{\varepsilon}(\lambda)}^{L^{\varepsilon}(\frac{1}{\varepsilon^2})}\frac{G_{j-1}^{+,n+1-j}(t)}{(1+\frac{4}{\nu^4}G_n(t))^2}dt+o(1)\\
	=&\frac{\nu^2}{2}G_{j}^{+,n+1-j}(L^{\varepsilon}(\lambda))-\frac{\nu^2}{2}G_{j}^{+,n+1-j}(L^{\varepsilon}(\frac{1}{\varepsilon^2}))+o(1),
\end{align*}
where $L^{\varepsilon}(\frac{1}{\varepsilon^2})=\frac{\pi\hat{\lambda}^2}{\log\frac{1}{\varepsilon^2}}\log2\to 0$, and $L^{\varepsilon}(\lambda)=\frac{\pi\hat{\lambda}^2}{\log\frac{1}{\varepsilon^2}}\log(1+\frac{1}{\varepsilon^2\lambda})\to \pi\hat{\lambda}^2,$ as $\varepsilon\to 0$.
By the continuity of $G_j^{+,n+1-j}$,
\eqref{ineq:b_j} holds.				

Similarly for the case of \eqref{ineq:b_j_12}, taking $g_i=b_{j-1,i}^{\varepsilon,n}, i=1,2$ into Lemma \ref{lem:iter-func2}, then $t_i=b_{j,i}^{\varepsilon,n},i=1,2$. By \eqref{ineq-b2}, we have$$\langle(\lambda-\mathcal{L}_0)^\frac12 b_{j,1}^{\varepsilon,n},(\lambda-\mathcal{L}_0)^\frac12 b_{j,2}^{\varepsilon,n}\rangle_\varepsilon
=\langle(\lambda-G_{j-1}^{+,n+1-j}(L^\varepsilon(\lambda-\mathcal{L}_0))\mathcal{L}_0)b_{1,1}^{\varepsilon,n},b_{1,2}^{\varepsilon,n}\rangle_\varepsilon+ o(1),$$
and 
\begin{align*}
&\langle(\lambda-G_{j-1}^{+,n+1-j}(L^\varepsilon(\lambda-\mathcal{L}_0))\mathcal{L}_0)b_{1,1}^{\varepsilon,n},b_{1,2}^{\varepsilon,n}\rangle_\varepsilon\\
=&\frac{\hat{\lambda}^2}{\log\frac1\varepsilon}				\int_{\mathbb{R}^2}\frac{\widehat{V}_{\varepsilon}(p)}{|p|^2}(-p_1p_2)\frac{\lambda+\frac{\nu^2}{2}|p|^2G_{j-1}^{+,n+1-j}(L^{\varepsilon}(\lambda+\frac{\nu^2}{2}|p|^2))}{(\lambda+\frac{\nu^2}{2}|p|^2(1+\frac{4}{\nu^4}G_n(L^{\varepsilon}(\lambda+\frac{\nu^2}{2}|p|^2)))^2}dp=0.
\end{align*}
Therefore \eqref{ineq:b_j_12} is proved.
\end{proof}
			
Recall that $\tilde{b}^{\varepsilon,n}$ is the solution to \eqref{truc-geratoreq1}. In order to estimate $\mathcal{V}^\varepsilon-(\lambda-\mathcal{L}^\varepsilon)b^{\varepsilon,n}$, i.e. \eqref{5.4}, we analyze the difference between $b^{\varepsilon,n}$ and $\tilde{b}^{\varepsilon,n}$. The next lemma shows that the cost by replacing $\mathcal{H}^\varepsilon_j$ with $\frac{4}{\nu^4}(-\mathcal{L}_0)G_j(L^\varepsilon(\lambda-\mathcal{L}_0))$ vanishes as $\varepsilon\to 0$.
\begin{lemma}
Let $n,j\in\mathbb{N}$ such that $1\leq j\leq n$, $0<\varepsilon<1$, $s,\tilde{s}\in H_j^\varepsilon$, and $p^\varepsilon,\tilde{p}^\varepsilon$ be defined as \begin{equation}
        \begin{aligned}
		&p^\varepsilon=(\lambda-\mathcal{L}_0[1+\frac{4}{\nu^4}G_{n+1-j}(L^\varepsilon(\lambda-\mathcal{L}_0))])^{-1}s,\\
		&\tilde{p}^\varepsilon=(\lambda-\mathcal{L}_0+\mathcal{H}^\varepsilon_{n+1-j})^{-1}\tilde{s}.
	\end{aligned}
\end{equation}
Then there exists $C=C(n,j)>0$ such that
\begin{align}
||(\lambda-\mathcal{L}_0)^\frac12(\tilde{p}^\varepsilon-p^\varepsilon)||_\varepsilon\leq C(||(\lambda-\mathcal{L}_0)^{-\frac12}(\tilde{s}-s)||_\varepsilon+\delta_{\varepsilon}||(\lambda-\mathcal{L}_0)^{-\frac12}s||_\varepsilon),
	\label{5.19}
\end{align}
where $\delta_{\varepsilon}\to 0$ uniformly in $n,j$ as $\varepsilon\to 0$.
Moreover, if $s=\mathcal{A}^\varepsilon_{+}g,\tilde{s}=\mathcal{A}^\varepsilon_{+}\tilde{g}$, 
\begin{align}
||(\lambda-\mathcal{L}_0)^\frac12(\tilde{p}^\varepsilon-p^\varepsilon)||_\varepsilon\leq C(||(-\mathcal{L}_0)^{\frac12}(\tilde{g}-g)||_\varepsilon+\delta_{\varepsilon}||(-\mathcal{L}_0)^{\frac12}g||_\varepsilon). \label{5.20}
\end{align}
\end{lemma}
			
\begin{proof}
	First note that \eqref{5.20} is a consequence of \eqref{5.19} and \eqref{ineq:A+}. It suffices to prove \eqref{5.19}. Let $A_j=\lambda-\mathcal{L}_0+\mathcal{H}^\varepsilon_{n+1-j},\;
B_j=\lambda-\mathcal{L}_0-\frac{4}{\nu^4}G_{n+1-j}(L^\varepsilon(\lambda-\mathcal{L}_0))\mathcal{L}_0$. Similar to the proof of Proposition \ref{prop:rep}, the left hand side of \eqref{5.19} equals
\begin{align*}
||(\lambda-\mathcal{L}_0)^\frac12(A_j^{-1}\tilde{s}-B_j^{-1}s)||_\varepsilon &\leq ||(\lambda-\mathcal{L}_0)^\frac12(A_j^{-1}-B_j^{-1})s||_\varepsilon+||(\lambda-\mathcal{L}_0)^\frac12A_j^{-1}(\tilde{s}-s)||_\varepsilon\\
&=:(I)+(II).\end{align*}
For $(I)$, since $\mathcal{H}^\varepsilon_{n+1-j}$ is positive semi-definite, $$(I)=||(\lambda-\mathcal{L}_0)^\frac12A_j^{-1}(B_j-A_j)B_j^{-1}s||_\varepsilon\leq ||(\lambda-\mathcal{L}_0)^{-\frac12}(B_j-A_j)B_j^{-1}s||_\varepsilon.$$
Furthermore, by \eqref{ineq:proprep} in Proposition \ref{prop:rep}, we have
\begin{align*}
&||(\lambda-\mathcal{L}_0)^{-\frac12}(B_j-A_j)B_j^{-1}s||_\varepsilon^2\\
=&-\langle[\mathcal{H}_{n+1-j}^{\varepsilon}+\frac{4}{\nu^4}G_{n+1-j}(L^{\varepsilon}(\lambda-\mathcal{L}_0))\mathcal{L}_0]B_j^{-1}s,(\lambda-\mathcal{L}_0)^{-1}(B_j-A_j)B_j^{-1}s\rangle_{\varepsilon}\\
\lesssim&\delta_{\varepsilon}||(-\mathcal{L}_0)^{\frac12}B_j^{-1}s||_{\varepsilon}||(\lambda-\mathcal{L}_0)^{-\frac12}(B_j-A_j)B_j^{-1}s||_{\varepsilon}\\
\leq&\frac{\delta_{\varepsilon}}{2}||(\lambda-\mathcal{L}_0)^{-\frac12}(B_j-A_j)B_j^{-1}s||_{\varepsilon}^2+\frac{\delta_{\varepsilon}}{2}||(\lambda-\mathcal{L}_0)^{\frac12}B_j^{-1}s||_{\varepsilon}^2.
\end{align*}
Again since $-G_{n+1-j}(L^\varepsilon(-\mathcal{L}_0))\mathcal{L}_0$ in $B_j$ is positive semi-definite, when $\varepsilon$ is small enough, we have
$$||(\lambda-\mathcal{L}_0)^{-\frac12}(B_j-A_j)B_j^{-1}s||_\varepsilon\lesssim\sqrt{\delta_{\varepsilon}}||(\lambda-\mathcal{L}_0)^{\frac12}B_j^{-1}s||_{\varepsilon}\leq\sqrt{\delta_{\varepsilon}}||(\lambda-\mathcal{L}_0)^{-\frac12}s||_{\varepsilon}.$$
Finally using the non-negativity of $\mathcal{H}^\varepsilon_{n+1-j}$, $(II)$ is bounded by $||(\lambda-\mathcal{L}_0)^{-\frac12}(\tilde{s}-s)||_\varepsilon.$
\end{proof}
			
By the above Lemma, we now can estimate the difference between $b^{\varepsilon,n}$ and $\tilde{b}^{\varepsilon,n}$.

\begin{proposition}\label{prop:b-btild}
Let $n\in\mathbb{N}$ be fixed. For $i=1,2$, 
$\lim_{\varepsilon\to 0}||(\lambda-\mathcal{L}_0)^\frac12 ((b^{\varepsilon,n})_i-(\tilde{b}^{\varepsilon,n})_i)||_\varepsilon^2=0.$
\end{proposition}
\begin{proof}
Since $b^{\varepsilon,n}=\sum_{j=1}^n b^{\varepsilon,n}_j,$ it suffices to prove that $\lim_{\varepsilon\to 0}||(\lambda-\mathcal{L}_0)^\frac12 (b_{j,i}^{\varepsilon,n}-\tilde{b}_{j,i}^{\varepsilon,n})||_\varepsilon^2=0,$ for each $j=1,\dots,n,\;i=1,2$.
Recalling the recursive structure of $b^{\varepsilon,n}_{j,i}$ in \eqref{iter-b}, by \eqref{5.20} and induction for $j-1,\dots,1$, we have $$||(\lambda-\mathcal{L}_0)^\frac12 (b_{j,i}^{\varepsilon,n}-\tilde{b}_{j,i}^{\varepsilon,n})||_\varepsilon\lesssim||(-\mathcal{L}_0)^\frac12 (b_{1,i}^{\varepsilon,n}-\tilde{b}_{1,i}^{\varepsilon,n})||_\varepsilon+\delta_{\varepsilon}\sum_{r=1}^{j-1}||(-\mathcal{L}_0)^\frac12b_{r,i}^{\varepsilon,n}||_\varepsilon.$$ By \eqref{ineq-b2}, since $n,\lambda$ are fixed, there exists a constant $C=C(n)$ such that $$\sum_{r=1}^{j-1}||(-\mathcal{L}_0)^\frac12b_{r,i}^{\varepsilon,n}||_\varepsilon\leq nC^n.$$
Take $s=\tilde{s}=\mathcal{V}_i^\varepsilon$ in \eqref{5.19}, we have
$$||(-\mathcal{L}_0)^\frac12 (b_{1,i}^{\varepsilon,n}-\tilde{b}_{1,i}^{\varepsilon,n})||_\varepsilon \lesssim \delta_\varepsilon||(\lambda-\mathcal{L}_0)^{-\frac12}\mathcal{V}_i^\varepsilon||_\varepsilon^2=\delta_\varepsilon\frac{\hat{\lambda}^2}{\log\frac1\varepsilon}\int_{\mathbb{R}^2}\frac{\hat{V}_\varepsilon(p)}{|p|^2}\frac{|p_2|^2}{\lambda+\frac{\nu^2}{2}|p|^2}dp\lesssim \delta_\varepsilon.$$ Consequently the statement follows.
\end{proof}

\subsection{Proof of Theorem \ref{thm5.2}}
			
Since for $i=1,2$, $n\in\mathbb{N}^{+}$, $(b^{\varepsilon,n})_i=\sum_{j=1}^{n}b_{j,i}^{\varepsilon,n}$, \eqref{5.3} holds immediately by \eqref{ineq-b1}.
For \eqref{5.4}, observe that
\begin{align*}
(\lambda-\mathcal{L}^\varepsilon) (b^{\varepsilon,n})_i=
\lambda (b^{\varepsilon,n})_i-\mathcal{L}_n^\varepsilon (b^{\varepsilon,n})_i-\mathcal{A}_+^\varepsilon b_{n,i}^{\varepsilon,n}
=\mathcal{V}_i^\varepsilon+(\lambda-\mathcal{L}_n^\varepsilon) ((b^{\varepsilon,n})_i-(\tilde{b}^{\varepsilon,n})_i)-\mathcal{A}_+^\varepsilon b_{n,i}^{\varepsilon,n}.
\end{align*}
Recall that $\mathcal{L}_n^\varepsilon=1_{\leq n}\mathcal{L}^\varepsilon 1_{\leq n}$ and $\tilde{b}^{\varepsilon,n}$ is defined as the solution to \eqref{truc-geratoreq1}.
Then the left hand side of \eqref{5.4} can be bounded by
\begin{align*}
||(\lambda-\mathcal{L}_0)^{-\frac12}(\mathcal{V}_i^\varepsilon-(\lambda-\mathcal{L}^\varepsilon)b_i^{\varepsilon,n})||_\varepsilon
&\leq ||(\lambda-\mathcal{L}_0)^{-\frac12}(\lambda-\mathcal{L}_n^\varepsilon) ((b^{\varepsilon,n})_i-(\tilde{b}^{\varepsilon,n})_i)||_\varepsilon+
||(\lambda-\mathcal{L}_0)^{-\frac12}\mathcal{A}_+^\varepsilon b_{n,i}^{\varepsilon,n}||_\varepsilon\\
&=: (I)+(II).
\end{align*}

For the first summand, $(\lambda-\mathcal{L}_n^\varepsilon)((b^{\varepsilon,n})_i-(\tilde{b}^{\varepsilon,n})_i)$ equals
$$(\lambda-\mathcal{L}_0) ((b^{\varepsilon,n})_i-(\tilde{b}^{\varepsilon,n})_i)
				-\sum_{j=1}^{n-1}\mathcal{A}_{+}^\varepsilon (b_{j,i}^{\varepsilon,n}-\tilde{b}_{j,i}^{\varepsilon,n})
				+\sum_{j=2}^{n}(\mathcal{A}_{+}^\varepsilon)^{*} (b_{j,i}^{\varepsilon,n}-\tilde{b}_{j,i}^{\varepsilon,n}).$$
Then by Lemma \ref{lem:estimate}, we have
\begin{align*}
&||(\lambda-\mathcal{L}_0)^{-\frac12}(\lambda-\mathcal{L}_n^\varepsilon) ((b^{\varepsilon,n})_i-(\tilde{b}^{\varepsilon,n})_i)||_\varepsilon^2\\
\lesssim&||(\lambda-\mathcal{L}_0)^{\frac12} ((b^{\varepsilon,n})_i-(\tilde{b}^{\varepsilon,n})_i)||_\varepsilon^2
+\sum_{j=1}^{n-1}||(\lambda-\mathcal{L}_0)^{-\frac12}\mathcal{A}_{+}^\varepsilon (b_{j,i}^{\varepsilon,n}-\tilde{b}_{j,i}^{\varepsilon,n})||_\varepsilon^2\\
&+\sum_{j=2}^{n}||(\lambda-\mathcal{L}_0)^{-\frac12}\mathcal{A}_{-}^\varepsilon (b_{j,i}^{\varepsilon,n}-\tilde{b}_{j,i}^{\varepsilon,n})||_\varepsilon^2\\
\lesssim&(1+\frac1\lambda)||(\lambda-\mathcal{L}_0)^{\frac12} ((b^{\varepsilon,n})_i-(\tilde{b}^{\varepsilon,n})_i)||_\varepsilon^2,
\end{align*}
which goes to 0 as $\varepsilon\to 0$ by Proposition \ref{prop:b-btild}.
				
For $(II)$, applying Lemma \ref{lem:estimate} again, we have $||(\lambda-\mathcal{L}_0)^{-\frac12}\mathcal{A}_{+}^{\varepsilon}b_{n,i}^{\varepsilon,n}||_{\varepsilon}\lesssim
||(\lambda-\mathcal{L}_0)^{\frac12}b_{n,i}^{\varepsilon,n}||_{\varepsilon}$,
By \eqref{ineq:b_j} in Proposition \ref{prop:lim-b}, $$\lim_{\varepsilon\to 0}||(\lambda-\mathcal{L}_0)^{\frac12}b_{n,i}^{\varepsilon,n}||_{\varepsilon}=\frac{\nu^2}{2}G_n^{+,1}(\pi\hat{\lambda}^2).$$
By \eqref{5.13}, $G_n^{+,1}(\pi\hat{\lambda}^2)\leq\left(\frac{4}{\nu^4}\right)^{n-1}
\frac{(\pi\hat{\lambda}^2)^{n-1}}{(n-1)!}\to 0,$ as $n\to\infty$. Therefore \eqref{5.4} holds.
				
Next we consider \eqref{5.5} and \eqref{5.6}. It is easy to see that \eqref{5.6} is a direct consequence of Wiener chaos decomposition and Proposition \ref{prop:lim-b}. It remains to prove \eqref{5.5}. 
For $i=1,2$, there holds that
$$\lim_{n\to\infty}\lim_{\varepsilon\to 0}||(\lambda-\mathcal{L}_0)^{\frac12}(b^{\varepsilon,n})_i||^2_\varepsilon
=\lim_{n\to\infty}\frac{\nu^2}{2}\sum_{j=1}^{n}G_j^{+,n+1-j}(\pi\hat{\lambda}^2).$$
Setting $S_n(x):=\sum_{j=0}^{n}G_j^{+,n+1-j}(x)$,
by Lemma \ref{lem:fun-G}, we can obtain that 
$$S_n(x)\lesssim e^x, \;|S_n^{'}(x)+S_n^{''}(x)|\lesssim e^x, \;S_{n+1}^{'}(x)=\frac{\frac{4}{\nu^4}S_n(x)}{(1+\frac{4}{\nu^4}G_{n+1}(x))^2}.$$
A similar argument to the proof of Theorem \ref{thm2.1} yields that there exists a continuous differentiable function $S$, such that $S_n, S_n^{'}$ converge uniformly to $S$ and $S^{'}$ respectively on every compact set. Then 
$$S^{'}(x)=\frac{\frac{4}{\nu^4}S(x)}{(1+\frac{4}{\nu^4}G(x))^2},\;S(0)=1,$$
where $G$ is the limit point of $G_n$, given by \eqref{def:G}. This implies that $S$ is given explicitly by $S(x)=\sqrt{\frac{8}{\nu^4}x+1}.$
Therefore the left hand side of \eqref{5.5} equals $$\lim_{n\to\infty}\sum_{j=1}^{n}G_j^{+,n+1-j}(\pi\hat{\lambda}^2)
=\lim_{n\to\infty}S_n(\pi\hat{\lambda}^2)-1=S(\pi\hat{\lambda}^2)-1=\frac{c(\nu)^2}{2\nu^2}.$$

\section{Convergence}\label{sec6}
	
This section aims to complete the proof of Theorem \ref{thm2.2}. 
We start by deriving the limit of the quadratic variation of the Dynkin martingales, which approximate $N^\varepsilon$ according to the estimates in the preceding Section. Subsequently, in Subsection \ref{sec6.1} we prove Theorem \ref{thm2.2}, while the proof for the limit of the variance of quadratic variation is provided in Subsection \ref{sec6.2}.
 
Recall that $b^\varepsilon=\begin{pmatrix}
b^\varepsilon_1\\b^\varepsilon_2
\end{pmatrix}$, $b_i^\varepsilon\in\mathcal{D}(\mathcal{L}^\varepsilon),\;i=1,2$, and $M_t(b^\varepsilon)$ is defined by \eqref{dynmart}. By \eqref{eq:mart} (in which $h$ is replaced by $b_i^\varepsilon$),  $$M_t(b^\varepsilon_i)=\nu\int_0^t Db^\varepsilon_i(\eta_{s}^\varepsilon)\cdot dB_s,$$ and the quadratic variation 
$$\langle M(b^\varepsilon_i),M(b^\varepsilon_j)\rangle_t=\nu^2\sum_{k=1}^{2}\int_0^t D_kb_i^\varepsilon(\eta_{s}^\varepsilon)D_kb_j^\varepsilon(\eta_{s}^\varepsilon)ds,$$ for $i,j=1,2$.
Therefore we have 
\begin{align}
\langle M(b_i^\varepsilon),B\rangle_t&=\nu\int_0^t Db^\varepsilon_i(\eta_{s}^\varepsilon)ds,\label{6.2_1}\\
\textbf{E}(\langle M(b_i^\varepsilon),M(b_j^\varepsilon)\rangle_t)&=\nu^2t \langle-\mathcal{L}_0 b_i^\varepsilon, b_j^\varepsilon\rangle_\varepsilon.\label{eq:quadv}
\end{align}
			
In the previous section, we introduced the functional $b^{\varepsilon,n}$ as an approximation to $\Tilde{b}^{\varepsilon,n}$ and provided corresponding estimates, where $b^{\varepsilon,n}=\sum_{j=1}^nb_j^{\varepsilon,n}$ is defined by \eqref{iter-b}. In the proof of Theorem \ref{thm2.2}, we can see that as $\varepsilon\to0$ and $n\to\infty$, the difference between $M_t(b^{\varepsilon,n})$ and $N_t^\varepsilon$ vanishes. Consequently, we proceed by establishing the limit of the quadratic variation of $M_t(b^{\varepsilon,n})$.
			
\begin{theorem}\label{thm6.1}
Recall that $c$ is the diffusion coefficient in Theorem \ref{thm2.1}. For $t\geq 0$, $i,j=1,2$, there holds that
\begin{align}
	&\lim_{\varepsilon\to 0}Var(\langle M((b^{\varepsilon,n})_i),M((b^{\varepsilon,n})_j)\rangle_t)=0,\label{6.1}\\
	&\lim_{n\to\infty}\lim_{\varepsilon\to 0}\mathbf{E}(\langle M((b^{\varepsilon,n})_i,M((b^{\varepsilon,n})_i)\rangle_t)=\frac{c(\nu)^2}{2}t,\label{6.2}\\
	&\lim_{n\to\infty}\lim_{\varepsilon\to 0}\mathbf{E}(\langle M((b^{\varepsilon,n})_1),M((b^{\varepsilon,n})_2)\rangle_t)=0.\label{6.3}
     \end{align}
\end{theorem}
\begin{proof}
	\eqref{6.1} will be proved in Section \ref{sec6.2}. We only prove \eqref{6.2} and \eqref{6.3} below. By \eqref{eq:quadv}, 
\begin{align*}
    \textbf{E}(\langle M((b^{\varepsilon,n})_i),M((b^{\varepsilon,n})_j)\rangle_t)&=
\nu^2 t(\langle(\lambda-\mathcal{L}_0) (b^{\varepsilon,n})_i, (b^{\varepsilon,n})_j\rangle_\varepsilon-\lambda\langle (b^{\varepsilon,n})_i,(b^{\varepsilon,n})_j\rangle_\varepsilon)\\
&\leq\nu^2t (\langle(\lambda-\mathcal{L}_0) (b^{\varepsilon,n})_i, (b^{\varepsilon,n})_j\rangle_\varepsilon+\lambda ||(b^{\varepsilon,n})_i||_\varepsilon||(b^{\varepsilon,n})_j||_\varepsilon).
\end{align*}
Then the conclusion follows by \eqref{ineq-b1} and Theorem \ref{thm5.2}.
\end{proof}
			
\subsection{The proof of Theorem \ref{thm2.2}}\label{sec6.1}
By Theorem \ref{thm:tightness} and \cite[Theorem 5.4 in Chapter 0]{Revuz.1999}, the distributions of $\{X_{\cdot}^\varepsilon\}$ are weakly relatively compact in $C([0,T],\mathbb{R}^2)$, for every $T>0$. Since $C([0,\infty),\mathbb{R}^2)$ is a Polish space with the topology of uniform convergence on compact subsets of $\mathbb{R}^{+}$, this sequence of distributions is weakly relatively compact in $C([0,\infty),\mathbb{R}^2)$ by \cite[Proposition 1.5 in Chapter XIII]{Revuz.1999}. Suppose $h:\Omega\times\Sigma\to C([0,\infty),\mathbb{R}^2)$ is the limit in distribution of a subsequence of $\{X_{\cdot}^\varepsilon\}$. According to L\'evy's characterization, it suffices to prove that $h^i$ is a martingale with respect to $\sigma\{h^i_s,s\leq t\}$ and has the quadratic variation $\langle h^i\rangle_t=\left(\frac{c(\nu)^2}{2}+\nu^2\right)t,\;i=1,2$, and $\langle h^1,h^2\rangle_t=0$. For simplicity, we still denote this subsequence as $\{X_{\cdot}^\varepsilon\}$.
                
Fix $i=1$ or $2$, we first prove that $h^i_t$ is a martingale with respect to its natural filtration $\mathcal{F}_t:=\sigma\{h^i_s,s\leq t\}$. That is, for $s<t$ and any continuous bounded functional $G$: $C([0,s])\to \mathbb{R}$, $\textbf{E}[(h_t^i-h_s^i)G(h^i\upharpoonright_{[0,s]})]=0$, where $h^i\upharpoonright_{[0,s]}$ denotes the restriction of $h^i$ to the interval $[0,s]$.

For $A>0$ define $f_A$ on $\mathbb{R}$ as follows:
\begin{align*}
f_A(x)=\left\{
\begin{aligned}
x,& \text{$\quad if\quad |x|\leq A$}, \\
A,& \text{$\quad if\quad x>A$},\\
\text{$-$}A,& \text{$\quad if\quad x<-A$}.\\
\end{aligned}\right.
\end{align*}
Then $f_A$ is bounded and continuous. Since $G$ is bounded, for $r=t$ or $s$, we have
\begin{align*}
    |\textbf{E}[h_r^iG(h^i\upharpoonright_{[0,s]})-X_r^{\varepsilon,i}G(X^{\varepsilon,i}\upharpoonright_{[0,s]})]|&\leq \textbf{E}|h_r^i-f_A(h_r^i)|+\textbf{E}|X_r^{\varepsilon,i}-f_A(X_r^{\varepsilon,i})|
 \\&+|\textbf{E}[f_A(h_r^i)G(h^i\upharpoonright_{[0,s]})-f_A(X_r^{\varepsilon,i})G(X^{\varepsilon,i}\upharpoonright_{[0,s]})]|.
 \end{align*}
By \eqref{eq:N}, $X^{\varepsilon,i}_r$ is uniformly integrable and $\textbf{E}|h^i_r|<\infty$. Since $X^{\varepsilon,i}$ converges in law to $h^i$, we obtain that
$$\textbf{E}[(h_t^i-h_s^i)G(h^i\upharpoonright_{[0,s]})]=\lim_{\varepsilon\to 0}\textbf{E}[(X_t^{\varepsilon,i}-X_s^{\varepsilon,i})G(X^{\varepsilon,i}\upharpoonright_{[0,s]})].$$
Recalling the definition of $M_t(b^{\varepsilon,n})$ \eqref{dynmart} (where $b^\varepsilon$ is replaced by $b^{\varepsilon,n}$), for $r=t$ or $s$, we can get that
\begin{align}
&\textbf{E}\left|X_r^{\varepsilon,i}-[M_r((b^{\varepsilon,n})_i)+\nu B_r^i]\right|=\textbf{E}|N_r^{\varepsilon,i}-M_r((b^{\varepsilon,n})_i)|\nonumber\\
\leq &\textbf{E}|(b^{\varepsilon,n})_i(\eta_r^\varepsilon)-(b^{\varepsilon,n})_i(\eta_0^\varepsilon)|+\textbf{E}\left|\int_0^r(\mathcal{V}^\varepsilon_i(\eta_s^\varepsilon)-(\lambda-\mathcal{L}^\varepsilon) (b^{\varepsilon,n})_i)ds\right|
+\lambda\textbf{E}\left|\int_0^r(b^{\varepsilon,n})_i(\eta_s^\varepsilon) ds\right|,\label{*}
\end{align}
where $B_r^i$ is the $i$-th component of $B_r$.
By \eqref{5.3} in Theorem \ref{thm5.2} and stationarity, when $\varepsilon\to 0, n\to\infty$, there holds that
$$\textbf{E}|(b^{\varepsilon,n})_i(\eta_r^\varepsilon)-(b^{\varepsilon,n})_i(\eta_0^\varepsilon)|\leq 2||(b^{\varepsilon,n})_i||_\varepsilon\to 0,\quad
\lambda\textbf{E}\left|\int_0^r(b^{\varepsilon,n})_i(\eta_s^\varepsilon) ds\right|\lesssim||(b^{\varepsilon,n})_i||_\varepsilon\to 0.$$
Therefore the first and last term of \eqref{*} vanish. Moreover, by H\"older's inequality and \eqref{eq:Ito}, the intermediate term can be bounded by
\begin{align*}
\left[\textbf{E}\left|\int_0^r(\mathcal{V}^\varepsilon_i(\eta_s^\varepsilon)-(\lambda-\mathcal{L}^\varepsilon) (b^{\varepsilon,n})_i(\eta_s^\varepsilon))ds\right|^2\right]^{\frac12}\lesssim(r+r^2\lambda)^\frac12||(\lambda-\mathcal{L}_0)^{-\frac12}(\mathcal{V}^\varepsilon_i-(\lambda-\mathcal{L}^\varepsilon) (b^{\varepsilon,n})_i)||_\varepsilon,
\end{align*}
which goes to 0 as $\varepsilon\to 0, n\to\infty$	by \eqref{5.4} in Theorem \ref{thm5.2}.
Thus we have $$\textbf{E}[(h_t^i-h_s^i)G(h^i\upharpoonright_{[0,s]})]=
\lim_{n\to\infty}\lim_{\varepsilon\to 0}\textbf{E}[(M_t((b^{\varepsilon,n})_i)-M_s((b^{\varepsilon,n})_i)+\nu B_t^i-\nu B_s^i)G(X^{\varepsilon,i}\upharpoonright_{[0,s]})].$$
It can be observed that $B_t$ is measurable with respect to $\mathcal{G}_t=\sigma\{\eta^\varepsilon_s,s\leq t\}$.Then $\sigma\{X^{\varepsilon,i}_s,s\leq t\}\subseteq \mathcal{G}_t$. Since $M_t((b^{\varepsilon,n})_i)$ is the Dynkin martingale w.r.t $\mathcal{G}_t$, $\{h^i_t\}_{t\geq 0}$ is a martingale with respect to $\{\mathcal{F}_t\}_{t\geq 0}$.
				
Next we consider the quadratic variation of $h_t$. By \eqref{6.2_1}, 
\begin{align}
&\langle M_t((b^{\varepsilon,n})_i+\nu B_t^i, M_t((b^{\varepsilon,n})_j)+\nu B_t^j\rangle_t\label{quar}\\
=&\langle M_t((b^{\varepsilon,n})_i, M_t((b^{\varepsilon,n})_j)\rangle_t+\nu^2\left(\int_0^t D_i (b^{\varepsilon,n})_j(\eta_s^\varepsilon)ds+\int_0^t D_j (b^{\varepsilon,n})_i(\eta_s^\varepsilon)ds\right)+\nu^2\delta_{i,j}t,\nonumber
\end{align}
where $\delta_{i,j}$ denotes the Kronecker delta.
Since $(b^\varepsilon)_i\in\oplus_{k=1}^\infty H_k^\varepsilon$, we have $$\textbf{E}\left(\int_0^t D(b^\varepsilon)_i(\eta_{s}^\varepsilon)ds\right)=t\mathbb{E}^\varepsilon(D(b^\varepsilon)_i)=0.$$ Then the expectation of \eqref{quar} equals
\begin{align}\label{expect}
\textbf{E}(\langle M_t((b^{\varepsilon,n})_i, M_t((b^{\varepsilon,n})_j)\rangle_t)+\nu^2\delta_{ij}t,
\end{align}
and the variance of \eqref{quar} is bounded by a constant times
\begin{align}
Var(\langle M_t((b^{\varepsilon,n})_i, M_t((b^{\varepsilon,n})_j)\rangle_t)+\textbf{E}\left(\int_0^t D_i (b^{\varepsilon,n})_j(\eta_s^\varepsilon)ds\right)^2+\textbf{E}\left(\int_0^t D_j(b^{\varepsilon,n})_i(\eta_s^\varepsilon)ds\right)^2.\label{var}
\end{align}
By It\^o trick\eqref{eq:Ito} and the stationarity of $\eta^\varepsilon$, for fixed $\lambda>0$, we have 
\begin{align*}
\textbf{E}\left(\int_0^t D_i (b^{\varepsilon,n})_j(\eta_s^\varepsilon)ds\right)^2&\lesssim (t+t^2\lambda)||(\lambda-\mathcal{L}_0)^{-\frac12}D_i (b^{\varepsilon,n})_j||_\varepsilon^2\\
&=(t+t^2\lambda)\langle-(\lambda-\mathcal{L}_0)^{-1}(D_i)^2(b^{\varepsilon,n})_j,(b^{\varepsilon,n})_j\rangle_\varepsilon\\
&\leq (t+t^2\lambda)||(b^{\varepsilon,n})_j||_\varepsilon^2,
\end{align*}
which vanishes as $\varepsilon\to 0$ by \eqref{ineq-b1}. Thus by \eqref{6.1}, the variance of the quadratic variance $X_t^\varepsilon$ vanishes as $\varepsilon\to0.$

Since $h$ and $X^\varepsilon$ are continuous martingales, and $X^\varepsilon$ converges in distribution to $h$, we have
$\textbf{E}(\langle h^i,h^j\rangle_t)=\textbf{E}(h_t^i h_t^j)=\lim_{\varepsilon\to 0}\textbf{E}(X_t^{\varepsilon,i}X_t^{\varepsilon,j}).$
Using Minkowski inequality and similarly to \eqref{*}, $\textbf{E}|N_r^{\varepsilon,i}-M_r((b^{\varepsilon,n})_i)|^2$ and $\textbf{E}|N_r^{\varepsilon,1}N_r^{\varepsilon,2}-M_r((b^{\varepsilon,n})_1)M_r((b^{\varepsilon,n})_2)|^2$ also vanish as $\varepsilon\to 0, n\to\infty$ for $i=1,2.$ Therefore by \eqref{expect}, for $i,j=1,2$, \begin{align}
\textbf{E}(\langle h^i,h^j\rangle_t)&=\lim_{n\to\infty}\lim_{\varepsilon\to 0}\textbf{E}(M_t((b^{\varepsilon,n})_i)+\nu B_t^i)(M_t((b^{\varepsilon,n})_j)+\nu B_t^j)\\
&=\lim_{n\to\infty}\lim_{\varepsilon\to 0}\textbf{E}(\langle M((b^{\varepsilon,n})_i),M((b^{\varepsilon,n})_j)\rangle_t)+\nu^2\delta_{i,j}t.
\end{align}
By \eqref{6.1} in Theorem \ref{thm6.1} and \eqref{var}, we can replace the quadratic variation of $h^i$ by its expectation. Then the proof has completed following Theorem \ref{thm6.1}.

\subsection{The variance of quadratic variation}\label{sec6.2}
This section is devoted to the proof of \eqref{6.1}. Fix $n\in\mathbb{N}^{+}, \lambda>0$, $i,j=1,2$. Let
\begin{align*}
&f^\varepsilon:=\sum_{k=1}^{2} (D_k(b^{\varepsilon,n})_iD_k(b^{\varepsilon,n})_j-\langle D_k(b^{\varepsilon,n})_i,D_k(b^{\varepsilon,n})_j\rangle_\varepsilon).
\end{align*}
Then $f^\varepsilon\in\oplus_{m=0}^{2n} H_m^\varepsilon$. By Wiener chaos decomposition, we write $f^\varepsilon$ as $f^\varepsilon=\sum_{m=0}^{2n}f_m^\varepsilon$. Note that $f_0^\varepsilon=0.$ By It\^o trick \eqref{eq:Ito}, it suffices to prove that for each $1\leq m\leq 2n$,
\begin{equation}
\lim_{\varepsilon\to 0}||(\lambda-\mathcal{L}_0)^{-\frac12}f_m^\varepsilon||_\varepsilon= 0.
\end{equation}
Let $\Pi_m$ be the projection from $L^2(\mathbb{P}^\varepsilon)$ onto $H_m^\varepsilon$. Since for $k=1,2$, $D_k(H_n^\varepsilon)\subseteq H_n^\varepsilon$, for all $n\in\mathbb{N}$, $f_m^\varepsilon$ can be written as 
\begin{align*}
f_m^\varepsilon=\sum_{k=1}^2\sum_{j_1,j_2=1}^n \Pi_m\left(D_kb^{\varepsilon,n}_{j_1,i}D_k b^{\varepsilon,n}_{j_2,j}\right)
=:\sum_{k=1}^2\sum_{j_1,j_2=1}^n f_{m,k,j_1,j_2}^\varepsilon.
\end{align*}
Thus we only need to prove that for $k=1,2$, $1\leq m\leq 2n$, $1\leq j_1,j_2\leq n$, 
\begin{equation}\label{6.6}
\lim_{\varepsilon\to 0}||(\lambda-\mathcal{L}_0)^{-\frac12}f_{m,k,j_1,j_2}^\varepsilon||_\varepsilon= 0.
\end{equation}

By the recursive definition of $b^{\varepsilon,n}$ in \eqref{iter-b}, we need to distinguish two cases: either both $j_i,i=1,2$ are larger than 1, or at least one of them equals 1. From now on, we denote $\mathcal{F}(\psi)$ for the Fourier transform of $\psi\in H_n^\varepsilon.$
			
\subsubsection{The case where $j_1$ or $j_2$=1}
First we prove \eqref{6.6} when $j_1=j_2=1$. By the definition of $b_1^{\varepsilon,n}$ in \eqref{iter-b} and Remark \ref{rem:V}, for $i=1$ or 2,
\begin{align}\label{b}
b_{1,i}^{\varepsilon,n}&=\frac{\hat{\lambda}}{\sqrt{\log\frac1\varepsilon}}[\lambda-\mathcal{L}_0(1+\frac{4}{\nu^4}G_n(L^\varepsilon(\lambda-\mathcal{L}_0)))]^{-1}\int_{\mathbb{R}^2}\delta_0(x)\pi_i^xdx=:\int_{\mathbb{R}^2}\gamma(x)\pi_i^xdx,
\end{align}
with its Fourier transform
\begin{align}\label{6.14}
\mathcal{F}(b_{1,i}^{\varepsilon,n})(p)=\frac{\hat{\lambda}}{\sqrt{\log\frac1\varepsilon}}\frac{(-\imath p_i^\perp)}{\lambda+\frac{\nu^2}{2}|p|^2(1+\frac{4}{\nu^4}G_n(L^\varepsilon(\lambda+\frac{\nu^2}{2}|p|^2)))}.
\end{align}
By the definition of $D_k$ \eqref{def:D}, we have
\begin{align}
D_kb_{1,i}^{\varepsilon,n}&=-\int_{\mathbb{R}^2}\partial_{x_k}\gamma(x)\pi_i^xdx,\label{6.16-1}\\
D_kb_{1,i}^{\varepsilon,n}D_kb_{1,j}^{\varepsilon,n}&=
\int_{\mathbb{R}^2\times \mathbb{R}^2}\partial_{x_k}\gamma(x)\partial_{y_k}\gamma(y)\pi_i^x\pi_j^ydxdy\nonumber\\
&=\int_{\mathbb{R}^2\times \mathbb{R}^2}\partial_{x_k}\gamma(x)\partial_{y_k}\gamma(y):\pi_i^x\pi_j^y:dxdy+\int_{\mathbb{R}^2\times \mathbb{R}^2}\partial_{x_k}\gamma(x)\partial_{y_k}\gamma(y)\mathbb{E}^\varepsilon(\pi_i^x\pi_j^y)dxdy.\label{6.16}
\end{align}
Since the second part of \eqref{6.16} belongs to $H_0^\varepsilon$, we only need to focus on the first part with its Fourier transform
\begin{align}\label{6.15}
-p_{1,k}p_{2,k}\mathcal{F}(b_{1,i}^{\varepsilon,n})(p_1)\mathcal{F}(b_{1,j}^{\varepsilon,n})(p_2).
\end{align}

Take \eqref{6.15} into the square of the left hand side of \eqref{6.6} with $m=2$. Then it equals
\begin{align}\label{eq:6.19}
&\int_{\mathbb{R}^2\times\mathbb{R}^2}\frac{\widehat{V}_\varepsilon(p_1)\widehat{V}_\varepsilon(p_2)}{|p_1|^2|p_2|^2}\frac{|\mathcal{F}(\sum_{k=1}^2\Pi_1(D_kb_{1,i}^{\varepsilon,n}D_kb_{1,j}^{\varepsilon,n}))(p_1,p_2)|^2}{\lambda+\frac{\nu^2}{2}|p_1+p_2|^2}dp_1dp_2\\
=&\int_{\mathbb{R}^2\times\mathbb{R}^2}\frac{\widehat{V}_\varepsilon(p_1)\widehat{V}_\varepsilon(p_2)}{|p_1|^2|p_2|^2}\frac{|p_1\cdot p_2\mathcal{F}(b_{1,i}^{\varepsilon,n})(p_1)\mathcal{F}(b_{1,j}^{\varepsilon,n})(p_2)|^2}{\lambda+\frac{\nu^2}{2}|p_1+p_2|^2}dp_1dp_2.
\end{align}
Since $G_n\geq 0,$ we make variable substitution, by \eqref{6.14}, it can be upper bounded by
\begin{align}
\frac{\hat{\lambda}^4}{(\log\frac1\varepsilon)^2}\int_{\mathbb{R}^2\times\mathbb{R}^2}\frac{\varepsilon^2\hat{V}(p_1)\hat{V}(p_2)|p_1|^2|p_2|^2}{(\varepsilon^2\lambda+\frac{\nu^2}{2}|p_1|^2)^2(\varepsilon^2\lambda+\frac{\nu^2}{2}|p_2|^2)^2(\varepsilon^2\lambda+\frac{\nu^2}{2}|p_1+p_2|^2)}dp_1dp_2.\label{eq:6.9}
\end{align}
For every $p_1\in\mathbb{R}^2$, by H\"older's inequality, there holds that
\begin{align*}
&\int_{\mathbb{R}^2}\frac{\varepsilon^2\hat{V}(p)|p|^2}{(\varepsilon^2\lambda+\frac{\nu^2}{2}|p|^2)^2(\varepsilon^2\lambda+\frac{\nu^2}{2}|p+p_1|^2)}dp\\
\leq&\varepsilon^2\int_{\mathbb{R}^2}\frac{2\hat{V}(p)}{\nu^2(\varepsilon^2\lambda+\frac{\nu^2}{2}|p|^2)(\varepsilon^2\lambda+\frac{\nu^2}{2}|p+p_1|^2)}dp\\
\lesssim&\frac{2}{\nu^2} \varepsilon^2\left(\int_{\{|p|\leq1\}}\frac{1}{(\varepsilon^2\lambda+\frac{\nu^2}{2}|p|^2)^2}dp\right)^{\frac12}\left(\int_{\{|p|\leq1\}}\frac{1}{(\varepsilon^2\lambda+\frac{\nu^2}{2}|p+p_1|^2)^2}dp\right)^{\frac12}\\
\leq& \frac{2}{\nu^4\lambda}.
\end{align*}
Therefore, taking this estimate into \eqref{eq:6.9}, we have \begin{align*}
\eqref{eq:6.9}\lesssim\frac{\hat{\lambda}^4}{\lambda(\log\frac1\varepsilon)^2}\int_{\mathbb{R}^2}\frac{\hat{V}(p_1)|p_1|^2}{(\varepsilon^2\lambda+\frac{\nu^2}{2}|p_1|^2)^2}dp_1
\lesssim\frac{\log(1+\frac{\nu^2}{2\varepsilon^2\lambda})}{\lambda(\log\frac1\varepsilon)^2}
\lesssim \frac{1}{\lambda\log\frac1\varepsilon},
\end{align*}
which proves \eqref{6.6}.

Then for the case where $\{j_1=1,j_2>1\}$ or $\{j_1>1,j_2=1\}$, it suffices to consider $\{j_1=1,j_2>1\}$. The definition of $b_{1,i}^{\varepsilon,n}$ is the same as \eqref{b}, and we write $b_{j_2,j}^{\varepsilon,n}$ as 
\begin{align}
b_{j_2,j}^{\varepsilon,n}
&=[\lambda-\mathcal{L}_0(1+\frac{4}{\nu^4}G_{n+1-j_2}(L^\varepsilon(\lambda-\mathcal{L}_0)))]^{-1}\mathcal{A}_{+}^\varepsilon b_{j_2-1,j}^{\varepsilon,n}\label{Ab}\\
&=:\sum_{l_1,\dots,l_{j_2}=1}^2\int_{\mathbb{R}^{2j_2}}\gamma_{l_{1:j_2}}(x_{1:j_2}):\pi_{l_1}^{x_1}\dots\pi_{l_{j_2}}^{x_{j_2}}:dx_{1:j_2}.
\end{align}
For simplicity of notations, by linearity, below we omit the sum $\sum_{l_1,\dots,l_{j_2}=1}^2$. Since 
\begin{align*}
D_kb_{j_2,j}^{\varepsilon,n}=-\sum_{t=1}^{j_2}\int_{\mathbb{R}^{2j_2}}\partial_{x_{t,k}}\gamma_{l_{1:j_2}}(x_{1:j_2}):\pi_{l_1}^{x_1}\dots\pi_{l_{j_2}}^{x_{j_2}}:dx_{1:j_2},
\end{align*}
by the property of Wick Product and \eqref{6.16-1}, we have
\begin{align*}
D_kb_{1,i}^{\varepsilon,n}&D_kb_{j_2,j}^{\varepsilon,n}
=\sum_{t=1}^{j_2}\int_{\mathbb{R}^{2(j_2+1)}}\partial_{x_{t,k}}\gamma_{l_{1:j_2}}(x_{1:j_2})\partial_{x_{1,k}}\gamma(x_{j_2+1}):\pi_{l_1}^{x_1}\dots\pi_{l_{j_2}}^{x_{j_2}}\pi_i^{x_{j_2+1}}:dx_{1:j_2+1}\\
&+\sum_{t,m=1}^{j_2}\int_{\mathbb{R}^{2(j_2+1)}}\partial_{x_{t,k}}\gamma_{l_{1:j_2}}(x_{1:j_2})\partial_{x_{1,k}}\gamma(x_{j_2+1})\mathbb{E}^\varepsilon[\pi_{l_m}^{x_m}\pi_i^{x_{j_2+1}}]:\displaystyle\prod_{a=1,a\neq m}^{j_2}\pi_{l_a}^{x_a}:dx_{1:j_2+1}.
\end{align*}
Then we can observe that
\begin{align*}
\mathcal{F}(\Pi_{j_2+1}D_kb_{1,i}^{\varepsilon,n}D_kb_{j_2,j}^{\varepsilon,n})(p_{1:j_2+1})&=-p_{j_2+1,k}\left(\sum_{t=1}^{j_2}p_{t,k}\right)\mathcal{F}(b_{j_2,j}^{\varepsilon,n})(p_{1:j_2})\mathcal{F}(b_{1,i}^{\varepsilon,n})(p_{j_2+1}),\\
\mathcal{F}(\Pi_{j_2-1}D_kb_{1,i}^{\varepsilon,n}D_kb_{j_2,j}^{\varepsilon,n})(p_{1:j_2-1})&=- j_2
\int_{\mathbb{R}^2}\frac{\widehat{V}_{\varepsilon}(p)}{|p|^2}p_{j_2+1,k}\left(\sum_{t=1}^{j_2}p_{t,k}\right)\mathcal{F}(b_{j_2,j}^{\varepsilon,n})(p_{1:j_2-1},p)\overline{\mathcal{F}(b_{1,i}^{\varepsilon,n})(p)}dp.
\end{align*}
Similar to \eqref{eq:6.19}, we take the above equations into the square of the left hand side of \eqref{6.6}, with $m=j_2+1$. Since $G_n,G_{n+1-j_2}\geq 0$, by \eqref{6.14} and \eqref{Ab}, it can be upper bounded by
\begin{align*}
&\frac{\hat{\lambda}^2}{\log\frac1\varepsilon}\int_{\mathbb{R}^{2(j_2+1)}}\prod_{k=1}^{j_2+1}\frac{\hat{V}_\varepsilon(p_k)}{|p_k|^2}\frac{|(\sum_{k=1}^{j_2}p_k)\cdot p_{j_2+1}|^2}{\lambda+\frac{\nu^2}{2}|\sum_{k=1}^{j_2+1}p_k|^2}\frac{|p_{j_2+1}|^2}{(\lambda+\frac{\nu^2}{2}|\sum_{k=1}^{j_2}p_k|^2)^2(\lambda+\frac{\nu^2}{2}|p_{j_2+1}|^2)^2}\times\\
&\left|\mathcal{F}(\mathcal{A}_{+}^\varepsilon b_{j_2-1,j}^{\varepsilon,n})(p_{1:j_2})\right|^2dp_{1:j_2+1}.
\end{align*}
By the action of $\mathcal{A}_{+}^\varepsilon$ \eqref{gen}, it has the bound 
\begin{align}
&\frac{\hat{\lambda}^4}{(\log\frac1\varepsilon)^2} \int_{\mathbb{R}^{2(j_2+1)}}\prod_{k=1}^{j_2+1}\frac{\hat{V}_\varepsilon(p_k)}{|p_k|^2}\frac{|(\sum_{k=1}^{j_2}p_k)\cdot p_{j_2+1}|^2}{\lambda+\frac{\nu^2}{2}|\sum_{k=1}^{j_2+1}p_k|^2}
\frac{|p_{j_2+1}|^2|p_{j_2}\times (\sum_{k=1}^{j_2-1}p_k)|^2}{(\lambda+\frac{\nu^2}{2}|\sum_{k=1}^{j_2}p_k|^2)^2(\lambda+\frac{\nu^2}{2}|p_{j_2+1}|^2)^2}\times\nonumber\\
&\left|\mathcal{F}(b_{j_2-1,j}^{\varepsilon,n})(p_{1:j_2-1})\right|^2dp_{1:j_2+1} \nonumber\\	
\lesssim& \frac{2\hat{\lambda}^4}{\nu^2(\log\frac1\varepsilon)^2} ||(-\mathcal{L}_0)^{\frac12}b_{j_2-1,j}^{\varepsilon,n}||_\varepsilon^2\times\nonumber\\
&\left(\sup_{x\in\mathbb{R}^2}\int_{\mathbb{R}^2\times \mathbb{R}^2}\frac{\hat{V}_\varepsilon(p)\hat{V}_\varepsilon(q)|x+p|^2|q|^2}{(\lambda+\frac{\nu^2}{2}|x+p|^2)^2(\lambda+\frac{\nu^2}{2}|x+p+q|^2)(\lambda+\frac{\nu^2}{2}|q|^2)^2}dpdq\right).\label{eq:6.11-1}
\end{align}
Similar to the discussion with \eqref{eq:6.9}, for the term in the supremum in \eqref{eq:6.11-1}, by H\"older's inequality, we have \begin{align*}
&\int_{\mathbb{R}^2\times \mathbb{R}^2}\frac{\hat{V}_\varepsilon(p)\hat{V}_\varepsilon(q)|x+p|^2|q|^2}{(\lambda+\frac{\nu^2}{2}|x+p|^2)^2(\lambda+\frac{\nu^2}{2}|x+p+q|^2)(\lambda+\frac{\nu^2}{2}|q|^2)^2}dpdq\\
=&\int_{\mathbb{R}^2}\frac{\hat{V}(p)|\varepsilon x+p|^2}{(\varepsilon^2\lambda+\frac{\nu^2}{2}|\varepsilon x+p|^2)^2}\left(\int_{\mathbb{R}^2}\frac{\varepsilon^2\hat{V}(q)|q|^2}{(\varepsilon^2\lambda+\frac{\nu^2}{2}|\varepsilon x+p+q|^2)(\varepsilon^2\lambda+\frac{\nu^2}{2}|q|^2)^2}dq\right)dp\\
\lesssim&\frac1\lambda\int_{\mathbb{R}^2}\frac{\hat{V}(p)|\varepsilon x+p|^2}{(\varepsilon^2\lambda+\frac{\nu^2}{2}|\varepsilon x+p|^2)^2}dp
\lesssim\frac{\log\frac{1}{\varepsilon}}{\lambda}.
\end{align*}

Next, for $m=j_2-1$, the square of the left hand side of \eqref{6.6} is bounded by 
\begin{align}
(j_2)^2\int_{\mathbb{R}^{2(j_2-1)}}&\prod_{k=1}^{j_2-1}\frac{\hat{V}_\varepsilon(p_k)}{|p_k|^2}\frac{1}{\lambda+\frac{\nu^2}{2}|\sum_{k=1}^{j_2-1}p_k|^2}\times\nonumber\\
&\left|\int_{\mathbb{R}^2}\left(\sum_{k=1}^{j_2-1}p_k+p\right)\cdot p\frac{\widehat{V}_\varepsilon(p)}{|p|^2}\mathcal{F}(b_{j_2,j}^{\varepsilon,n})(p_{1:j_2-1},p)\overline{\mathcal{F}(b_{1,i}^{\varepsilon,n}(p))}dp\right|^2dp_{1:j_2-1}. \label{6.12}
\end{align}
By H\"older's inequality and \eqref{6.14},\eqref{Ab} again, it has the bound
\begin{align}
&(j_2)^2\int_{\mathbb{R}^{2(j_2-1)}}\prod_{k=1}^{j_2-1}\frac{\hat{V}_\varepsilon(p_k)}{|p_k|^2}\frac{|\mathcal{F}(b_{j_2-1,j}^{\varepsilon,n})(p_{1:j_2-1})|^2}{\lambda+\frac{\nu^2}{2}|\sum_{k=1}^{j_2-1}p_k|^2}\times\nonumber\\
&\left(\int_{\mathbb{R}^2}\frac{(\widehat{V}_\varepsilon(p))^2}{|p|^4}
\frac{|(\sum_{k=1}^{j_2-1}p_k+p)\cdot p|^2|p\times (\sum_{k=1}^{j_2-1}p_k)|^2}{|\sum_{k=1}^{j_2-1}p_k+p|^4}\frac{|p|^2}{(\lambda+\frac{\nu^2}{2}|p|^2)^2}dp\right)dp_{1:j_2-1}\nonumber\\	
\lesssim&\frac{2\hat{\lambda}^4}{\nu^2(\log\frac1\varepsilon)^2} ||b_{j_2-1,j}^{\varepsilon,n}||_\varepsilon^2
\left(\int_{\mathbb{R}^2}\frac{(\widehat{V}_{\varepsilon}(p))^2}{(\lambda+\frac{\nu^2}{2}|p|^2)^2}dp\right)\label{eq:6.12}.
\end{align}
\eqref{eq:6.12} has the same estimate as before,
\begin{align*}
\int_{\mathbb{R}^2}\frac{(\widehat{V}_{\varepsilon}(p))^2}{(\lambda+\frac{\nu^2}{2}|p|^2)^2}dp=\varepsilon^2\int_{\mathbb{R}^2}\frac{(\hat{V}(p))^2}{(\varepsilon^2\lambda+\frac{\nu^2}{2}|p|^2)^2}dp\lesssim\frac{1}{\nu^2\lambda}.
\end{align*}
Then combining with \eqref{eq:6.11-1} and Proposition \ref{prop5.1}, we complete the proof of \eqref{6.6}.	

\subsubsection{The case where $j_1,j_2>1$}
In order to use some useful properties of Wick Product, we recall the following definition, which has been given in \cite[Definition 5.5]{Cannizzaro.2023}.
\begin{definition}
Let $n,m\in\mathbb{N}^{+}$, and $\mathcal{G}[n,m]$ be the set of collections $G$ of disjoint edges connecting a point in $V_1^n=\{(1,i),i=1,\dots,n\}$ to one in $V_2^m=\{(2,j),j=1,\dots,m\}$. Denote $V[G]$ as the set of points in $V_1^n\cup V_2^m$ which are not endpoints of any edge in $G$. Let $V_1[G]:=V[G]\cap V_1^n, V_2[G]:=V[G]\cap V_2^m$.
\end{definition}
\cite[Theorem 3.15]{Janson.1997} gives the product property of Wick Product. For $n,m\in\mathbb{N}^{+}$,$l_{1:n},t_{1:m}\in\{1,2\}$, there holds 
\begin{align}
:\prod_{v\in V_1^n}\pi^{x_v}_{l_{v(2)}}:
:\prod_{v\in V_2^m}\pi^{x_v}_{t_{v(2)}}:
=\sum_{G\in\mathcal{G}[n,m]}\prod_{(a,b)\in G} \mathbb{E}^\varepsilon[\pi_{l_{a(2)}}^{x_a}\pi_{t_{b(2)}}^{x_b}]:\prod_{v\in V_1[G]}\pi_{l_{v(2)}}^{x_v}\prod_{v\in V_2[G]}\pi_{t_{v(2)}}^{x_v}:,\label{wick}
\end{align}
where $v(2)$ is the second component of $v$.
Then by \eqref{wick}, for fixed $j_1,j_2>1$, $1\leq m\leq 2n$, $k=1,2$, $\Pi_m(D_kb_{j_1,i}^{\varepsilon,n}D_kb_{j_2,j}^{\varepsilon,n})\neq 0$ only if $m=j_1+j_2-2r$, for some $r\in\mathbb{N}.$ Its Fourier transform equals   
\begin{align*}
\mathcal{F}(\Pi_m(D_kb_{j_1,i}^{\varepsilon,n}D_kb_{j_2,j}^{\varepsilon,n}))(p_{1:m})=\sum_{G\in\mathcal{G}[j_1,j_2]:|V[G]|=m}&\int_{\mathbb{R}^{2r}}\prod_{k=1}^r\frac{\hat{V}_\varepsilon(q_k)}{|q_k|^2}\mathcal{F}(D_kb_{j_1,i}^{\varepsilon,n})(q_{1:r},p_{1:m_1})\times\\
&\mathcal{F}(D_kb_{j_2,j}^{\varepsilon,n})(-q_{1:r},p_{m_1+1:m})dq_{1:r}\\
=:\sum_{G\in\mathcal{G}[j_1,j_2]:|V[G]|=m}&\mathcal{F}(f_{G,k}^\varepsilon)(p_{1:m}),
\end{align*}
where $m_1:=|V_1[G]|.$

Thus for the proof of \eqref{6.6}, it suffices to show $||(\lambda-\mathcal{L}_0)^{-\frac12}f_G^\varepsilon||\to 0$ when $\varepsilon\to0$, for all $G\in\{\mathcal{G}[j_1,j_2]:|V[G]=m|\}$, where $f_G^\varepsilon:=\sum_{k=1}^2f_{G,k}^\varepsilon$. We apply H\"older's inequality to get
\begin{align*}
|\mathcal{F}(f_{G}^\varepsilon)(p_{1:m})|\leq& 
\int_{\mathbb{R}^{2r}}\prod_{k=1}^{r}\frac{\hat{V}_\varepsilon(q_k)}{|q_k|^2}\left|\left(\sum_{t=1}^{r}q_t+\sum_{l=1}^{m_1}p_l\right)\cdot\left(\sum_{t=1}^{r}q_t+\sum_{l=m_1+1}^{m}p_l\right)\right|\times\\
&|\mathcal{F}(b_{j_1,i}^{\varepsilon,n})(q_{1:r},p_{1:m_1})\mathcal{F}(b_{j_2,j}^{\varepsilon,n})(-q_{1:r},p_{m_1+1:m})|dq_{1:r}\\
\leq&\left(\int_{\mathbb{R}^{2r}}\prod_{k=1}^{r}\frac{\hat{V}_\varepsilon(q_k)}{|q_k|^2}\left|\sum_{t=1}^{r}q_t+\sum_{l=1}^{m_1}p_l\right|^2|\mathcal{F}(b_{j_1,i}^{\varepsilon,n})(q_{1:r},p_{1:m_1})|^2dq_{1:r}\right)^\frac12\times\\
&\left(\int_{\mathbb{R}^{2r}}\prod_{k=1}^{r}\frac{\hat{V}_\varepsilon(q_k)}{|q_k|^2}\left|\sum_{t=1}^{r}q_t+\sum_{l=m_1+1}^{m}p_l\right|^2|\mathcal{F}(b_{j_2,j}^{\varepsilon,n})(q_{1:r},p_{m_1+1:m})|^2dq_{1:r}\right)^\frac12.
\end{align*}
Recalling the recursive definition $b_{j_1,i}^{\varepsilon,n}=[\lambda-\mathcal{L}_0(1+G_{n+1-i}(L^\varepsilon(\lambda-\mathcal{L}_0)))]^{-1}\mathcal{A}_{+}^\varepsilon b_{j_1-1,i}^{\varepsilon,n}$, we
take the above inequality into the square of the left hand side of \eqref{6.6}. By $G_{n+1-i}\geq 0$, we have 
\begin{align*}
&||(\lambda-\mathcal{L}_0)^{-\frac12}f_G^\varepsilon||_\varepsilon^2\\
\leq&\frac{\hat{\lambda}^4}{(\log\frac1\varepsilon)^2}\int_{\mathbb{R}^{2m}}\prod_{k=1}^{m}\frac{\hat{V}_\varepsilon(p_k)}{|p_k|^2}\frac{1}{\lambda+\frac{\nu^2}{2}|\sum_{k=1}^{m}p_k|^2}\times\\
&\left(\int_{\mathbb{R}^{2r}}\prod_{k=1}^{r}\frac{\hat{V}_\varepsilon(q_k)}{|q_k|^2}\frac{|\zeta_1|^2|x_1+\zeta_1|^2}{(\lambda+\frac{\nu^2}{2}|x_1+\zeta_1|^2)^2}\mathcal{F}(b_{j_1-1,i}^{\varepsilon,n})(q_{1:r},p_{1:m_1}/\zeta_1)|^2dq_{1:r}\right)\times\\
&\left(\int_{\mathbb{R}^{2r}}\prod_{k=1}^{r}\frac{\hat{V}_\varepsilon(q_k)}{|q_k|^2}\frac{|\zeta_2|^2|x_2+\zeta_2|^2}{(\lambda+\frac{\nu^2}{2}|x_2+\zeta_2|^2)^2}\mathcal{F}(b_{j_2-1,j}^{\varepsilon,n})(q_{1:r},p_{m_1+1:m}/\zeta_2)|^2dq_{1:r}\right)dp_{1:m}\\
\leq&\frac{\hat{\lambda}^4}{(\log\frac1\varepsilon)^2} ||b_{j_1-1,i}^{\varepsilon,n}||_\varepsilon^2||b_{j_2-1,j}^{\varepsilon,n}||_\varepsilon^2 \times\\
&\left(\sup_{x_1,x_2,y\in\mathbb{R}^2}\int_{\mathbb{R}^2\times \mathbb{R}^2}\frac{\hat{V}_\varepsilon(\zeta_1)\hat{V}_\varepsilon(\zeta_2)|x_1+\zeta_1|^2|x_2+\zeta_2|^2}{(\lambda+\frac{\nu^2}{2}|x_1+\zeta_1|^2)^2(\lambda+\frac{\nu^2}{2}|y+\zeta_1+\zeta_2|^2)(\lambda+\frac{\nu^2}{2}|x_2+\zeta_2|^2)^2}d\zeta_1\zeta_2\right),
\end{align*}
where $(q_{1:r},p_{1:m_1}/\zeta_1)$ denotes 
that $(q_{1:r},p_{1:m_1})$ removes $\zeta_1$. Since $\mathcal{F}(b_{j_1-1,i}^{\varepsilon,n})$ is symmetric, we can choose $\zeta_1$ in the way that there exists $v\in V_1[G]$ such that $\zeta_1=p_v$. Similarly, $\zeta_2=p_u$ for some $u\in V_2[G]$. $x_1$ is defined as $\sum_{t=1}^{r}q_t+\sum_{l=1}^{m_1}p_l-\zeta_1$, 
$x_2:=\sum_{t=1}^{r}q_t+\sum_{l=m_1+1}^{m}p_l-\zeta_2$.

Similarly as in the previous discussion, for the term in the supremum in the right hand side of the above inequality, for every $x_1,x_2,y\in\mathbb{R}^2$, by H\"older's inequality, there holds that
\begin{align*}
&\int_{\mathbb{R}^2\times \mathbb{R}^2}\frac{\hat{V}_\varepsilon(\zeta_1)\hat{V}_\varepsilon(\zeta_2)|x_1+\zeta_1|^2|x_2+\zeta_2|^2}{(\lambda+\frac{\nu^2}{2}|x_1+\zeta_1|^2)^2(\lambda+\frac{\nu^2}{2}|y+\zeta_1+\zeta_2|^2)(\lambda+\frac{\nu^2}{2}|x_2+\zeta_2|^2)^2}d\zeta_1\zeta_2\\
\lesssim& \frac{1}{\nu^2\lambda}\int_{\mathbb{R}^2}\frac{\hat{V}(\zeta_1)|\varepsilon x_1+\zeta_1|^2}{(\varepsilon^2\lambda+\frac{\nu^2}{2}|\varepsilon x_1+\zeta_1|^2)^2}d\zeta_1\lesssim\frac{\log\frac1\varepsilon}{\lambda}.
\end{align*}
Then the proof is concluded by \eqref{ineq-b1}.

\appendix
\section*{Appendix}\label{app}
\setcounter{equation}{0}
\renewcommand\theequation{A.\arabic{equation}}
\setcounter{theorem}{0}
\renewcommand\thetheorem{A.\arabic{theorem}}
\setcounter{section}{1}
\renewcommand\thesection{A.\arabic{section}}

The appendix is devoted to give an estimate for the integral
\begin{align}\label{A1}
\frac{\hat{\lambda}^2}{\log\frac{1}{\varepsilon}}\frac{2}{\nu^2}\int_{\mathbb{R}^2}\hat{V}(\varepsilon q)(\sin^2\theta)
\frac{\lambda+\frac{\nu^2}{2}|\sum_{i=1}^{n}x_i+q|^2H^{+}(L^{\varepsilon}(\lambda+\frac{\nu^2}{2}|\sum_{i=1}^{n}x_i+q|^2))}{(\lambda+\frac{\nu^2}{2}|\sum_{i=1}^{n}x_i+q|^2H(L^{\varepsilon}(\lambda+\frac{\nu^2}{2}|\sum_{i=1}^{n}x_i+q|^2)))^2}dq
\end{align} in Lemma \ref{lem:rep}, where $\theta$ is the angle between $q$ and $\sum_{i=1}^{n}x_i$.
For simplicity, we denote that
\begin{equation}\label{4.10}
\lambda_{\varepsilon}=\varepsilon^2\lambda,\;\Gamma(q)=\frac{\nu^2}{2}\left|\varepsilon\sum_{i=1}^{n}x_i+q\right|^2,\;\Gamma_1(q)=\lambda_{\varepsilon}+\Gamma(q),\;\Gamma_2(q)=\frac{\nu^2}{2}\left(|q|^2+\left|\varepsilon\sum_{i=1}^{n}x_i\right|^2\right)+\lambda_{\varepsilon}.
\end{equation}
By variable substitution, \eqref{A1} equals
\begin{equation}
\frac{\hat{\lambda}^2}{\log\frac{1}{\varepsilon}}\frac{2}{\nu^2}\int_{\mathbb{R}^2}\hat{V}(q)(\sin^2\theta)
\frac{\lambda_{\varepsilon}+\Gamma(q)H^{+}(L^{\varepsilon}(\frac{1}{\varepsilon^2}\Gamma_1(q))}{(\lambda_{\varepsilon}+\Gamma(q)H(L^{\varepsilon}(\frac{1}{\varepsilon^2}\Gamma_1(q))))^2}dq.
\end{equation}
The main result of appendix is the following lemma.
\begin{lemma}\label{lem7.1}
Fix $\lambda>0, n\in\mathbb{N}^{+}$. Let $x_i\in\mathbb{R}^2,i=1,\dots,n$. If $\sum_{i=1}^nx_i\neq 0$, 
\begin{align}\label{A4}
\left|\int_{\mathbb{R}^2}\hat{V}(q)(\sin^2\theta)
\frac{\lambda_{\varepsilon}+\Gamma(q)H^{+}(L^{\varepsilon}(\frac{1}{\varepsilon^2}\Gamma_1(q))}{(\lambda_{\varepsilon}+\Gamma(q)H(L^{\varepsilon}(\frac{1}{\varepsilon^2}\Gamma_1(q))))^2}dq-
\frac{\pi}{\nu^2}\int_{\Gamma_1(0)}^{1}\frac{H^{+}(L^{\varepsilon}(\frac{\rho}{\varepsilon^2}))}{\rho(\rho+1)(H(L^{\varepsilon}(\frac{\rho}{\varepsilon^2})))^2}
d\rho\right|
\end{align}
is bounded uniformly in $\varepsilon,x_{1:n}$, where $\theta$ is the angle between $\sum_{i=1}^nx_i$ and $q$.
If $\sum_{i=1}^nx_i=0$, 
\begin{align}\label{A5}
\left|\int_{\mathbb{R}^2}\hat{V}(q)
\frac{\lambda_{\varepsilon}+\frac{\nu^2}{2}|q|^2H^{+}(L^{\varepsilon}(\frac{1}{\varepsilon^2}(\lambda_{\varepsilon}+\frac{\nu^2}{2}|q|^2))}{(\lambda_{\varepsilon}+\frac{\nu^2}{2}|q|^2H(L^{\varepsilon}(\frac{1}{\varepsilon^2}(\lambda_{\varepsilon}+\frac{\nu^2}{2}|q|^2))))^2}dq-
\frac{2\pi}{\nu^2}\int_{\varepsilon^2\lambda}^{1}\frac{H^{+}(L^{\varepsilon}(\frac{\rho}{\varepsilon^2}))}{\rho(\rho+1)(H(L^{\varepsilon}(\frac{\rho}{\varepsilon^2})))^2}d\rho\right|
\end{align}
is bounded uniformly in $\varepsilon,x_{1:n}$.
\end{lemma}
\begin{remark}\label{rem:A1}
	In the proof of this lemma, the argument of $H, H^{+}$ is always $L^{\varepsilon}(\frac{1}{\varepsilon^2}\Gamma_1(q)),L^{\varepsilon}(\frac{1}{\varepsilon^2}\Gamma_2(q))$, where $\Gamma_1,\Gamma_2$ is defined in \eqref{4.10}. When $n$ is fixed, we have $\Gamma_1(y),\Gamma_2(y) \geq\lambda_{\varepsilon}$.
Thus $L^{\varepsilon}(\frac{1}{\varepsilon^2}\Gamma_i(y)),i=1,2$ are bounded, the condition $(i)$ in Lemma \ref{lem:rep} is always satisfied.
\end{remark}
\begin{proof}
First we assume $\sum_{i=1}^nx_i\neq 0$.
For \eqref{A4}, we add and subtract some terms to get
\begin{align}
&\left|\int_{\mathbb{R}^2}\hat{V}(q)(\sin^2\theta)
\frac{\lambda_{\varepsilon}+\Gamma(q)H^{+}(L^{\varepsilon}(\frac{1}{\varepsilon^2}\Gamma_1(q)))}{(\lambda_{\varepsilon}+\Gamma(q)H(L^{\varepsilon}(\frac{1}{\varepsilon^2}\Gamma_1(q))))^2}dq-
\int_{\mathbb{R}^2}\hat{V}(q)(\sin^2\theta)
\frac{H^{+}(L^{\varepsilon}(\frac{1}{\varepsilon^2}\Gamma_1(q)))}{\Gamma_1(q)(H(L^{\varepsilon}(\frac{1}{\varepsilon^2}\Gamma_1(q))))^2}dq\right|, \label{7.4}\\
&\left|\int_{\mathbb{R}^2}\hat{V}(q)(\sin^2\theta)
\frac{H^{+}(L^{\varepsilon}(\frac{1}{\varepsilon^2}\Gamma_1(q)))}{\Gamma_1(q)(H(L^{\varepsilon}(\frac{1}{\varepsilon^2}\Gamma_1(q))))^2}dq-
\int_{\mathbb{R}^2}\hat{V}(q)(\sin^2\theta)
\frac{H^{+}(L^{\varepsilon}(\frac{1}{\varepsilon^2}\Gamma_2(q)))}{\Gamma_2(q)(H(L^{\varepsilon}(\frac{1}{\varepsilon^2}\Gamma_2(q))))^2}dq\right|,\label{7.5}\\
&\left|\int_{\mathbb{R}^2}\hat{V}(q)(\sin^2\theta)
\frac{H^{+}(L^{\varepsilon}(\frac{1}{\varepsilon^2}\Gamma_2(q)))}{\Gamma_2(q)(H(L^{\varepsilon}(\frac{1}{\varepsilon^2}\Gamma_2(q))))^2}dq-
\frac{\pi}{\nu^2}\int_{\Gamma_1(0)}^{1}\frac{H^{+}(L^{\varepsilon}(\frac{\rho}{\varepsilon^2}))}{\rho(\rho+1)(H(L^{\varepsilon}(\frac{\rho}{\varepsilon^2})))^2}
d\rho\right|.\label{7.6}
\end{align}
In the following it suffices to prove that \eqref{7.4}-\eqref{7.6} are bounded uniformly in $\varepsilon,x_{1:n}$.
				
Let us start with \eqref{7.4}. In order to lighten the notation, we omit the arguments of the functions involved. By Assumption \ref{rho}, we have
\begin{align}
&\left|\int_{\mathbb{R}^2}\hat{V}(q)(\sin^2\theta)
\frac{\lambda_{\varepsilon}+\Gamma(q)H^{+}(L^{\varepsilon}(\frac{1}{\varepsilon^2}\Gamma_1(q))}{(\lambda_{\varepsilon}+\Gamma(q)H(L^{\varepsilon}(\frac{1}{\varepsilon^2}\Gamma_1(q))))^2}dq-
\int_{\mathbb{R}^2}\hat{V}(q)(\sin^2\theta)
\frac{H^{+}(L^{\varepsilon}(\frac{1}{\varepsilon^2}\Gamma_1(q)))}{\Gamma_1(q)(H(L^{\varepsilon}(\frac{1}{\varepsilon^2}\Gamma_1(q))))^2}dq\right|\nonumber\\
\lesssim&\int_{\{|q|\leq 1\}}\left|\frac{\lambda_{\varepsilon}+\Gamma H^{+}}{(\lambda_{\varepsilon}+\Gamma H)^2}-\frac{\lambda_{\varepsilon}+\Gamma H^{+}}{(\Gamma_1 H)^2}\right|dq-
\int_{\{|q|\leq 1\}}\left|\frac{\lambda_{\varepsilon}+\Gamma H^{+}}{(\Gamma_1H)^2}-\frac{H^{+}}{\Gamma_1H^2}\right|dq\nonumber\\
=&\int_{\{|q|\leq 1\}}\frac{(\lambda_{\varepsilon}+\Gamma H^{+})|(\Gamma_1 H)^2-(\lambda_{\varepsilon}+\Gamma H)^2|}{(\lambda_{\varepsilon}+\Gamma H)^2(\Gamma_1 H)^2}-\frac{|\lambda_{\varepsilon}+\Gamma H^{+}-\Gamma_1H^{+}|}{(\Gamma_1H)^2}dq.\label{numerator}
\end{align}
By Remark \ref{rem:A1} and $\Gamma_1=\Gamma+\lambda_\varepsilon$, the numerators in the integral of \eqref{numerator} have estimates respectively:
\begin{align*}
0\leq &(\Gamma_1 H)^2-(\lambda_{\varepsilon}+\Gamma H)^2=(H-1)\lambda_{\varepsilon}[\lambda_{\varepsilon}(H+1)+2\Gamma H]\lesssim \lambda_{\varepsilon}\Gamma_1,\\
&|\lambda_{\varepsilon}+\Gamma H^{+}-\Gamma_1H^{+}|=|\lambda_{\varepsilon}(1-H^{+})|\lesssim \lambda_{\varepsilon}.
\end{align*}
Therefore we have
\begin{align*}
\eqref{7.4}\lesssim\int_{|y|\leq 1}\frac{\lambda_{\varepsilon}}{(\Gamma_1)^2}dy
&=\int_{|y|\leq 1}\frac{\lambda_{\varepsilon}}{(\lambda_{\varepsilon}+\frac{\nu^2}{2}|\varepsilon\sum_{i=1}^{n}x_i+y|^2)^2}dy\\
&\leq \int_{|t|\leq 1}\frac{\lambda_{\varepsilon}}{(\lambda_{\varepsilon}+\frac{\nu^2}{2}|t|^2)^2}dt\lesssim 1.
\end{align*}
				
For \eqref{7.5}. we split the domain of the integral into $\{|q+\varepsilon\sum_{i=1}^{n}x_i|<|\varepsilon\sum_{i=1}^{n}x_i|\}$ and $\{|q+\varepsilon\sum_{i=1}^{n}x_i|\geq|\varepsilon\sum_{i=1}^{n}x_i|\}$. Let $p:=\varepsilon\sum_{i=1}^{n}x_i$. Then $\theta$ is the angle between $p$ and $q$ and it holds that $\sin^2\theta\leq\frac{|p+q|^2}{|p|^2}$.
For the first domain $\{q:|p+q|<|p|\}\subseteq\{q:|q|<2|p|\}$, there exists a constant $C$ such that the integral can be bounded by
\begin{align*}
C\int_{|p+q|<|p|}\left(\frac{\sin^2\theta}{\Gamma_1(q)}+\frac{1}{\Gamma_2(q)}\right)dq
&\lesssim \int_{|p+q|<|p|}\frac{|p+q|^2}{|p|^2(\lambda_{\varepsilon}+\frac{\nu^2}{2}|p+q|^2)}+\frac{1}{\frac{\nu^2}{2}(|q|^2+|p|^2)+\lambda_{\varepsilon}}dq\\
&\lesssim\frac{2}{\nu^2|p|^2}\int_{|q|<2|p|}dq\lesssim 1.
\end{align*}
For another domain $\{q:|p+q|\geq|p|\}$, we still add and subtract one term, and it has the bound
\begin{align} 	                                 &\int_{|p+q|\geq|p|}\left|\frac{H^{+}(L^{\varepsilon}(\frac{1}{\varepsilon^2}\Gamma_1(q)))}{\Gamma_1(q)(H(L^{\varepsilon}(\frac{1}{\varepsilon^2}\Gamma_1(q))))^2}-\frac{H^{+}(L^{\varepsilon}(\frac{1}{\varepsilon^2}\Gamma_1(q)))}{\Gamma_2(q)(H(L^{\varepsilon}(\frac{1}{\varepsilon^2}\Gamma_2(q))))^2}\right| \label{7.7}\\
&+\left|\frac{H^{+}(L^{\varepsilon}(\frac{1}{\varepsilon^2}\Gamma_1(q)))}{\Gamma_2(q)(H(L^{\varepsilon}(\frac{1}{\varepsilon^2}\Gamma_2(q))))^2}-\frac{H^{+}(L^{\varepsilon}(\frac{1}{\varepsilon^2}\Gamma_2(q)))}{\Gamma_2(q)(H(L^{\varepsilon}(\frac{1}{\varepsilon^2}\Gamma_2(q))))^2}\right|dq. \label{7.8}
\end{align}
For \eqref{7.7}, we set function $G(z):=z(H(L^{\varepsilon}(\frac{1}{\varepsilon^2}z)))^2$. Then
$$G^{'}(z)=(H(L^{\varepsilon}(\frac{1}{\varepsilon^2}z)))^2-2H^{'}(L^{\varepsilon}(\frac{1}{\varepsilon^2}z))\frac{\pi\hat{\lambda}^2}{\log\frac{1}{\varepsilon^2}}\frac{1}{z+1},$$ and $|G'(z)|\lesssim 1$. Thus $|G(\Gamma_2)-G(\Gamma_1)|\lesssim |\Gamma_2-\Gamma_1|$, 
and $$\eqref{7.7}\lesssim\int_{|p+q|\geq|p|}\frac{|\Gamma_2-\Gamma_1|}{\Gamma_1\Gamma_2}dq.$$
For \eqref{7.8}, since 
$$\left|\frac{d}{dz}H^{+}(L^{\varepsilon}(\frac{1}{\varepsilon^2}z))\right|=(H^{+})^{'}\left(L^{\varepsilon}(\frac{1}{\varepsilon^2}z)\right)\frac{\pi\hat{\lambda}^2}{\log\frac{1}{\varepsilon^2}}\frac{1}{z(1+z)}\lesssim \frac1z,$$
$\Gamma_1\leq 2\Gamma_2$, and $\Gamma_1\wedge\Gamma_2\geq \frac{\Gamma_1}{2}$, we have $$\eqref{7.8}\lesssim\int_{|p+q|\geq|p|}\frac{|\Gamma_2-\Gamma_1|}{\Gamma_2(\Gamma_1\wedge\Gamma_2)}dq\lesssim \int_{|p+q|\geq|p|}\frac{|\Gamma_2-\Gamma_1|}{\Gamma_1\Gamma_2}dq.$$
In summary, the integral in the second domain is bounded by
\begin{align}
C\int_{\{|q+p|\geq |p|\}}\frac{|\Gamma_2-\Gamma_1|}{\Gamma_1\Gamma_2}dq
&\lesssim\int_{\{|q+p|\geq |p|\}}\frac{|pq|}{|p+q|^2(|p|^2+|q|^2)}dq\\
&=\left(\int_{\{|q+p|\geq |p|\}\cap\{|q|\geq 2|p|\}}+\int_{\{|q+p|\geq |p|\}\cap\{|q|<2|p|\}}\right)\frac{|pq|}{|p+q|^2(|p|^2+|q|^2)}dq.\label{7.10}
\end{align}
Since $\{|q|\geq 2|p|\}\subseteq\{|p+q|\geq\frac{|q|}{2}\}$, 
\begin{align*}\eqref{7.10}
&\lesssim|p|\int_{\{|q|\geq 2|p|\}}\frac{|q|}{|q|^4}dq+\int_{\{|q|< 2|p|\}}\frac{|p||q|}{|p|^4}dq\\
&=|p|\int_{2|p|}^{1}\frac{1}{r^2}dr+\frac{1}{|p|^3}\int_{0}^{2|p|}r^2dr\lesssim 1.
\end{align*}
Finally we consider \eqref{7.6}.
We split the first integral in \eqref{7.6} over two regions $\{q:\frac{2}{\nu^2}(1-\lambda_{\varepsilon})-\varepsilon^2|p|^2\leq|q|^2\leq 1\}$, and $\{q:|q|^2< 1-\lambda_{\varepsilon}-\varepsilon^2|p|^2\}.$ In the first domain, there exists a constant $C$ such that it can be bounded by 
$$C\int_{\{q:1\geq|q|^2\geq \frac{2}{\nu^2}(1-\lambda_{\varepsilon})-\varepsilon^2|p|^2\}}\frac{\hat{V}(q)}{\lambda_{\varepsilon}+\frac{\nu^2}{2}(\varepsilon^2|p|^2+|q|^2)}dq\lesssim 1.$$ For the second domain, there holds that
\begin{align}
&\int_{\{|q|^2<\frac{2}{\nu^2}(1-\lambda_{\varepsilon})-\varepsilon^2|p|^2\}}\hat{V}(q)(\sin^2\theta)
\frac{H^{+}(L^{\varepsilon}(\frac{1}{\varepsilon^2}\Gamma_2(q)))}{\Gamma_2(q)(H(L^{\varepsilon}(\frac{1}{\varepsilon^2}\Gamma_2(q))))^2}dq\\
=&\int_{\{|q|^2< \frac{2}{\nu^2}(1-\lambda_{\varepsilon})-\varepsilon^2|p|^2\}}(\hat{V}(q)-\hat{V}(0))(\sin^2\theta)\frac{H^{+}(L^{\varepsilon}(\frac{1}{\varepsilon^2}\Gamma_2(q)))}{\Gamma_2(q)(H(L^{\varepsilon}(\frac{1}{\varepsilon^2}\Gamma_2(q))))^2}dq  \label{7.12}\\
+&\int_{\{|q|^2< \frac{2}{\nu^2}(1-\lambda_{\varepsilon})-\varepsilon^2|p|^2\}}\hat{V}(0)(\sin^2\theta)
\frac{H^{+}(L^{\varepsilon}(\frac{1}{\varepsilon^2}\Gamma_2(q)))}{\Gamma_2(q)(H(L^{\varepsilon}(\frac{1}{\varepsilon^2}\Gamma_2(q))))^2}dq.\label{7.13}
\end{align}
Since $\hat{V}$ is a smooth and radial function, there exists a constant $C$ such that for all $|q|<1$, $|\hat{V}(q)-\hat{V}(0)|\leq C|q|^2$. Therefore 
$$\eqref{7.12}\leq \int_{\{|q|^2< 1-\lambda_{\varepsilon}-\varepsilon^2|p|^2\}}\frac{2|\hat{V}(q)-\hat{V}(0)|\sin^2(\theta)H^{+}}{\nu^2|q|^2H^2}dq
\lesssim\frac{2}{\nu^2}\int_{|q|\leq 1}dq\lesssim 1.$$
For \eqref{7.13}, since $\hat{V}(0)=\int V(q)dq=1$, we have
\begin{align*}
&\int_{\{|q|^2< \frac{2}{\nu^2}(1-\lambda_{\varepsilon})-\varepsilon^2|p|^2\}}\hat{V}(0)(\sin^2\theta)
\frac{H^{+}(L^{\varepsilon}(\frac{1}{\varepsilon^2}\Gamma_2(q)))}{\Gamma_2(q)(H(L^{\varepsilon}(\frac{1}{\varepsilon^2}\Gamma_2(q))))^2}dq\\
=&\int_{0}^{2\pi}\sin^2\theta d\theta\int_{0}^{\sqrt{\frac{2}{\nu^2}(1-\lambda_{\varepsilon})-\varepsilon^2|p|^2}}\frac{rH^{+}(L^{\varepsilon}(\frac{1}{\varepsilon^2}(\frac{\nu^2}{2}(r^2+\varepsilon^2|p|^2)+\lambda_{\varepsilon})))}{(\frac{\nu^2}{2}(r^2+\varepsilon^2|p|^2)+\lambda_{\varepsilon})(H(L^{\varepsilon}(\frac{1}{\varepsilon^2}(\frac{\nu^2}{2}(r^2+\varepsilon^2|p|^2)+\lambda_{\varepsilon}))))^2}dr\\
=&\frac{\pi}{\nu^2}\int_{\frac{\nu^2}{2}\varepsilon^2|p|^2+\lambda_{\varepsilon}}^{1}\frac{H^{+}(L^{\varepsilon}(\frac{\rho}{\varepsilon^2}))}{\rho(H(L^{\varepsilon}(\frac{\rho}{\varepsilon^2})))^2}d\rho,
\end{align*}
and $$\left|\int_{\Gamma_1(0)}^{1}\frac{H^{+}(L^{\varepsilon}(\frac{\rho}{\varepsilon^2}))}{\rho(H(L^{\varepsilon}(\frac{\rho}{\varepsilon^2})))^2}d\rho-\int_{\Gamma_1(0)}^{1}\frac{H^{+}(L^{\varepsilon}(\frac{\rho}{\varepsilon^2}))}{\rho(\rho+1)(H(L^{\varepsilon}(\frac{\rho}{\varepsilon^2})))^2}d\rho\right|\lesssim \int_{0}^1\frac{1}{\rho+1}d\rho\leq 1.$$
Thus the first part of the proof is concluded.	
				
If $\sum_{i=1}^nx_i=0$, similarly we can get the bound for \eqref{7.4},\eqref{7.5}. For \eqref{7.6}, estimates are the same except for \eqref{7.13},
\begin{align*}
&\int_{\{|q|^2< \frac{2}{\nu^2}(1-\lambda_{\varepsilon})\}}\hat{V}(0)
\frac{H^{+}(L^{\varepsilon}(\frac{1}{\varepsilon^2}\Gamma_2(q)))}{\Gamma_2(q)(H(L^{\varepsilon}(\frac{1}{\varepsilon^2}\Gamma_2(q))))^2}dq\\
=&2\pi\int_{0}^{\sqrt{\frac{2}{\nu^2}(1-\lambda_{\varepsilon})}}\frac{rH^{+}(L^{\varepsilon}(\frac{1}{\varepsilon^2}(\frac{\nu^2}{2}r^2+\lambda_{\varepsilon})))}{(\frac{\nu^2}{2}r^2+\lambda_{\varepsilon})(H(L^{\varepsilon}(\frac{1}{\varepsilon^2}(\frac{\nu^2}{2}r^2+\lambda_{\varepsilon}))))^2}dr\\
=&\frac{2\pi}{\nu^2}\int_{\lambda_{\varepsilon}}^{1}\frac{H^{+}(L^{\varepsilon}(\frac{\rho}{\varepsilon^2}))}{\rho(H(L^{\varepsilon}(\frac{\rho}{\varepsilon^2})))^2}d\rho.
\end{align*}
Then the proof is concluded.
\end{proof}

\textbf{Acknowledgements:} We would like to thank Prof.Xiangchan Zhu for enlightening discussions and fruitful suggestions. 

\bibliographystyle{alpha}

\end{document}